%% file: diag-canonique.tex
\documentclass[reqno, 11pt]{article}
\usepackage[all]{xy}
\usepackage{latexsym}
\usepackage{amsmath}
\usepackage[latin1]{inputenc}
\usepackage[T1]{fontenc}
\usepackage[francais]{babel}
\usepackage{amssymb,amscd,amsthm}
\usepackage[mathscr]{eucal}
\usepackage{bbm}

\input{def}

\newtheorem{thm}{Th\'eor\`eme}[section]
\newtheorem{lem}[thm]{Lemme}
\newtheorem{cor}[thm]{Corollaire}
\newtheorem{prop}[thm]{Proposition}
\newtheorem{defn}[thm]{D\'efinition}

\newtheorem{rem}[thm]{Remarque}
\newtheorem{exem}[thm]{Exemple}

\begin{document}

\title{Diagrammes canoniques et repr\'esentations modulo $p$ de $\GL_2(F)$
}

\author{Yongquan HU}
\date{}

\maketitle
\textbf{Abstract} --
Let $p$ be a prime number and $F$ a non-Archimedean local
field with residual characteristic $p$. In this article, to an
irreducible smooth representation of $\GL_2(F)$ over $\bFp$ with
central character, we associate a diagram which determines the
original representation up to isomorphism. We also explicitly determine it in some cases.

\tableofcontents

\section{Introduction}\label{section-Hecke}\label{notations-1}

Soient $p$ un nombre premier et $F$ un corps local complet pour une
valuation discr\`ete de corps r\'esiduel fini de caract\'eristique $p$. Fixons un corps alg\'ebriquement clos de caract\'eristique $p$ not\'e $\bFp$.
L'\'etude des repr\'esentations lisses irr\'eductibles avec caract\`ere
central de $\GL_2(F)$ sur $\bFp$ a \'et\'e initi\'ee par Barthel et
Livn\'e (\cite{BL1}, \cite{BL2}). Ils ont class\'e ces repr\'esentations en quatre ``cat\'egories'' (les \emph{caract\`eres}, les
\emph{s\'eries principales}, les \emph{s\'eries sp\'eciales} et les
\emph{supersinguli\`eres}) et ont compl\`etement \'etudi\'e la structure des trois premi\`eres. Dans \cite{Br1}, Breuil a d\'etermin\'e les
supersinguli\`eres dans le cas particulier (mais important) o\`u
$F=\Q_p$, ce qui lui a permis
 de d\'efinir une bijection entre les classes d'isomorphisme de
  repr\'esentations supersinguli\`eres de
$\GL_2(\Q_p)$ et les classes d'isomorphisme de repr\'esentations
continues irr\'eductibles de dimension 2 de $\Gal(\bQp/\Q_p)$ sur
$\bFp$. Mais, lorsque $F\neq \Q_p$, au moins lorsque $F$ est une
extension finie non ramifi\'ee de $\Q_p$ de degr\'e $\geq 2$, il
existe une tr\`es grande quantit\'e de repr\'esentations lisses
admissibles supersinguli\`eres (\cite{Pa}, \cite{BP}, \cite{Hu}).



Une grosse diff\'erence entre la th\'eorie des repr\'esentations modulo $p$ de $\GL_2(F)$ et la th\'eorie classique (i.e. en caract\'eristique $0$) est la non existence de mesure de Harr non triviale sur $\GL_2(F)$ \`a valeurs dans $\bFp$, ce qui fait que beaucoup d'outils classiques ne sont plus applicables. Dans \cite{Pa}, Pa\v{s}k\={u}nas a donn\'e une construction tr\`es g\'en\'erale de repr\'esentations lisses de $\GL_2(F)$ \`a partir d'objets appel\'es ``diagrammes''.
Dans cet article, \`a   une repr\'esentation lisse irr\'eductible de $\GL_2(F)$ avec caract\`ere central, on associe un diagramme canonique qui d\'etermine la classe d'isomorphisme de la repr\'esentation de d\'epart.
\vspace{1mm}

Fixons quelques notations pour \'enoncer notre r\'esultats. 
Notons $\cO$
l'anneau des entiers de $F$, $\varpi$ une uniformisante de $\cO$ (fix\'ee une fois pour toutes),
 \[G=\GL_2(F),\ \ K=\GL_2(\cO),\]
$Z$ le centre de $G$, $I\subset K$ le sous-groupe des matrices triangulaires sup\'erieures
modulo $\varpi$ et $N$ le normalisateur de $I$ dans $G$. Par d\'efinition (\cite{Pa}),  un \emph{diagramme} est
un triplet $(D_0,D_1,r)$ o\`u $D_0$ est une
repr\'esentation lisse de $KZ$, $D_1$ est une repr\'esentation lisse
de $N$ et $r: D_1\ra D_0$ est un morphisme $IZ$-\'equivariant.  Les diagrammes forment une cat\'egorie et l'on dispose d'un foncteur $H_0$ de cette cat\'egorie dans celle des repr\'esentations lisses de $G$. Inversement, si $\pi$ est une $\bFp$-repr\'esentation lisse irr\'eductible de $G$ avec  caract\`ere central, on peut  lui associer un diagramme, appel\'e  le \emph{diagramme canonique} de $\pi$ (\S\ref{section-diag-canonique}):
\[D(\pi)=(D_0(\pi),D_1(\pi),\mathrm{can})\]
o\`u $D_1(\pi)$ est une certaine sous-$N$-repr\'esentation de $\pi$ \`a pr\'eciser,  $D_0(\pi)=\langle KZ\cdot D_1(\pi)\rangle$ est la sous-$KZ$-repr\'esentation de $\pi$
engendr\'ee par $D_1(\pi)$ et $\mathrm{can}$ d\'esigne l'inclusion naturelle $D_1(\pi)\hookrightarrow D_1(\pi)$.

Les r\'esultats centraux de cet article sont les suivants.
\begin{thm}
(corollaire \ref{cor-H0(D)}) On a un isomorphisme naturel de $G$-repr\'esentations 
\[H_0(D(\pi))\cong\pi.\]
\end{thm}


Dans certains cas, on peut d\'eterminer explicitement le diagramme
canonique. 
Notons $I_1\subset I$ le sous-groupe des matrices
unipotentes sup\'erieures modulo $\varpi$. 

\begin{thm}\label{theorem-intro-expli} (proposition \ref{prop-I1-invariant-in-D1} et th\'eor\`emes \ref{theorem-non-super}, \ref{theorem-Breuil-Paskunas} et \ref{theorem-main})
Soit $\pi$ une repr\'esentation lisse irr\'eductible de $G$ avec
caract\`ere central.

(i) On a toujours $\pi^{I_1}\subseteq D_1(\pi)$.

(ii) On a $D_1(\pi)=\pi^{I_1}$ si $\pi$ est non supersinguli\`ere ou si $F=\Q_p$.

(iii) On a $D_1(\pi)=\pi$ si $F$ est de caract\'eristique $p$ et si $\pi$ est supersinguli\`ere.
\end{thm}

Mais malheureusement il ne semble pas facile de calculer $D(\pi)$
dans le cas o\`u $F$ est une extension finie de $\Q_p$ et $\pi$ est supersinguli\`ere, m\^eme de d\'emontrer si $D_1(\pi)$ est ou
non de dimension finie sur $\bFp$. N\'eanmoins, on peut clarifier le lien entre
certaines propri\'et\'es de $D(\pi)$ et certaines propri\'et\'es de $\pi$, qui sont valables quelle que soit la repr\'esentation irr\'eductible $\pi$ avec caract\`ere central, quel que soit le corps $F$.

\begin{thm} (th\'eor\`eme \ref{theorem-Pe}, corollaires \ref{coro-presfinie-equv} et \ref{coro-finitude-->adm}) \label{theorem-intro-proporiete}
Soit $\pi$ une repr\'esentation lisse irr\'eductible de $G$ avec
caract\`ere central. Alors les trois conditions suivantes sont
\'equivalentes:

(i) $D_1(\pi)$ est de dimension finie;

(ii) $\pi$ admet une pr\'esentation finie (voir \S\ref{subsection-D1(sigma,pi)});

(iii) $\pi$ admet une pr\'esentation standard (au sens de \cite{Co}).\vspace{1mm}

\noindent Si l'une de ces conditions est v\'erifi\'ee, alors\vspace{1mm}

(iv) $\pi$ est admissible;

(v) l'espace $I^+(\pi)^{\smatr{1}{\cO}01}$ est de dimension
finie.
\end{thm}

Remarquons qu'il est int\'eressant d'obtenir des pr\'esentations finies pour les repr\'esentations lisses irr\'eductibles de $G$ en raison des travaux de Colmez \cite{Co}, Schneider-Vign\'eras \cite{SV} et Vign\'eras \cite{Vi2}.

Soulignons les cons\'equences suivantes des th\'eor\`emes \ref{theorem-intro-expli}(iii) et \ref{theorem-intro-proporiete} dans le cas o\`u $F$ est de caract\'eristique $p$.
\begin{cor} (corollaire \ref{coro-non-presfinie} et th\'eor\`emes \ref{theorem-irr} et \ref{theorem-morph})
\label{cor-intro}
Si $F$ est de caract\'eristique $p$ et si $\pi$ est une repr\'esentation supersinguli\`ere de $G$, alors

(i) $\pi$ n'est pas de pr\'esentation finie;

(ii) $\pi|_{P^+}$ est irr\'eductible;

(iii) pour toute repr\'esentation lisse $\pi'$ de $G$, on a
\[\Hom_{P^+Z}(\pi,\pi')\cong\Hom_G(\pi,\pi').\]
\end{cor}

Pa\v{s}k\={u}nas (\cite{Pa3}) a montr\'e les analogues de (ii) et de (iii)  du corollaire \ref{cor-intro} sans restriction sur $F$ mais en rempla\c{c}ant $P^+$ par $P$, le sous-groupe de Borel des matrices triangulaires sup\'erieures de $G$. Par ailleurs, apr\`es la premi\`ere r\'edaction de cet article, il m'a signal\'e une preuve de (iii) qui est plus conceptuelle et qui est valable pour \emph{tout} $F$. On la pr\'esentera apr\`es notre preuve originale. 
 \vv


Introduisons maintenant les principales (autres) notations de cet article.
\vspace{1mm} \label{notations-2}

Notons $\p:=\varpi\cO$ l'id\'eal maximal de $\cO$ et $q$ le cardinal
du corps r\'esiduel $\cO/\p$.
On identifie $\cO/\p$ avec $\F_q$ et
on note $[\lambda]$ le repr\'esentant multiplicatif dans $\cO$ de
$\lambda\in\F_q$. 
Si $x\in\cO$, on note $\overline{x}\in \F_q$ sa r\'eduction modulo $\varpi$. On note $\mathrm{val}_F$ la valuation sur $F$ normalis\'ee par $\mathrm{val}_F(\varpi)=1$. \vv

On  d\'esigne par $P^+$  le sous-mono\"ide de $G$ d\'efini par
\[ P^+:=\matr{\cO-\{0\}}{\cO}01.\]
Pour $n\geq 1$, on note
\[K_n:=\matr{1+\p^n}{\p^n}{\p^n}{1+\p^n},\ \  I_n:=\matr{1+\p^n}{\p^{n-1}}{\p^n}{1+\p^{n}}.\]
Remarquons que $K$ normalise $K_n$ et $N$ normalise $I_n$.
On note $U^+=\smatr 1F01$ (resp. $U^-=\smatr10F1$) le sous-groupe
des matrices unipotentes sup\'erieures (resp. inf\'erieures) de telle sorte que $I_1\cap U^+=\smatr1{\cO}01$ (resp. $I_1\cap U^-=\smatr10{\p}1$). On note
$\cH\subset I$ le sous-groupe des matrices de la forme
$\smatr{[\lambda]}00{[\mu]}$ avec $\lambda$, $\mu\in\F_q^{\times}$.
Enfin, on d\'esigne par $\Pi$ (resp. $s$) la matrice    $\smatr01{\varpi}0$ (resp. $\smatr0110$) de $G$. Notons que $N$ est
engendr\'e par $I$ et ${\mathrm{\Pi}}$.

\vv

Toutes les repr\'esentations consid\'er\'ees dans cet article sont sur
des $\bFp$-espaces vectoriels. Pour celles de $G$, on suppose
qu'elles admettent un caract\`ere central, et on notera $\chi_{\pi}$ le caract\`ere central si $\pi$ est une telle repr\'esentation.\vv

Une repr\'esentation $M$ d'un groupe localement profini $H$ est dite \emph{lisse} si chaque vecteur de $M$ est fix\'e par un
sous-groupe ouvert de $H$. Elle est dite \emph{admissible} si $M^{H_1}$, l'espace des $H_1$-invariants de $M$, est de dimension finie sur $\bFp$ pour tout sous-groupe ouvert $H_1$ de $H$, ou de mani\`ere \'equivalente, pour un seul pro-$p$-sous-groupe ouvert $H_1$ (voir, par exemple, \cite[th\'eor\`eme 6.3.2]{Pa}).  On d\'esigne par $\Rep_G$
(resp. $\Rep_K$, $\Rep_I$, etc.) la cat\'egorie des repr\'esentations
lisses de $G$ (resp. $K$, $I$, etc.) sur $\bFp$. 


\vv

Nous utiliserons le r\'esultat classique suivant, qui est essentiel \`a la th\'eorie des repr\'esentations \emph{lisses} modulo $p$ de $G$. Si $M$ est une repr\'esentation lisse d'un
\emph{pro-$p$-groupe} $H$ (e.g. $H=I_1$ ou $I_1\cap U^+$), alors
$M^H$ est non trivial (voir par exemple \cite[lemme 3(1)]{BL2}). Remarquons que si $\sigma$ est une repr\'esentation irr\'eductible de $K$ sur $\bFp$, alors $\sigma^{I_1}$ est de dimension 1.

\vv

Si $M$ est une repr\'esentation d'un groupe ou d'un mono\"{i}de $H$, et si $v\in M$ est un vecteur, on notera $\langle H\cdot v\rangle$  le sous-espace vectoriel de $M$ engendr\'e par $v$ sous l'action de $H$.\vv

Si $M$ est une repr\'esentation lisse d'un groupe profini $H$ (e.g. $H=K$ ou $I$), on d\'efinit le socle de $M$, et on le note $\rsoc_H(M)$, comme la plus grande sous-repr\'esentation semi-simple
de $M$. De m\^eme, on d\'efinit le radical de $M$, et on le note $\rad_H(M)$,
comme la plus petite sous-repr\'esentation de $M$ telle que le
quotient $M/\rad_H(M)$ soit semi-simple. Notons que l'on a toujours $\rad_H(M)\subsetneq M$ par \cite[\S1, proposition 4]{Al}.\vv

Enfin, si $H$ est un sous-groupe ferm\'e de $G$ et si $M$ est une repr\'esentation lisse de $H$ (sur $\bFp$), on note $\cInd_{H}^{G}M$ le $\bFp$-espace vectoriel des fonctions $f:G\ra M$ qui sont localement constantes, \`a support compact modulo $H$, et telles que $f(hg)=h\cdot f(g)$ ($h\in H$, $g\in G$). Il est muni de l'action \`a gauche de $G$ donn\'ee par $(g'\cdot f)(g):=f(gg')$ ($g,g'\in G$). On obtient ainsi une repr\'esentation lisse de $G$. Pour $g\in G$ et $v\in M$, on d\'esigne par $[g,v]$
l'\'el\'ement de $\cInd_{H}^GM$
de support $Hg^{-1}$ et de valeur $v$ en $g^{-1}$. Cela co\"{i}ncide avec la d\'efinition pr\'ec\'edente de $\cInd_{KZ}^G\sigma$. Lorsque $H=P$, on \'ecrira $\Ind_P^G$ au lieu de $\cInd_P^G$. On d\'efinit l'induction $\Ind_I^K$ de mani\`ere analogue. Remarquons que $\cInd_H^G$ est un foncteur exacte de la cat\'egorie des repr\'esentations lisses de $H$ dans la cat\'egorie des repr\'esentations lisses de $G$ (voir par exemple \cite[proposition 4.1.5]{Em} dans le cas $H=P$).
On a la m\^eme chose pour le foncteur $\Ind_I^K$.
\vv \vv


\textbf{Remerciement.} Ce travail s'est accompli sous la direction
de C. Breuil. Je le remercie chaleureusement pour avoir partag\'e
avec moi ses id\'ees et ses connaissances et pour toutes ses
remarques. 
Je remercie sinc\`erement V. Pa\v{s}k\={u}nas pour des
discussions constructives, pour m'avoir envoy\'e l'article
\cite{Pa2}, et particuli\`erement
pour m'avoir permis de pr\'esenter ici sa preuve du th\'eor\`eme \ref{theorem-morph} pour $F$ g\'en\'eral. Je remercie M.-F. Vign\'eras pour les discussions que nous avons eues ensemble et l'int\'er\^et qu'elle a port\'e \`a ce travail.
Je remercie R. Abddellatif et B. Schraen pour leur commentaires et suggestions \`a la premi\`ere version. Enfin, je remercie vivement le referee pour ses nombreuses remarques et suggestions constructives. 

\section{Rappels et compl\'ements}

Ce chapitre   de pr\'eliminaires commence par quelques rappels pour lesquels on se r\'ef\`ere principalement \`a \cite{BL2} et  \`a \cite{Br1}.


\subsection{Rappels sur des d\'ecompositions de $G$}

On rassemble des d\'ecompositions classiques de  $G$ et de  $K$  qui seront utiles au long de l'article.
Commen\c{c}ons par la d\'ecomposition de Cartan:
\begin{equation}
 \begin{array}{rll}G&=&\coprod_{n\geq 0}KZ\matr{\varpi^n}001KZ\\ 
&=&
\Bigl(\coprod_{n\geq 0}IZ\matr{\varpi^n}001KZ\Bigr)\coprod\Bigl(\coprod_{n\geq 0}IZ\Pi\matr{\varpi^n}001KZ\Bigr)\label{equation-Cartan}.
\end{array}
 \end{equation}
D'apr\`es la d\'efinition de $P^+$ et l'\'equation $\smatr{\varpi^na}{b}01=\smatr{\varpi^n}{b}01\smatr{a}001$ si $a\in\cO^{\times}$ et $b\in\cO$, on trouve que \begin{equation}\label{equation-P+=union}P^+KZ=\coprod_{n\geq 0}\matr{\varpi^n}{\cO}01KZ.\end{equation}
Par cons\'equent, en utilisant le lemme \ref{lemma-IP=PI} ci-apr\`es, (\ref{equation-Cartan}) se r\'e\'ecrit sous la forme (voir aussi \cite[lemme 11]{Vi2}):
\begin{equation}\label{equation-vigneras}G=P^+KZ\coprod \Pi P^+KZ.\end{equation}

\begin{lem}\label{lemma-IP=PI}
On a les \'egalit\'es 
 $$IZ\matr{\varpi^n}{\cO}01=\matr{\varpi^n}{\cO}01IZ=IZ\matr{\varpi^{n}}001IZ.$$ 
\end{lem}
\begin{proof}
Voir la preuve de \cite[lemme 4.6]{Pa2} pour la premi\`ere \'egalit\'e. La deuxi\`eme s'en d\'eduit facilement. 
\end{proof}

Par ailleurs, la d\'ecomposition (\ref{equation-P+=union}) implique que tout \'el\'ement $g$ de  $P^+KZ$ s'\'ecrit de mani\`ere unique sous la forme:
\begin{equation} \label{equation-g=g^i-vigneras}
g=g^{(n)}g^{(n-1)}\cdots g^{(1)}k\end{equation}
o\`u $k\in KZ$, $g^{(i)}$ \'egale \`a
l'une des matrices $g_{\lambda}:=\smatr{\varpi}{[\lambda]}01$ avec $\lambda\in\F_q$, $n\geq 0$ et on convient que $n:=0$ si $g\in KZ$. On appelle $\ell(g):=n$ la \emph{longueur} de $g$. Pour $g\in \Pi P^+KZ$, on \'ecrit $g=\Pi g^+$ avec $g^+\in P^+KZ$ et on appelle $\ell(g):=\ell(g^+)+1$ la longueur de $g$. Remarquons que $\ell(g)=0$ si et seulement si $g\in KZ$.\vv

Enfin, on rappelle la d\'ecomposition d'Iwahori:
\begin{equation}\label{equation-Iwahori}I_1=\matr1{\cO}01\matr{1+\p}00{1+\p}\matr10{\p}1\end{equation}
et la d\'ecomposition suivante de $K$:
 \begin{equation}\label{equation-decom-K/I}
 K=I\coprod\Bigl(\coprod_{\lambda\in\F_q}\matr{[\lambda]}110I\Bigr).
 \end{equation}

\subsection{Poids et alg\`ebre de Hecke relative \`a un poids}\label{subsection-Hecke}

\subsubsection{L'op\'erateurs $T$}\label{subsubsection-T}

D'apr\`es \cite[proposition 4]{BL2} et la remarque qui suit, toute repr\'esentation irr\'eductible lisse de $KZ$ sur $\bFp$ est triviale
sur $K_1$ et admet un caract\`ere central.
On appellera  \emph{poids} une telle repr\'esentation en modifiant la d\'efinition originale de \cite{BDJ}, o\`u un poids d\'esigne une repr\'esentation irr\'eductible lisse de $K$. On peut classifier les poids \`a isomorphisme pr\`es (voir par exemple \cite[proposition 1]{BL2}).\vv



Fixons $\sigma$ un poids. Dans $\cInd_{KZ}^G\sigma$, on d\'efinit pour $n\geq 0$:
\[R_n^+(\sigma)=\Bigl[\matr{\varpi^n}{\cO}01,\sigma\Bigr],\ \ \ R_n^-(\sigma)={\mathrm{\Pi}}\cdot R_n^+(\sigma).\]
En particulier, $R_0^+(\sigma)=[\mathrm{Id},\sigma]$ o\`u $\mathrm{Id}$ d\'esigne la matrice identit\'e de $G$. Posons \'egalement $$R_0(\sigma):=R_0^+(\sigma),\ \
R_n(\sigma):=R_n^+(\sigma)\oplus R_{n-1}^-(\sigma) \  {\rm si} \  n\geq 1.$$
Par le lemme \ref{lemma-IP=PI}, $R_n^+(\sigma)$ et $R_n^-({\sigma})$ sont stables par $IZ$ et $R_n(\sigma)$ par $KZ$. De la d\'ecomposition de Cartan (\ref{equation-Cartan}), on d\'eduit que
\[\cInd_{KZ}^G\sigma|_{KZ}=\bigoplus_{n\geq 0}R_n(\sigma)\]
et
\[\cInd_{KZ}^G\sigma|_{IZ}=\Bigl(\bigoplus_{n\geq0}R_n^+(\sigma)\Bigr)\bigoplus\Bigl(\bigoplus_{n\geq 0}R_n^-(\sigma)\Bigr).\]
On a le lemme suivant:


\begin{lem}\label{lemma-decomposition}
  Pour $n\geq 1$, $R_n(\sigma)$ est la sous-$KZ$-repr\'esentation de $\cInd_{KZ}^G\sigma$ engendr\'ee par $R_{n-1}^-(\sigma)$ et
\begin{equation}\label{equation-R_n(sigma)=summ}R_{n}^+(\sigma)= \bigoplus_{\lambda\in\F_q}\matr{[\lambda]}110R_{n-1}^-(\sigma).\end{equation}

\end{lem}
\begin{proof}
 Tout $x\in\cO$ s'\'ecrit de mani\`ere unique sous la forme $x=[{\lambda}_x]+\varpi x'$ avec $\lambda_x\in\F_q$ et $x'\in\cO$ d\'ependant de $x$, de telle sorte que
\[\matr{\varpi^n}x01=\matr{\varpi}{[\lambda_x]}01\matr{\varpi^{n-1}}{x'}01=\matr{[\lambda_x]}110\Pi\matr{\varpi^{n-1}}{x'}01.\]
Le r\'esultat s'en d\'eduit.
\end{proof}

%

L'alg\`ebre de Hecke $\cH(KZ,\sigma)$ (relativement \`a $KZ$ et \`a $\sigma$) est par
d\'efinition l'alg\`ebre des $\bFp$-endomorphismes de $\cInd_{KZ}^G\sigma$ qui commutent \`a l'action de $G$. Par r\'eciprocit\'e de
Frobenius, elle s'identifie \`a l'alg\`ebre de convolution des fonctions
$\varphi:G\ra \End_{\bFp}(\sigma)$ \`a support compact modulo
$Z$ telles que $\varphi(k_1gk_2)=\sigma(k_1)\circ\varphi(g)\circ\sigma(k_2)$
pour $k_1,k_2\in KZ$ et $g\in G$. Si $\varphi$ est une telle fonction et
$T$ l'endomorphisme correspondant de $\cInd_{KZ}^G\sigma$, alors on a la formule (\cite[\S2.4, (3)]{Br1}):
\[T([g,v])=\summ_{g'KZ\in G/KZ}[gg',\varphi(g'^{-1})(v)].\]
Soit $\varphi: G\ra \End_{\bFp}(\sigma)$ l'unique fonction \`a support dans $KZ\smatr{1}{0}{0}{\varpi^{-1}}KZ$ v\'erifiant $\varphi(k_1gk_2)=\sigma(k_1)\circ\varphi(g)\circ\sigma(k_2)$ comme pr\'ec\'edemment  et telle que $\varphi(\smatr{1}{0}{0}{\varpi^{-1}})=U_{r_1}\otimes \cdots \otimes U_{r_{f}}$   (voir \cite[\S3.1]{BL2} ou \cite[\S2.7]{Br1} pour cette notation o\`u les $r_i, 1\leq i\leq f$ sont des entiers associ\'es \`a $\sigma$), et soit $T$ l'op\'erateur
de Hecke correspondant. D'apr\`es
\cite[proposition 8]{BL2}, $\cH(KZ,\sigma)$
est isomorphe \`a l'alg\`ebre des polyn\^omes $\bFp[T]$.

L'action de $T$ se d\'ecrit explicitement  par le lemme suivant.  Soit $v_0\in \sigma$ un vecteur non nul fix\'e par $I_1$.
\begin{lem}\label{lemma-Sv=Tv}
(i) Si $\sigma$ est de dimension 1 (sur $\bFp$), alors
\[T([1,v_0])=[\Pi,v_0]+\summ_{\lambda\in\F_q}\Bigl[\matr{\varpi}{[\lambda]}01,v_0\Bigr].\]

(ii) Si $\sigma$ est de dimension $\geq 2$, alors
\[T([1,v_0])=\summ_{\lambda\in\F_q}\Bigl[\matr{\varpi}{[\lambda]}01,v_0\Bigr].\]
\end{lem}
\begin{proof}
Il d\'ecoule de la formule \cite[\S2.5, (7)]{Br1}, du choix de $T$ que l'on a fait ci-dessus. 
\end{proof}




Si $n\geq 1$, on v\'erifie que $T(R_n^-(\sigma))\subseteq R_{n+1}^-(\sigma)\oplus
R_{n-1}^-(\sigma)$ et que l'op\'erateur $T|_{R_n^-(\sigma)}:R_n^-(\sigma)\ra R_{n+1}^-(\sigma)\oplus
R_{n-1}^-(\sigma)$ est la somme d'une injection $IZ$-\'equivariante
$T^{+}|_{R_n^-(\sigma)}:R_n^-(\sigma)\hookrightarrow R_{n+1}^-(\sigma)$ et d'une
surjection $IZ$-\'equivariante
$T^{-}|_{R_n^-(\sigma)}:R_n^-(\sigma)\twoheadrightarrow R_{n-1}^-(\sigma)$. Explicitement, ils correspondent respectivement aux deux termes de droite de la formule \cite[\S2.5,  (6)]{Br1}. 

\begin{lem}\label{lemma-P(T)}
Soient $k\geq 0$, $f\in \oplus_{n\geq k}R_n^-(\sigma)$ et $P(T)\in
\bFp[T]$ un polyn\^ome de degr\'e $\geq 1$. Alors il existe $f'\in
\oplus_{n\geq k+1} R_n^-(\sigma)$, d\'ependant de $f$ et de $P(T)$, tel que
\[f+f'\in P(T)(\oplus_{n\geq k+1}R_n^-(\sigma)).\]
\end{lem}
\begin{proof}
Comme le degr\'e de $P(T)$ est non nul, on peut \'ecrire $P(T)=(T-\lambda)P_1(T)$
avec $\lambda\in\bFp$ et $P_1(T)\in\bFp[T]$ un polyn\^ome de degr\'e strictement inf\'erieur \`a celui de $P(T)$.
Soit $h\in \oplus_{n\geq k+1}R_n^-(\sigma)$ un
\'el\'ement tel que $T^-(h)=f$. Si $\deg P_1(T)=0$, alors
$f':=T^+(h)-\lambda h$ satisfait \`a la condition demand\'ee. Sinon,
par r\'ecurrence sur le degr\'e de $P(T)$, on peut trouver
$h',h''\in\oplus_{n\geq k+2}R_n^-(\sigma)$ tels que $h+h'=P_1(T)(h'')$, et
donc
\[\begin{array}{rll}
P(T)(h'')&=&(T-\lambda)(h+h')\\
&=&T^-(h)+T^+(h)-\lambda h+(T-\lambda)(h')\\
&=&f+\bigl(T^+(h)-\lambda h+(T-\lambda)(h')\bigr).\end{array}\]
Comme $h'\in \oplus_{n\geq k+2}R_n^-(\sigma)$ et $h\in \oplus_{n\geq k+1}R_n^-(\sigma)$, le vecteur
\[f':=T^+(h)-\lambda h+(T-\lambda)(h')\]
appartient \`a $\oplus_{n\geq
k+1}R_n^-(\sigma)$, ce qui ach\`eve la d\'emonstration.
\end{proof}

\subsubsection{Application aux quotients non triviaux de $\cInd_{KZ}^G\sigma$}

\textbf{Notation}: Si $\pi$ est un $G$-quotient \emph{non trivial} de
$\cInd_{KZ}^G\sigma$, on note $R_n(\sigma,\pi)$ (resp.
$R_n^+(\sigma,\pi)$, $R_n^-(\sigma,\pi)$) l'image de $R_n(\sigma)$ (resp.
$R_n^+(\sigma)$, $R_n^-(\sigma)$) dans $\pi$. Pour un vecteur  $f\in \cInd_{KZ}^G\sigma$, on note $\overline{f}$ l'image de $f$ dans $\pi$.

\begin{prop}\label{lemma-R_n<R_n+2}
Soient $\pi$ un $G$-quotient non trivial de $\cInd_{KZ}^G\sigma$
et $v_0\in\sigma$ un vecteur non nul fix\'e par $I_1$. Alors\vv

(i)  $\overline{[\mathrm{Id},v_0]}\in \sum_{n\geq
0}R_n^-(\sigma,\pi)$;\vspace{1mm}

(ii) $R_{0}(\sigma,\pi)\subset \sum_{n\geq 1}R_{n}(\sigma,\pi)$.
\end{prop}

\begin{proof}
(i) Par la preuve de \cite[lemme 3.2]{Pa3}, on voit que la
surjection $G$-\'equivariante $\cInd_{KZ}^G\sigma\twoheadrightarrow
\pi$ se factorise par
\[\cInd_{KZ}^G\sigma\twoheadrightarrow \cInd_{KZ}^G\sigma/P(T)\twoheadrightarrow\pi\]
avec $P(T)\in\bFp[T]$ un polyn\^ome de degr\'e $\geq 1$ (remarquons que la
preuve utilise \cite[proposition 32]{BL2} qui reste vraie pour
tout $G$-quotient \emph{non trivial} de $\cInd_{KZ}^G\sigma$).
Il suffit donc de v\'erifier que
\begin{equation}\label{equation-R_n<R_n+2}[\mathrm{Id},v_0]\in P(T)(\cInd_{KZ}^G\sigma)+ \bigl(\oplus_{n\geq 0}R_n^-(\sigma)\bigr).
\end{equation}

\'Ecrivons $P(T)=(T-\lambda)P_1(T)$  avec $\lambda\in\bFp$ et $P_1(T)\in\bFp[T]$ un polyn\^ome de degr\'e strictement inf\'erieur \`a celui de $P(T)$. Posons
\[f=\Bigl[\matr{1}00{\varpi},v_0\Bigr]=\Bigl[\matr01{\varpi}0,\matr0110v_0\Bigr]\in R_0^-(\sigma).\]
Alors un calcul facile en utilisant \cite[\S2.5, (8)]{Br1} montre que:
\begin{equation}\label{equation-T-inprop}
(T-\lambda)(f)=[\mathrm{Id},v_0]+h
\end{equation} avec $h\in \oplus_{n\geq
0}R_n^-(\sigma)$ (en fait $h\in R_1^-(\sigma)$). Si le degr\'e de $P_1(T)$ est nul, i.e. $P(T)=a(T-\lambda)$ avec
$a\in\bFp^{\times}$, alors (\ref{equation-R_n<R_n+2}) est d\'ej\`a prouv\'e par
(\ref{equation-T-inprop}). Sinon, d'apr\`es le lemme \ref{lemma-P(T)},
il existe $f'\in\oplus_{n\geq 1}R_n^-(\sigma)$ tel que $f+f'\in
P_1(T)(\oplus_{n\geq 1}R_n^-(\sigma))$, et donc par
(\ref{equation-T-inprop}):
\[[\mathrm{Id},v_0]+h+(T-\lambda)(f')\in P(T)(\cInd_{KZ}^G\sigma).\]
Cela entra\^ine (\ref{equation-R_n<R_n+2}) et termine la preuve de (i).\vv

(ii) Comme $R_0(\sigma,\pi)$ est engendr\'ee par
$\overline{[\mathrm{Id},v_0]}$ en tant que $K$-repr\'esentation et comme
$\sum_{n\geq 1}R_n(\sigma,\pi)$ est une $K$-repr\'esentation
contenant $\sum_{n\geq 0}R_n^-(\sigma,\pi)$, l'assertion (ii) d\'ecoule de (i).
\end{proof}

\begin{rem}\label{remark-Ind-nonadm}
Comme $\cInd_{KZ}^G\sigma$ est r\'eductible (\cite[th\'eor\`eme 25]{BL2})
et non admissible (\cite[proposition 14(2)]{BL2}), la
proposition \ref{lemma-R_n<R_n+2} s'applique en particulier
lorsque $\pi$ est  un $G$-quotient irr\'eductible ou admissible de
$\cInd_{KZ}^G\sigma$.
\end{rem}

\subsubsection{L'op\'erateur $S$}

La d\'efinition suivante, qui sera tr\`es importante pour la suite, est extraite de la formule pour $T$ du  lemme \ref{lemma-Sv=Tv}(ii).

\begin{defn}On d\'efinit l'op\'erateur $S$, op\'erant sur toute repr\'esentation de $G$,  par
\begin{equation}\label{equation-define-Sv}
S:=\summ_{\lambda\in\F_q}\matr{\varpi}{[\lambda]}01\in \bFp[G].
\end{equation}
En posant $S^1=S$, on d\'efinit par r\'ecurrence $S^n:=S\cdot S^{n-1}\in\bFp[G]$ pour $n\geq 2$.   
\end{defn}

Donnons des propri\'et\'es \'el\'ementaires concernant $S$.

\begin{lem}\label{lemma-H-invariant}
Soient $\pi$ une repr\'esentation de $G$ et $v\in\pi$ un vecteur.

(i) Si $v$ est fix\'e par $I_1\cap U^+$, alors il en est de m\^eme de $Sv$. 

(ii) Si $v$ est fix\'e par $I_1\cap U^+$, alors il en est de m\^eme de tout vecteur de $\langle \smatr{1+\p}00{1+\p}\cdot v\rangle$ et on a
$h\cdot Sv=S(h\cdot v)$
pour tout $h\in \smatr{1+\p}00{1+\p}$.


 (iii) Si $v$ est fix\'e par $\smatr{1+\p}{\cO}0{1+\p}$, alors il en est de m\^eme de $Sv$.

(iv) Si $v$ est fix\'e par $I_1$, alors il en est de m\^eme de $Sv$.
\end{lem}
\begin{proof}
(i) Soit  $x\in\cO$ que l'on d\'ecompose en $x=\sum_{n\geq
0}\varpi^n[\mu_n]$ avec $\mu_n\in \F_q\subset \bFp$ pour tout $n\geq 0$. L'\'enonc\'e se d\'eduit du calcul suivant
\[\begin{array}{rll}
\matr1{x}01Sv&=&\summ_{\lambda\in\F_q}\matr{1}{x}01\matr{\varpi}{[\lambda]}01v\\
&=&\summ_{\lambda\in\F_q}\matr{\varpi}{[\mu_0]+[\lambda]+\sum_{n\geq 1}\varpi^n[\mu_n]}01v\\
&=&\summ_{\lambda\in\F_q}\matr{\varpi}{[\mu_0+\lambda]+\varpi X}01v \\
&=&\summ_{\lambda\in\F_q}\matr{\varpi}{[\mu_0+\lambda]}01\matr{1}{X}01v\\
&=&\summ_{\lambda\in\F_q}\matr{\varpi}{[\mu_0+\lambda]}01v\\
&=&Sv\end{array}\] 
o\`u l'on a utilis\'e le fait que $[\mu_0]+[\lambda]-[\mu_0+\lambda]\in\p$ (voir \cite[\S II.6]{Se3}) pour obtenir $X\in\cO$ qui d\'epend de $x$ et de
$\lambda$.

(ii) Le premier \'enonc\'e r\'esulte de la formule suivante (o\`u $a,b,d\in\cO$):
\[\matr1b01\matr{1+\varpi a}00{1+\varpi d}
=\matr{1+\varpi a}00{1+\varpi d}\matr{1}{\frac{b(1+\varpi d)}{1+\varpi a}}01.
\]
D\'emontrons le deuxi\`eme. Soit $h=\smatr{1+\varpi a}00{1+\varpi d}\in \smatr{1+\p}00{1+\p}$. On calcule:
\[\begin{array}{rll}h\cdot Sv&=& \matr{1+\varpi a}00{1+\varpi d}\summ_{\lambda\in\F_q}\matr{\varpi}{[\lambda]}01v\\
&=&\summ_{\lambda\in\F_q}\matr{\varpi}{\frac{[\lambda](1+\varpi a)}{1+\varpi d}}01\matr{1+\varpi a}00{1+\varpi d}v\\
&=&\summ_{\lambda\in\F_q}\matr{\varpi}{[\lambda]}01\matr{1}{X}01 \matr{1+\varpi a}00{1+\varpi d}v\\
&{=}&\summ_{\lambda\in\F_q}\matr{\varpi}{[\lambda]}01\matr{1+\varpi a}00{1+\varpi d}v\\
&=&S(h\cdot v)\end{array}\]
o\`u $X$ est l'unique \'el\'ement de $\cO$, d\'ependant de $a$, $d$ et $\lambda$, tel que 
$\frac{[\lambda](1+\varpi a)}{1+\varpi d}=[\lambda]+\varpi X$.

(iii) C'est une cons\'equence de (i) et  (ii).

(iv) Compte tenu de (iii), on est ramen\'e  par la d\'ecomposition d'Iwahori (\ref{equation-Iwahori}) \`a regarder l'action de $\smatr{1}0{\p}1$ sur $Sv$. Comme pr\'ec\'edemment, on a le calcul suivant (o\`u $c\in\cO$):
\[\begin{array}{rll}\matr10{\varpi  c}1Sv&=&\matr10{\varpi c}1\summ_{\lambda\in\F_q}\matr{\varpi}{[\lambda]}01v\\
&=&\summ_{\lambda\in\F_q}\matr{\varpi}{\frac{[\lambda]}{1+\varpi c[\lambda]}}01\matr{\frac{1}{1+\varpi c[\lambda]}}0{\varpi^{2}c}{1+\varpi c[\lambda]}v\\
&=&\summ_{\lambda\in\F_q}\matr{\varpi}{\frac{[\lambda]}{1+\varpi c[\lambda]}}01v\\
&=&Sv \end{array}\]
d'o\`u le r\'esultat. 
\end{proof}

Soit maintenant $\pi$ une repr\'esentation lisse \emph{irr\'eductible} de $G$ (admettant un caract\`ere central). D'apr\`es \cite{BL2}, de telles repr\'esentations sont class\'ees en quatre cat\'egories: les \emph{caract\`eres}, les \emph{s\'eries principales}, les \emph{s\'eries sp\'eciales} et les \emph{supersinguli\`eres}. On y renvoie le lecteur pour les d\'etails. 

\begin{lem}\label{lemma-S^m=0-pas}
Si $\pi$ est supersinguli\`ere, alors pour tout $v\in \pi^{I_1}$, il existe $m\geq 1$ tel que $S^mv=0$.
\end{lem}
\begin{proof}
Rappelons que $\cH\subset I$ d\'esigne le sous-groupe des matrices de la forme $\smatr{[\lambda]}00{[\mu]}$ avec $\lambda,\mu\in\F_q^{\times}$. C'est un groupe abelien d'ordre premier \`a $p$, de telle sorte que
$v$ puisse s'\'ecrire sous la forme $v=\sum_{i}v_i$ avec $v_i$ des vecteurs propres de $\cH$ de caract\`eres propres distincts l'un de l'autre. On est donc ramen\'e au cas o\`u $v$ est un vecteur propre de $\cH$. D'apr\`es \cite[lemmes 2.6, 2.7]{BP}, on a ou bien $Sv=0$ (d'o\`u le lemme), ou
bien $Sv\neq 0$ et $\langle K\cdot Sv\rangle\subset \pi$
est une repr\'esentation irr\'eductible de $K$ auquel cas le lemme d\'ecoule de \cite[corollaire 3.3]{Pa3}.
\end{proof}


\subsection{Le foncteur $\Ind_I^K$}\label{subsection-Ind-I-K}
\subsubsection{R\'eciprocit\'e de Frobenius }

Rappelons que, si $M$ est une repr\'esentation lisse de $I$, on d\'esigne par $\Ind_I^KM$ la
$K$-repr\'esentation induite, ce qui fournit un foncteur exact de $\Rep_I$ dans $\Rep_K$.
On note $W=\Ind_I^KM$ et on pose
$\mathrm{pr}_{M}:W\twoheadrightarrow M$ le morphisme $I$-\'equivariant naturel
induit par l'identit\'e $\id:W\simto W$ par r\'eciprocit\'e de Frobenius. Plus
pr\'ecis\'ement, si
\[f=a_1[1,v_1]+\summ_{k\in K/I, k\neq 1}a_k[k,v_k]\in\Ind_I^KM\]
avec $a_k\in\bFp$ et $v_k\in M$ pour tout $k\in K/I$, alors
$\mathrm{pr}_M(f):=a_1v_1$. Ce morphisme admet
une section $I$-\'equivariante $i_M: v\in M\mapsto [1,v]\in W$. Ceci r\'ealise la
d\'ecomposition de Mackey pour $W$ en tant que $I$-repr\'esentation:
\begin{equation}\label{equation-W=M+W+}W=\Ind_I^KM=M\oplus W^+\end{equation}
avec $W^+$ le noyau de $\mathrm{pr}_M$. L'espace sous-jacent à
$W^+$ est l'espace vectoriel engendr\'e par
\[\Bigl\{\matr{[\lambda]}110[1,v],\ \  v\in M, \lambda\in\F_q\Bigr\},\]
ou, de mani\`ere \'equivalente, par
\begin{equation}\nonumber
\Bigl\{\summ_{\lambda\in \F_q}\lambda^i\matr{[\lambda]}110[1,v],\ v\in M,\ 0\leq i\leq q-1\Bigr\}.\end{equation}

Soit $Q$ une
$K$-repr\'esentation contenant $M$ et engendr\'ee par $M$. On a par r\'eciprocit\'e de Frobenius une surjection
$K$-\'equivariante $\alpha:W=\Ind_I^KM\twoheadrightarrow Q$ donn\'ee par $[k,v]\mapsto k\cdot v$ pour $k\in K$ et $v\in M$. Soit $W_1$ son noyau et soit $M_1$ l'image de $W_1$ dans $M$ via le morphisme compos\'e
\[W_1\hookrightarrow W \overset{\mathrm{pr}_M}{\twoheadrightarrow} M.\]
On note $Q^+$ l'image de $W^+\subset W$ dans $Q$ de sorte que l'on puisse consid\'erer l'intersection $M\cap Q^+$ comme un sous-espace de $Q$. Les remarques pr\'ec\'edentes sont illustr\'ees par le diagramme commutatif suivant \`a lignes et \`a colonnes exactes (o\`u les morphismes $\overline{\alpha}$ et $\overline{\mathrm{pr}}_M$ sont d\'efinis de mani\`ere \'evidente):
\[\xymatrix{&&0\ar[d]&0\ar[d]&\\
&&\ar[r]\ar[d]W^+\ar[r]\ar[d]&Q^+\ar[r]\ar[d]&0\\
0\ar[r]&W_1\ar[r]\ar[d]&W\ar^{\alpha}[r]\ar[d]_{\mathrm{pr}_M}&Q\ar[r]\ar^{\overline{\mathrm{pr}}_M}[d]&0\\
0\ar[r]&M_1\ar[r]\ar[d]&M\ar^{\overline{\alpha}}[r]\ar[d]\ar@{.>}@/_/[u]_{i_M}&M/M_1\ar[r]\ar[d]&0\\
&0&0&0}\]
dont l'exactitude de la colonne de droite r\'esulte du
lemme du serpent.
Le lemme suivant sera utilis\'e de mani\`ere cruciale au
\S\ref{section-diag-canonique}, en particulier au lemme
\ref{lemme-deschoices}.

\begin{lem}\label{lemma-intersection}
Avec les notations pr\'ec\'edentes, on a \[M_1=M\cap Q^+,\ \  W_1\subseteq
\Ind_I^KM_1;\]
de plus, $M_1$ est la plus petite sous-repr\'esentation de $M$ ayant cette derni\`ere propri\'et\'e.
\end{lem}
\begin{proof}
Par la d\'efinition de $i_M$, on a $\alpha\circ i_M=(\mathrm{id}:M\hookrightarrow Q)$.
Par cons\'equent, si $v\in M$, on a:
\[\begin{array}{rll}
v\in M_1 &\Longleftrightarrow & \overline{\alpha}(v)=0\\
&\Longleftrightarrow &\overline{\mathrm{pr}}_M\circ\alpha\circ i_M(v)=0\\
&\Longleftrightarrow& \overline{\mathrm{pr}}_M(v)=0\\
&\Longleftrightarrow& v\in Q^+,
\end{array}\]
d'o\`u le premier \'enonc\'e.

Pour les autres, en rempla\c{c}ant $M$ par $M/M_1$ et $W_1$ par
$W_1/(\Ind_I^KM_1\cap W_1)$, on est ramen\'e \`a montrer l'\'enonc\'e suivant:
\emph{si $W_1\subset\Ind_I^KM$ est une sous-$K$-repr\'esentation,
alors $W_1=0$ si et seulement si $M_1=0$.} Or, ceci est une
cons\'equence directe de la r\'eciprocit\'e de Frobenius.
\end{proof}

Donnons une application du lemme \ref{lemma-intersection}:
\begin{cor}\label{cor-intersection}
Sous les hypoth\`eses ci-dessus, on suppose que $Q$ est une repr\'esentation irr\'eductible de $K$ engendr\'ee par son espace $I_1$-invariant $Q^{I_1}$. On choisit $M=Q^{I_1}$. Alors $Q$ et $Q^+$ co\"{i}ncident. C'est-\`a-dire, $Q$ est engendr\'e sur $\bFp$ par les vecteurs
\[\biggl\{\matr{[\lambda]}110v,\ \  \lambda\in\F_q\biggr\}.\]
\end{cor}
\begin{proof}
Comme on l'a indiqu\'e dans l'introduction, $M$ est non trivial de dimension 1 sur $\bFp$. Consid\'erons la surjection naturelle $K$-\'equivariante $W=\Ind_I^KM\twoheadrightarrow Q$. Comme $\Ind_I^KM$ n'est pas irr\'eductible par \cite[lemmes 2.3, 2.4]{BP}, on voit que le noyau $W_1$ est non trivial et qu'il en est de m\^eme de $M_1\subseteq M$ puisque $W_1\subseteq \Ind_I^KM_1$ d'apr\`es le lemme \ref{lemma-intersection}. Or, $M$ est de dimension 1, on a donc $M_1=M$, puis $M$ s'injecte dans $Q^+$ en appliquant \`a nouveau le lemme \ref{lemma-intersection}. L'\'enonc\'e s'en d\'eduit puisque l'espace vectoriel $Q$ est engendr\'e par $M$ et $Q^+$.
\end{proof}

\subsubsection{Invariants sous le groupe $I_1\cap U^+$}

Si $M$ est une repr\'esentation de $I$, on note ${\mathrm{\Pi}}(M)$ la repr\'esentation de $I$ d\'efinie par:
\begin{enumerate}
\item[--] l'espace sous-jacent de $\Pi(M)$ est l'ensemble des symboles $\{\Pi(v),\ v\in M\}$ qui forment un
$\bFp$-espace vectoriel de mani\`ere \'evidente;

\item[--] l'action de $I$ sur $\Pi(M)$ est donn\'ee par:
\begin{equation}\label{equation-define-Pi(M)}
h\cdot {\mathrm{\Pi}}(v):=\Pi\bigl({\mathrm{\Pi}}^{-1}h{\mathrm{\Pi}}\cdot v\bigr),\ \ h\in I, v\in M.
\end{equation}
\end{enumerate}
\begin{rem}
Rappelons que $N$ est le sous-groupe de $G$ engendr\'e par $I$ et $\Pi$. Si $M'$ est une repr\'esentation de $N$ et si $M\subset M'$ est
un sous-espace vectoriel stable par $I$, alors
l'espace ${\mathrm{\Pi}}(M)$ est aussi stable par $I$ avec l'action d\'efinie
par (\ref{equation-define-Pi(M)}).
\end{rem}

\begin{prop}\label{prop-U+-invariant}
Soit $M$ une repr\'esentation lisse de $I$. Posons
$W=\Ind_I^K{\mathrm{\Pi}}(M)$ et $W^+\subset W$ le sous-espace vectoriel engendr\'e par
\[\Bigl\{F_{i,v}:=\summ_{\lambda\in\F_q}\lambda^i\matr{[\lambda]}110[1,{\mathrm{\Pi}}(v)],\ \ v\in M,\ 0\leq i\leq q-1\Bigr\}.\]

(i) Pour tout vecteur $v\in M$ fix\'e par ${I_1\cap U^+}$, l'\'el\'ement
$F_{0,v}\in W^+$ est fix\'e par ${I_1\cap U^+}$.

(ii) L'espace $W^+$ est stable par $I_1\cap U^+$ et $(W^+)^{I_1\cap U^+}$ est engendr\'e sur $\bFp$ par $\{F_{0,v},\ v\in M^{I_1\cap U^+}\}$.
\end{prop}
\begin{proof}
(i) La m\^eme preuve que celle du lemme \ref{lemma-H-invariant}(i).

(ii) Par la d\'ecomposition (\ref{equation-W=M+W+}), $W^+$ est $I$-stable et donc $I_1\cap U^+$-stable. Soit $f\in W^+$ un vecteur non nul fix\'e par $I_1\cap U^+$.  \'Ecrivons $f$ sous la forme
$f=\sum_{i=0}^{q-1}F_{i,v_i}$
avec $v_i\in M$. Montrons que $v_0\in M^{I_1\cap U^+}$ et que $v_i=0$ pour tout $i>0$. 

Soit $i_0>0$ un indice fix\'e tel que $v_{i_0}\neq 0$. Comme $I_1\cap U^+$ est un pro-$p$-groupe, la repr\'esentation $\langle (I_1\cap U^+)\cdot v_{i_0}\rangle$ poss\`ede des vecteurs $I_1\cap U^+$-invariants non nuls, donc il existe un \'el\'ement $Q\in\bFp[I_1\cap U^+]$ tel que $Qv_{i_0}$ soit non nul fix\'e par $I_1\cap U^+$. De plus, en appliquant le m\^eme raisonnement \`a tous les $i$, $1< i\leq q-1$, et quitte \`a modifier $Q$, on peut supposer que $Qv_i$ est fix\'e par $I_1\cap U^+$ pour tout $1<i\leq q-1$. Par la formule suivante (avec $a\in\cO$ et $\lambda\in\F_q$):
\begin{equation}\label{equation-U^+-inv}\matr1{\varpi a}01\matr{\varpi}{[\lambda]}01=\matr{\varpi}{[\lambda]}01\matr1{a}01\end{equation}
on voit que $F_{i,Qv_i}$ est fix\'e par $\smatr1{\p}01$ pour tout $i>0$.
Puis, encore d'apr\`es (\ref{equation-U^+-inv}), on peut trouver $Q'\in\bFp[I_1\cap U^+]$ v\'erifiant
\[Q'\matr{\varpi}{[\lambda]}01=\matr{\varpi}{[\lambda]}01 Q,\ \ \forall \lambda\in\F_q,\]
d'o\`u \begin{equation}\label{equation-U^+-inv-2}Q'f=\summ_{i=0}^{q-1}Q'F_{i,v_i}=\summ_{i=0}^{q-1}F_{i,Qv_i}.\end{equation} On en d\'eduit que $F_{0,Qv_0}$ est fix\'e par $\smatr1{\p}01$ parce que $Q'f$ et $F_{i,Qv_i}$ pour $i>0$ le sont, 
puis que $Qv_0$ est fix\'e par $I_1\cap U^+$ en utilisant encore l'\'equation (\ref{equation-U^+-inv}).


Maintenant, les vecteurs $Qv_i$ \'etant tous fix\'es par $I_1\cap U^+$, on a le calcul suivant pour \emph{tout} $0\leq i\leq q-1$, avec la convention que $F_{-1,Qv_0}:=0$ (comparer avec la d\'emonstration de \cite[lemme 2]{BL2}, particuli\`erement avec l'op\'erateur $l\in \mathrm{Lie} \mathbf{GL_2}$):
\[\begin{array}{rll} \displaystyle \summ_{\mu\in\F_q}\mu^{q-2}\matr{1}{[\mu]}01F_{i,Qv_i}&=&\summ_{\mu,\lambda\in\F_q}\mu^{q-2}\lambda^i
\matr{\varpi}{[\mu]+[\lambda]}01Qv_i\\
&=&\summ_{\mu,\lambda\in\F_q}\mu^{q-2}\lambda^i
\matr{\varpi}{[\mu+\lambda]}01Qv_i\\
&=&\summ_{\mu,\lambda\in\F_q}\mu^{q-2}(\lambda-\mu)^i\matr{\varpi}{[\lambda]}01Qv_i\\
&=&\summ_{k=0}^i\dbinom{i}{k}\Bigl(\summ_{\mu\in\F_q}\mu^{q-2}(-\mu)^{i-k}\Bigr)\Bigl(\summ_{\lambda\in\F_q}
\lambda^k\matr{\varpi}{[\lambda]}01Qv_i\Bigr)\\
&=&iF_{i-1,Qv_i}
\end{array}\]
o\`u la derni\`ere \'egalit\'e vient du fait que, si $0<j<2(q-1)$, alors la somme $\sum_{\mu\in\F_q}\mu^j$ vaut $-1$ si $j=q-1$  et $0$ sinon. D'autre part, comme $Q'f$ est fix\'e par $I_1\cap U^+$, on voit que:
\[\summ_{\mu\in\F_q}\mu^{q-2}\matr1{[\mu]}01 Q'f=\Bigl(\summ_{\mu\in\F_q}\mu^{q-2}\Bigr)Q'f=0.\]
En utilisant (\ref{equation-U^+-inv-2}), on en d\'eduit facilement que $F_{i-1,Qv_i}=0$ puis $Qv_i=0$ pour tout $i>0$, ce qui est impossible puisque $Qv_{i_0}\neq 0$. Cette contradiction montre que $v_i=0$ pour tout $i>0$ et donc $f=F_{0,v_0}$. Enfin, puisque $f$ est fix\'e par $I_1\cap U^+$, l'\'equation (\ref{equation-U^+-inv}) permet de conclure que $v_0$ l'est aussi. Cela termine la d\'emonstration.
\end{proof}
%

\subsubsection{Cons\'equences}\label{subsubsection-injective}
Rappelons que si $\sigma$ est un poids, on a d\'efini les
sous-espaces $IZ$-stables $R_n^+(\sigma)$ pour tout $n\geq 0$ de $\cInd_{KZ}^G\sigma$ au
\S\ref{subsection-Hecke}. Notons que l'\'el\'ement $S$ d\'efini par (\ref{equation-define-Sv}) induit une application lin\'eaire de $R_n^+(\sigma)$ dans $R_{n+1}^+(\sigma)$. Le corollaire suivant am\'eliore \cite[proposition 14(2)]{BL2}. Soit $v_0\in\sigma$ un vecteur non nul fix\'e par $I_1$.

\begin{prop}\label{prop-1-dim}
Pour tout $n\geq0$, on a $\dim_{\bFp} {R_n^+(\sigma)}^{I_1\cap U^+}=1$ et $R_n^+(\sigma)^{I_1\cap U^+}=\bFp S^n[\id,v_0]$.
\end{prop}
\begin{proof}
On prouve la proposition par r\'ecurrence sur $n$. Il est vrai au rang $n=0$ d'apr\`es \cite[lemme 2]{BL2}. Il reste donc \`a
d\'emontrer que\[\dim_{\bFp}
{R_n^+(\sigma)}^{I_1\cap U^+}=1\Longrightarrow\dim_{\bFp}
{R_{n+1}^+(\sigma)}^{I_1\cap U^+}=1\]
et que
\[R_n^+(\sigma)^{I_1\cap U^+}=\bFp v_n\Longrightarrow R_{n+1}(\sigma)^{I_1\cap U^+}=\bFp Sv_n.\]
Or, cela est une cons\'equence du lemme
\ref{prop-U+-invariant} puisque $R_n^+(\sigma)$ est isomorphe \`a l'espace $\bigl(\Ind_I^K\Pi(R_{n-1}^+(\sigma))\bigr)^{+}$ d'apr\`es (\ref{equation-R_n(sigma)=summ}).
\end{proof}

 On pose pour tout $n\geq 0$ (encore dans $\cInd_{KZ}^G\sigma$) \[M_n^+(\sigma)=\Bigl[\matr{\varpi^n}{\cO}01,\bFp v_0\Bigr]\] ce qui fait de $M_n^+(\sigma)$ une sous-$IZ$-repr\'esentation de $R_n^+(\sigma)$ par le lemme \ref{lemma-IP=PI}.

\begin{cor}\label{corollaire-M_n+}
(i) $M_n^+(\sigma)$ est triviale sur $I_{n+1}\cap U^+=\smatr{1}{\p^{n}}01$.

(ii) En tant que repr\'esentation de $I_1\cap U^+/I_{n+1}\cap U^+$, $M_n^+(\sigma)$ est une enveloppe injective de la repr\'esentation triviale dans la cat\'egorie $\Rep_{I_1\cap U^+/I_{n+1}\cap U^+}$.
\end{cor}

On renvoie le lecteur \`a \cite[\S6]{Pa} ou \cite[\S5]{Br2} pour la notion ``enveloppe injective''. 

\begin{proof}
(i) Il d\'ecoule de la formule suivante o\`u $a,b\in\cO$: 
\[\matr1{\varpi^nb}01\matr{\varpi^n}{a}01=\matr{\varpi^{n}}{a}01\matr{1}b01.\]

(ii)  La proposition \ref{prop-1-dim} implique que l'espace des $I_1\cap U^+$-invariants de $M_n^+(\sigma)$ est de dimension 1. D'autre part, $M_n^+(\sigma)$ est de dimension $q^n$ dont une base est form\'ee des vecteurs $\{[\smatr{\varpi^n}{b}01,v]\}$ avec $b\in\cO$ parmi un syst\`eme de repr\'esentants de $\cO/\p^n$.   L'\'enonc\'e s'en d\'eduit d'apr\`es \cite[\S5, corollaire 4]{Al}, car le groupe $I_1\cap U^+/I_{n+1}\cap U^+$,  \'etant isomorphe \`a $\cO/\p^n$,  est un $p$-groupe abelien d'ordre $q^n$.  
\end{proof}\vv

Dans la suite, soit $\pi$ une repr\'esentation lisse irr\'eductible de $G$ isomorphe \`a une s\'erie sp\'eciale ou \`a une s\'erie principale.  
D'apr\`es \cite[th\'eor\`eme 33]{BL2}, $\pi$ admet toujours une sous-$KZ$-repr\'esentation irr\'eductible $\sigma$ de dimension $\geq 2$ sur $\bFp$. Soit $\cInd_{KZ}^G\sigma\twoheadrightarrow \pi$ la surjection $G$-\'equivariante induite. Alors \cite[th\'eor\`eme 34]{BL2} impliqu'elle se factorise par $\cInd_{KZ}^G\sigma/(T-\lambda)$ avec $\lambda\in\bFp^{\times}$.  

Soient $v_0\in\sigma$ un vecteur non nul fix\'e par $I_1$ et $\langle P^+\cdot v_0\rangle$ le sous-espace vectoriel de $\pi$ engendr\'e par $v_0$ qui est $IZ$-stable  par le lemme \ref{lemma-IP=PI}. On en d\'eduit  les morphismes naturels $IZ$-\'equivariants:
  \[\iota_n:M_n^+(\sigma)\ra \langle P^+\cdot v_0\rangle.\]
Le lemme \ref{lemma-P+v=inj} suivant s'est inspir\'e de la preuve de \cite[proposition 11.1]{Pa2}.  Il sera utilis\'e au \S\ref{subsubsection-coro-cas(ii)}.

\begin{lem}\label{lemma-P+v=inj}
(i) Pour tout $n\geq 0$, $\iota_n$ est injectif.

(ii) Pour tout $n\geq 0$, on a $\im(\iota_n)\subseteq\im(\iota_{n+1})$.

(iii) En tant que repr\'esentation de $I_1\cap U^+$, $\langle P^+\cdot v_0\rangle$ est une enveloppe injective de la repr\'esentation triviale  dans la cat\'egorie $\Rep_{I_1\cap U^+}$.
\end{lem}
\begin{proof}
(i) Par le corollaire \ref{corollaire-M_n+}, l'espace des $I_1\cap U^+$-invariants de $M_n^+(\sigma)$ est de dimension 1 engendr\'e par le vecteur $S^n[\id,v_0]$, ce qui fait que, si $\iota_n$ n'est pas injectif, alors l'image de $S^n[\id,v_0]$ est nulle dans $\pi$. Or, c'est absurde parce que
\[S^n[\id,v_0]=T^n[\id,v_0]=\lambda^n[\id,v_0] \neq0 \mod (T-\lambda)\]
dont la premi\`ere \'egalit\'e vient du lemme \ref{lemma-Sv=Tv}(ii).

(ii) Compte tenu de (i), il suffit de v\'erifier que l'op\'erateur $T$ induit une injection $M_n^+(\sigma)\hookrightarrow M_{n+1}^+(\sigma)$ pour tout $n\geq 0$. Comme $\sigma$ est suppos\'e de dimension $\geq 2$, le cas o\`u $n=0$ r\'esulte du lemme \ref{lemma-Sv=Tv}(ii).  Supposons  $n\geq 1$. D'apr\`es la formule \cite[\S2.5, (5)]{Br1} de $T$, on est ramen\'e \`a v\'erifier  que  
\[\matr0110\varphi\Bigl(\matr100{\varpi^{-1}}\Bigr)\matr011{-[\lambda]}v_0\in\bFp v_0,\ \ \forall \lambda\in\F_q\]
et que 
\[\varphi\Bigl(\matr{1}00{\varpi^{-1}}\Bigr)v_0=0.\]
En effet, le premier \'enonc\'e d\'ecoule de \cite[lemme 3.1.1]{Br1} et le deuxi\`eme est une cons\'equence des faits que $v_0\in\sigma^{I_1}$ et que $\sigma$ est de dimension $\geq 2$.

(iii) Il d\'ecoule de (ii) et du corollaire \ref{corollaire-M_n+}(ii) en utilisant \cite[proposition 5.17]{Br2}.
\end{proof}


\section{Le diagramme canonique}\label{section-diag-canonique}

Dans ce chapitre, apr\`es des pr\'eliminaires au
\S\ref{subsection-preliminaire}, on d\'efinit au
\S\ref{subsection-definition} le diagramme canonique
associ\'e \`a une repr\'esentation lisse irr\'eductible de $G$ et on
d\'emontre que ce diagramme d\'etermine la classe d'isomorphisme de la repr\'esentation de d\'epart. Au \S\ref{subsection-quotient}, on discute le diagramme canonique associ\'e \`a un quotient non trivial de $\cInd_{KZ}^G\sigma$ avec $\sigma$ un poids fix\'e.


\subsection{Pr\'eliminaires}\label{subsection-preliminaire}
\subsubsection{Filtration sur un quotient non trivial de $\cInd_{KZ}^G\sigma$}

Fixons $\sigma$ un poids. Suivant \cite{Co},  on note $I^+(\sigma)$
le sous-espace vectoriel de $\cInd_{KZ}^G\sigma$ engendr\'e par les
$[g,\sigma]$ pour $g\in P^{+}=\smatr{\cO-\{0\}}{\cO}{0}{1}$.
De (\ref{equation-P+=union}), on voit que
\[I^{+}(\sigma)=\bigoplus_{n\in\N}\Bigl[\matr{\varpi^n}{\cO}{0}{1},\sigma\Bigr]=\bigoplus_{n\geq 0}R_n^+(\sigma).\]
On pose $I^-(\sigma):={\mathrm{\Pi}}\cdot I^+(\sigma)$ de sorte que
$I^-(\sigma)=\oplus_{n\geq 0}R_n^-(\sigma)$ et que
$\cInd_{KZ}^G\sigma=I^+(\sigma)\oplus I^-(\sigma)$ par la d\'ecomposition (\ref{equation-vigneras}). Les espaces
vectoriels $I^+(\sigma)$ et $I^-(\sigma)$ sont stables sous
l'action de $IZ$, et $\oplus_{n\geq
1}R_n(\sigma)$ est la sous-$KZ$-repr\'esentation de $\cInd_{KZ}^G\sigma$ engendr\'ee par
$I^-(\sigma)$ d'apr\`es le lemme \ref{lemma-decomposition}.

\vv

Dans la suite, on fixe $\pi$ un $G$-quotient non trivial (pas
forc\'ement irr\'eductible ni admissible) de $\cInd_{KZ}^G\sigma$ et on note $R(\sigma,\pi)$ le noyau correspondant. On
note $I^+(\sigma,\pi)$ (resp. $I^-(\sigma,\pi)$) l'image de
$I^+(\sigma)$ (resp. $I^-(\sigma)$) dans $\pi$. Ils sont stables par $IZ$. Pour un vecteur $f\in \cInd_{KZ}^G\sigma$, on note $\overline{f}$ l'image de $f$ dans $\pi$.

Il r\'esulte de la proposition
\ref{lemma-R_n<R_n+2}(ii) et de ce qui pr\'ec\`ede que $I^-(\sigma,\pi)$ engendre
enti\`erement l'espace $\pi$ sous l'action de $K$, de sorte que l'on obtient par
r\'eciprocit\'e de Frobenius une surjection
$K$-\'equivariante
\[
\Ind_I^KI^{-}(\sigma,\pi)\twoheadrightarrow \pi.\] On note
$W_1(\sigma,\pi)$ son noyau.
D'apr\`es le lemme \ref{lemma-intersection}, appliqu\'e \`a
$M=I^-(\sigma,\pi)$ et \`a $Q=\pi$, on
voit que $W_1(\sigma,\pi)$ s'identifie \`a une sous-$K$-repr\'esentation de
$\Ind_I^K(\sum_{n \geq 1}R_n^+(\sigma,\pi)\cap I^-(\sigma,\pi))$,
et donc de $\Ind_I^K(I^+(\sigma,\pi)\cap
I^-(\sigma,\pi))$. Autrement dit:
\begin{lem}\label{lemme-deschoices}
Pour toute \'egalit\'e dans $\pi$ de la forme
\[\summ_{\lambda\in\F_q}\matr{\varpi}{[\lambda]}01w_{\lambda}+{\mathrm{\Pi}}(w)=0\]
avec $w,w_{\lambda}\in I^+(\sigma,\pi)$, on a $w,w_{\lambda}\in
I^+(\sigma,\pi)\cap I^-(\sigma,\pi)$.
\end{lem}
\begin{proof}
Par ce qui pr\'ec\`ede, c'est une cons\'equence du lemme \ref{lemma-intersection} et de la d\'ecomposition (\ref{equation-decom-K/I}).
\end{proof}
\begin{rem}\label{remarque-deschoices}
On a une variante utile du lemme \ref{lemme-deschoices}: pour
toute \'egalit\'e dans $\pi$ de la forme
\[\summ_{i=0}^{q-1}\summ_{\lambda\in\F_{q}}\lambda^i\matr{\varpi}{[\lambda]}01w_i+{\mathrm{\Pi}}(w)=0\]
avec $w,w_i\in I^+(\sigma,\pi)$, on a $w,w_i\in
I^+(\sigma,\pi)\cap I^-(\sigma,\pi)$. 
\end{rem}
\begin{rem}
La raison principale pour laquelle on consid\`ere
$I^+(\sigma,\pi)\cap I^-(\sigma,\pi)$ plut\^ot que $\sum_{n\geq
1}R_n^+(\sigma,\pi)\cap I^-(\sigma,\pi)$ est que le premier est
stable sous l'action de ${\mathrm{\Pi}}$ (voir la d\'efinition
\ref{definition-diagcano} plus tard). Voir aussi le lemme
\ref{lemma-cK=intesection}.
\end{rem}


 Tout vecteur $v\in\pi$ peut s'\'ecrire sous la
forme $v=v^++v^-$ avec $v^+\in I^+(\sigma,\pi)$ et $v^-\in
I^-(\sigma,\pi)$. Une telle d\'ecomposition est unique \`a un vecteur dans $I^+(\sigma,\pi)\cap I^-(\sigma,\pi)$ pr\`es: si $v=v_1^++v_1^-$ est une autre d\'ecomposition, on
a
\[v^+-v_1^+=-v^- +v_1^-\in I^+(\sigma,\pi)\cap I^-(\sigma,\pi).\]

\vv

Posons $I^{+,0}(\sigma,\pi)=I^+(\sigma,\pi)\cap I^-(\sigma,\pi)$
et par r\'ecurrence,
\begin{equation}\label{equation-def-I+n}
I^{+,n}(\sigma,\pi)=I^+(\sigma,\pi)\cap \langle K\cdot
{\mathrm{\Pi}}(I^{+,n-1}(\sigma,\pi))\rangle.
\end{equation}
\'Evidemment, les $I^{+,n}(\sigma,\pi)$  sont des sous-espaces de $I^+(\sigma,\pi)$ stables par $IZ$.
\begin{lem}\label{lemma-I+,n}
Pour tout $n\geq 1$, on a

(i) $I^{+,n-1}(\sigma,\pi)\subseteq I^{+,n}(\sigma,\pi)$;

(ii) $\langle
K\cdot{\mathrm{\Pi}}(I^{+,n-1}(\sigma,\pi))\rangle={\mathrm{\Pi}}(I^{+,n-1}(\sigma,\pi))+
I^{+,n}(\sigma,\pi)$.
\end{lem}
\begin{proof}
(i) Par d\'efinition,
${\mathrm{\Pi}}(I^{+,0}(\sigma,\pi))=I^{+,0}(\sigma,\pi)$, donc
\[I^{+,1}(\sigma,\pi)=I^{+}(\sigma,\pi)\cap \langle K\cdot I^{+,0}(\sigma,\pi)\rangle\supseteq I^{+,0}(\sigma,\pi),\]
ce qui prouve l'\'enonc\'e dans le cas  o\`u $n=1$. Si $n\geq 2$, alors
par r\'ecurrence:
\[\begin{array}{rll}I^{+,n}(\sigma,\pi)&=&I^+(\sigma,\pi)\cap \langle
K\cdot {\mathrm{\Pi}}(I^{+,n-1}(\sigma,\pi))\rangle\\
&\supseteq& I^{+}(\sigma,\pi)\cap \langle
K\cdot{\mathrm{\Pi}}(I^{+,n-2}(\sigma,\pi))\rangle \\
&=& I^{+,n-1}(\sigma,\pi).
\end{array}\]

(ii) L'inclusion $\supseteq$ d\'ecoule de la d\'efinition
(\ref{equation-def-I+n}) et l'inclusion $\subseteq$  de la
d\'ecomposition (\ref{equation-decom-K/I}).
\end{proof}

On pose la
d\'efinition suivante:

\begin{defn}\label{definition-niveau}
Pour $v\in \pi$, on d\'efinit le \emph{niveau} de $v$ (relatif \`a
$\sigma$), not\'e $\ell_{\sigma}(v)$, comme suit:
\begin{itemize}
\item[(i)] supposons $v\in I^+(\sigma,\pi)$, on pose
\begin{itemize} \item[--] $\ell_{\sigma}(v):=0$  si $v\in
I^{+,0}(\sigma,\pi)$


\item[--] $\ell_{\sigma}(v):=n$ si $v\in I^{+,n}(\sigma,\pi)\backslash
I^{+,n-1}(\sigma,\pi)$;
\end{itemize}
\vspace{1mm}

\item[(ii)] supposons $v\in I^-(\sigma,\pi)$, on pose
$\ell_{\sigma}(v):=\ell_{\sigma}({\mathrm{\Pi}}(v))$;

\item[(iii)] enfin, pour $v=v^++v^-\in\pi$ avec $v^+\in
I^+(\sigma,\pi)$ et $v^-\in I^-(\sigma,\pi)$, on pose
$\ell_{\sigma}(v):=\max\{\ell_{\sigma}(v^+),\ell_{\sigma}(v^-)\}$.
\end{itemize}
\end{defn}

\begin{rem}\label{remark-sur-def-niveau}
Comme le groupe $G$ est engendr\'e par $K$ et ${\mathrm{\Pi}}$ et comme $I^+(\sigma,\pi)\cap
I^-(\sigma,\pi)$ contient un g\'en\'erateur de $\pi$ comme repr\'esentation de $G$ par la proposition \ref{lemma-R_n<R_n+2}(i), on a
$\ell_{\sigma}(v)<+\infty$ pour tout $v\in\pi$. D'autre part,
$\ell_{\sigma}(v)$ dans (iii) est bien d\'efini puisque
$\ell_{\sigma}(v^+)$ et $\ell_{\sigma}(v^-)$ ne d\'ependent que de $v$.
\end{rem}

Donnons des propri\'et\'es concernant $\ell_{\sigma}(\cdot)$.

\begin{lem}\label{lemma-niveau}
(i) Si $v_1,v_2\in \pi$, alors
$\ell_{\sigma}(v_1+v_2)\leq\max\{\ell_{\sigma}(v_1),\ell_{\sigma}(v_2)\}$;
si $\ell_{\sigma}(v_1)\neq \ell_{\sigma}(v_2)$, alors
$\ell_{\sigma}(v_1+v_2)=\max\{\ell_{\sigma}(v_1),\ell_{\sigma}(v_2)\}$.


(ii) Soit $v\in I^+(\sigma,\pi)$ un vecteur tel que $\ell_{\sigma}(v)=n\geq 1$. Alors $\ell_{\sigma}(Sv)=n+1$,  o\`u $S\in\bFp[G]$ est d\'efini par (\ref{equation-define-Sv}).

(iii) Soit $v\in I^-(\sigma,\pi)$ un vecteur tel que $\ell_{\sigma}(v)=n\geq 1$. Alors $\ell_{\sigma}(Sv)\leq n$ et
\[Sv\in I^{+,n}(\sigma,\pi)+\Pi(I^{+,n-1}(\sigma,\pi)).\]
\end{lem}
\begin{proof}
(i) C'est une cons\'equence triviale de la d\'efinition
\ref{definition-niveau}.

(ii) 
Par d\'efinition
$\ell_{\sigma}(Sv)\leq n+1$. Si $\ell_{\sigma}(Sv)\leq
n$, alors il existe des $w_i\in I^{+,n-1}(\sigma,\pi)$ pour $0\leq
i\leq q-1$ tels que
\[Sv=\summ_{\lambda\in\F_q}\matr{\varpi}{[\lambda]}01v=\summ_{i=0}^{q-1}\summ_{\lambda\in\F_{q}}\lambda^i\matr{\varpi}{[\lambda]}01w_i.\]
D'apr\`es la remarque \ref{remarque-deschoices}, cela entra\^ine que $v-w_0\in
I^+(\sigma,\pi)\cap I^-(\sigma,\pi)$, d'o\`u $\ell_{\sigma}(v)\leq
n-1$ par (i), ce qui donne une contradiction et l'\'enonc\'e s'en d\'eduit.

(iii) 
Comme $\Pi(v)\in I^+(\sigma,\pi)$, la condition $\ell_{\sigma}(v)=n$ implique que $\Pi(v)$ appartient \`a $I^{+,n}(\sigma,\pi)$, \emph{a fortiori}, \`a $\langle
K\cdot {\mathrm{\Pi}}(I^{+,n-1}(\sigma,\pi))\rangle$. Comme $Sv=\sum_{\lambda\in\F_q}\smatr{[\lambda]}110\Pi(v)$, on en d\'eduit l'appartenance de $Sv$ \`a $\langle
K\cdot {\mathrm{\Pi}}(I^{+,n-1}(\sigma,\pi))\rangle$, d'o\`u l'\'enonc\'e d'apr\`es le lemme \ref{lemma-I+,n}(ii).
\end{proof}

%
%
%


\begin{lem}\label{lemma-c_m=0}
Soit $v\in \pi$ un vecteur v\'erifiant une \'equation de la forme
\begin{equation}\label{equation-c_m=0}
c_0v+\summ_{n=1}^{m}c_nS^nv=0
\end{equation}
avec  $m\in\N_{\geq 1}$ et $(c_n)_{0\leq n\leq m}$ une famille d'\'el\'ements de $\bFp$ v\'erifiant $c_m\neq 0$.
Alors, en choisissant une d\'ecomposition $v=v^++v^-$ avec $v^+\in I^+(\sigma,\pi)$ et $v^-\in I^-(\sigma,\pi)$, on a   $\ell_{\sigma}(v)=\ell_{\sigma}(v^-)$. De plus, on a $\ell_{\sigma}(v^+)=\ell_{\sigma}(v^-)$ si et seulement si $\ell_{\sigma}(v)=0$.
\end{lem}
\begin{proof}
Supposons au contraire que $\ell_{\sigma}(v)=\ell_{\sigma}(v^+)>\ell_{\sigma}(v^-)$. En particulier, $\ell_{\sigma}(v^+)\geq 1$. D'apr\`es le lemme \ref{lemma-niveau}(ii), on a  pour tout $1\leq n\leq m$
\[\ell_{\sigma}(S^nv^+)=\ell_{\sigma}(v^+)+n,\]
ce qui donne une contradiction avec (\ref{equation-c_m=0}) si $\ell_{\sigma}(v^-)=0$, car on peut alors supposer $v^-=0$ et (\ref{equation-c_m=0}) devient $c_0v^++\sum_{n=1}^mc_nS^nv^+=0$. La derni\`ere conclusion dans ce cas est triviale.

Si $\ell_{\sigma}(v^-)\geq 1$, par le lemme \ref{lemma-niveau}(iii), on a $\ell_{\sigma}(Sv^-)\leq \ell_{\sigma}(v^-)$ et donc pour tout $1\leq n\leq m$:
\[\ell_{\sigma}(S^nv^-)=\ell_{\sigma}(S^{n-1}Sv^-)\leq \ell_{\sigma}(v^-)+n-1<\ell_{\sigma}(v^+)+m =\ell_{\sigma}(S^mv^+)\]
o\`u l'on a utilis\'e le fait que $\ell_{\sigma}(Sw)\leq \ell_{\sigma}(w)+1$ pour tout $w\in\pi$.
On obtient encore une contradiction avec (\ref{equation-c_m=0}) puisque $c_n\neq 0$ et que
\begin{equation}\label{equation-l(n)>others}\ell_{\sigma}(S^mv^+)>\max\{\ell_{\sigma}(S^kv^+),\ell_{\sigma}(S^nv^-)\}_{0\leq k\leq m-1,\ 0\leq n\leq m}.\end{equation}
Cela montre que $\ell_{\sigma}(v^+)\leq \ell_{\sigma}(v^-)$.
La derni\`ere conclusion se d\'eduit de ce qui pr\'ec\`ede, car si $\ell_{\sigma}(v^+)=\ell_{\sigma}(v^-)\geq 1$ on aurait encore (\ref{equation-l(n)>others}) et donc une contradiction.
\end{proof}
\begin{cor}\label{coro-adm-I1-in-D1}
Si $\pi$ est admissible, on a $\pi^{I_1}\subseteq I^+(\sigma,\pi)\cap I^-(\sigma,\pi)$.
\end{cor}
\begin{proof}
Soit $v\in \pi^{I_1}$ un vecteur non nul. Alors il en est de m\^eme de $S^nv$ pour tout $n\geq 1$  par le lemme \ref{lemma-H-invariant}(iv). L'admissibilit\'e de $\pi$ implique qu'il existe
une famille $(c_n)_{0\leq n\leq m}$ d'\'el\'ements de $\bFp$ avec $c_m\neq 0$ v\'erifiant l'\'equation (\ref{equation-c_m=0}).
En choisissant une d\'ecomposition $v=v^++v^-$ avec $v^+\in I^+(\sigma,\pi)$ et $v^-\in I^-(\sigma,\pi)$, on obtient par le lemme \ref{lemma-c_m=0} que $\ell_{\sigma}(v^+)\leq \ell_{\sigma}(v^-)$.
De m\^eme, en appliquant le m\^eme argument \`a $\Pi(v)\in\pi^{I_1}$ qui se d\'ecompose sous la forme $\Pi(v^-)+\Pi(v^+)$ avec $\Pi(v^-)\in I^+(\sigma,\pi)$ et $\Pi(v^+)\in I^-(\sigma,\pi)$, on voit que \[\ell_{\sigma}(v^-)=\ell_{\sigma}(\Pi v^-)\leq \ell_{\sigma}(\Pi v^+)=\ell_{\sigma}(v^+)\]
et donc $\ell_{\sigma}(v^+)=\ell_{\sigma}(v^-)$.  Le corollaire s'en d\'eduit en utilisant \`a nouveau le lemme \ref{lemma-c_m=0}.
\end{proof}
\begin{rem}
L'admissibilit\'e de $\pi$ est cruciale dans le corollaire \ref{coro-adm-I1-in-D1}. En fait, on verra plus tard (proposition \ref{prop-cal-D1}) que l'espace $I^+(\sigma,\cInd_{KZ}^G\sigma/T)\cap I^-(\sigma,\cInd_{KZ}^G\sigma/T)$ est de dimension finie o\`u $T$ est l'op\'erateur de Hecke d\'efini au \S\ref{subsection-Hecke}, tandis que l'espace des $I_1$-invariants de $\cInd_{KZ}^G\sigma/T$ est de dimension infinie d\`es que $F\neq \Q_p$ (voir \cite[remarque 4.2.6]{Br1}).
Malgr\'e tout, on a un \'enonc\'e analogue si $\pi$ est irr\'eductible (proposition \ref{prop-I1-invariant-in-D1}).
\end{rem}





\subsubsection{Techniques pour calculer $I^+(\sigma,\pi)\cap I^-(\sigma,\pi)$}\label{subsubsection-techniques}

Maintenant, on introduit un moyen qui permet de calculer explicitement $I^+(\sigma,\pi)\cap I^-(\sigma,\pi)$ dans certains cas.\vv

Consid\'erons le
morphisme compos\'e $IZ$-\'equivariant
\begin{equation}\label{equation-define-Phi}\Phi_{\sigma}=\Phi_{\sigma,\pi}:\cInd_{KZ}^G\sigma\twoheadrightarrow I^-(\sigma)
\twoheadrightarrow I^-(\sigma,\pi)\hookrightarrow \pi.\end{equation}
Plus pr\'ecis\'ement, si $f=f^++f^-\in \cInd_{KZ}^G\sigma$ avec $f^+\in I^+(\sigma)$ et $f^-\in I^-(\sigma)$, alors $\Phi_{\sigma}(f):=\overline{f^-}$ est l'image de $f^-$ dans $\pi$. \'Evidemment, $\Phi_\sigma$ est un morphisme $IZ$-\'equivariant. 
Rappelons que l'on a d\'efini  l'espace $R(\sigma,\pi)$ comme le noyau de l'application naturelle
$\cInd_{KZ}^G\sigma\twoheadrightarrow \pi$.
\begin{lem}\label{lemma-cK=intesection}
L'espace $I^+(\sigma,\pi)\cap I^-(\sigma,\pi)$ est l'image de
$R({\sigma},\pi)$ via $\Phi_{\sigma}$. 
\end{lem}
\begin{proof}
%
 Si $f=f^++f^-\in R(\sigma,\pi)$ avec $f^+\in I^+(\sigma)$ et $f^-\in
I^-(\sigma)$, alors
par d\'efinition,
\[\Phi_{\sigma}(f)=\overline{f^-}=-\overline{f^+}\in I^+(\sigma,\pi)\cap I^-(\sigma,\pi).\]
R\'eciproquement, soit $\overline{f}\in I^+(\sigma,\pi)\cap
I^-(\sigma,\pi)$, et $f^+$ (resp. $f^-$) un rel\`evement de $\overline{f}$ dans $I^+(\sigma)$
(resp. $I^-(\sigma)$), alors $f:=-f^++f^-$ est un \'el\'ement dans
$R(\sigma,\pi)$ v\'erifiant $\Phi_{\sigma}(f)=\overline{f}$.
\end{proof}
\begin{lem}\label{lemma-Phi(Pi)=-Pi(Phi)}
On a $\Phi_{\sigma}(\Pi(f))=-\Pi(\Phi_{\sigma}(f))$ pour $f\in R(\sigma,\pi)$.
\end{lem}
\begin{proof}
Si l'on \'ecrit $f=f^++f^-$ avec $f\in I^+(\sigma)$ et $f^-\in I^-(\sigma)$, alors $\Pi(f)=\Pi(f^-)+\Pi(f^+)$ avec $\Pi(f^-)\in I^+(\sigma)$ et $\Pi(f^+)\in I^-(\sigma)$ de sorte que
\[\Phi_{\sigma}(\Pi(f))=\overline{\Pi(f^+)}=\Pi(\overline{f^+})=-\Pi(\overline{f^-})=-\Pi(\Phi_{\sigma}(f))\]
dont la troisi\`eme \'egalit\'e r\'esulte du fait que $\overline{f^+}+\overline{f^-}=0$ en notant que $f\in R(\sigma,\pi)$.
\end{proof}


Soit $f=f^++f^-\in \cInd_{KZ}^G\sigma$ un vecteur avec $f^+\in I^+(\sigma)$ et $f^-\in I^-(\sigma)$. Comme $\Pi(f^-)\in I^+(\sigma)=R_0^+(\sigma)\oplus(\oplus_{n\geq 1}R_n^+(\sigma))$, la d\'ecomposition (\ref{equation-R_n(sigma)=summ}) donne  l'\'ecriture suivante de $\Pi(f^-)$:
\begin{equation}\label{equation-ecrire-Pi(f-)}\Pi(f^-)=y+\summ_{\mu\in\F_q}\matr{[\mu]}110x_{\mu}\end{equation}
avec $y\in R_0^+(\sigma)$ et $x_{\mu}\in I^-(\sigma)$. De plus, on a  $x_{\mu}\in \oplus_{0\leq n\leq m-1}R_n^-(\sigma)$ si $f^-\in\oplus_{0\leq n\leq m}R_n^-(\sigma)$.

\begin{lem}\label{lemma-cK-et-D1(pi)}
Avec les notations pr\'ec\'edentes, on a pour tout $\lambda\in\F_q$:
\[\Phi_{\sigma}\biggl(\matr{\varpi}{[\lambda]}01f\biggr)=\matr1{[\lambda]}01\overline{x}_0.\]
\end{lem}

\begin{proof}
On calcule en utilisant (\ref{equation-ecrire-Pi(f-)}):
\[\begin{array}{rll}\matr {\varpi}{[\lambda]}01f&=&
\matr {\varpi}{[\lambda]}01f^++\matr{[\lambda]}110{\mathrm{\Pi}} (f^-)\\
&=&\matr {\varpi}{[\lambda]}01f^++\matr{[\lambda]}110 y+\matr{1}{[\lambda]}01 x_0 \\
&& +\summ_{\mu\in\F_q^{\times}}\matr{[\lambda]+[\mu^{-1}]}{1}10
\matr{[\mu]}00{-[\mu^{-1}]}\matr{1}{[\mu^{-1}]}01
x_{\mu}\\
&\in&\matr{1}{[\lambda]}01 x_0+I^+(\sigma),
\end{array}\]
et la conclusion s'en d\'eduit par la d\'efinition de $\Phi_\sigma$ (\ref{equation-define-Phi}).
\end{proof}
\subsubsection{Exemples} 
On donne des exemples illustratant le calcul de l'espace $I^+(\sigma,\pi)\cap I^-(\sigma,\pi)$  en utilisant les techniques au \S\ref{subsubsection-techniques}.
Le cas le plus simple est celui o\`u $\pi=\cInd_{KZ}^G\sigma/P(T)$ avec $P(T)\in\bFp[T]$ un polyn\^ome non constant, o\`u $T$ est l'op\'erateur de Hecke d\'efini au \S\ref{subsubsection-T}.

Traitons d'abord le cas o\`u $P(T)$ est de degr\'e 1, c'est-\`a-dire,  on suppose que
\begin{equation}\label{equation-define-V()}
\pi=V(\sigma,\lambda):=\cInd_{KZ}^G\sigma/(T-\lambda)
\end{equation}
 avec $\lambda\in\bFp$. Si
$f\in\cInd_{KZ}^G\sigma$, on  d\'esigne par $\overline{f}$ son image
dans $\pi$. Soit $v_0\in\sigma$ un vecteur
non nul fix\'e par $I_1$.

\begin{prop}\label{prop-cal-D1}
Avec les notations pr\'ec\'edentes, on a
\[I^+(\sigma,V(\sigma,\lambda))\cap I^-(\sigma,V(\sigma,\lambda))=\bFp \overline{[\mathrm{Id},v_0]}
\oplus \bFp\overline{[{\mathrm{\Pi}},v_0]}.\]
\end{prop}
\begin{proof}
Commen\c{c}ons par remarquer que $\bFp\overline{[\id,v_0]}\cap \bFp\overline{[\Pi,v_0]}=0$, i.e.
\[[\Pi,v_0]\notin \bFp[\id,v_0]+(T-\lambda)(\cInd_{KZ}^G\sigma).\]
Ce fait peut se voir en examinant les supports comme dans la preuve de \cite[th\'eor\`eme 3.2.4]{Br1}.

L'inclusion $\supseteq$ r\'esulte de la proposition
\ref{lemma-R_n<R_n+2}(i) et du fait que le membre de gauche est stable par $\Pi$.
Pour l'autre inclusion, on utilise le lemme
\ref{lemma-cK=intesection} qui permet d'identifier
$I^+(\sigma,V(\sigma,\lambda))\cap
I^-(\sigma,V(\sigma,\lambda))$ avec l'image de
$(T-\lambda)(\cInd_{KZ}^G\sigma)$ via $\Phi_{\sigma}$ d\'efini par (\ref{equation-define-Phi}). Autrement dit, il faut v\'erifier que: quel que soit $g\in G$,
\[\Phi_{\sigma}\bigl((T-\lambda)([g,v_0])\bigr)\in \bFp \overline{[\mathrm{Id},v_0]}
\oplus \bFp\overline{[{\mathrm{\Pi}},v_0]}.\]
D'apr\`es le lemme \ref{lemma-Phi(Pi)=-Pi(Phi)}, on peut supposer $g\in P^+KZ$, ce qui fait que $g$ s'\'ecrit  sous la forme comme dans (\ref{equation-g=g^i-vigneras}):
\[g=g^{(n)}g^{(n-1)}\cdots g^{(1)}k.\]

En utilisant le lemme \ref{lemma-Sv=Tv} et la $G$-\'equivariance de $T$, on d\'eduit que
\[(T-\lambda)([\mathrm{Id},kv_0])=k\cdot(T-\lambda)([\id,v_0])\in \langle K\cdot[\Pi,v_0]\rangle.\]
Par la d\'ecomposition (\ref{equation-decom-K/I}) et la d\'efinition de $\Phi_{\sigma}$, cela implique que \[\Phi_{\sigma}\bigl((T-\lambda)([\mathrm{Id},kv_0])\bigr)\in\bFp\overline{[\Pi,v_0]},\]
d'o\`u l'\'enonc\'e dans le cas particulier o\`u $g=k\in KZ$ (i.e. $n=0$).
Traitons le cas o\`u $n\geq 1$. Comme $f^-\in R_0^-(\sigma)$, le $x_0$ d\'efini dans l'\'ecriture (\ref{equation-ecrire-Pi(f-)}) r\'eduit \`a $0$, donc
le lemme \ref{lemma-cK-et-D1(pi)} 
implique que
\[\Phi_{\sigma}\bigr((T-\lambda)([g^{(1)}k,v_0])\bigl)=0.\]
Le m\^eme raisonnement donne $\Phi_{\sigma}\bigl((T-\lambda)([g,v_0])\bigr)=0$, et le r\'esultat s'en d\'eduit.
\end{proof}



\begin{prop}\label{prop-I(sigma)/P(T)}
Supposons que $\pi=\cInd_{KZ}^G\sigma/P(T)$ avec $P(T)\in\bFp[T]$ un polyn\^ome de degr\'e $k\geq 1$. Alors $I^+(\sigma,\pi)\cap I^-(\sigma,\pi)\subseteq \pi^{I_1}$.
\end{prop}
\begin{proof}
Soient $v_0\in\sigma$ un vecteur non nul fix\'e par $I_1$ et $f=P(T)[\id,v_0]$ de telle sorte que $P(T)\cInd_{KZ}^G\sigma=\langle G\cdot f\rangle$ puisque $T$ est un endomorphisme $G$-\'equivariant de $\cInd_{KZ}^G\sigma$. Par le lemme \ref{lemma-cK=intesection}, l'\'enonc\'e \'equivaut \`a dire que  $\Phi_{\sigma}(P(T)\cInd_{KZ}^G\sigma)\subseteq\pi^{I_1}$, ou plut\^ot, que $\Phi_{\sigma}(g\cdot f)\in\pi^{I_1}$ pour tout $g\in G$. De plus, le lemme \ref{lemma-Phi(Pi)=-Pi(Phi)} permet de se ramener au cas $g\in P^+ KZ$.

Comme $f$ est fix\'e par $I_1$, il s'\'ecrit (de mani\`ere unique) sous la forme $f=f^++f^-$ avec $f^+\in I^+(\sigma)^{I_1}$ et $f^-\in I^-(\sigma)^{I_1}$, et le lemme \ref{lemma-Sv=Tv}  implique que $f^+\in\oplus_{0\leq m\leq k}R_{m}^+(\sigma)$ et $f^-\in\oplus_{1\leq n\leq k-1}R_n^-(\sigma)$. On en d\'eduit en utilisant la proposition \ref{prop-1-dim} que:
\[f^{+}=\summ_{0\leq m\leq k}c_mS^m[\id,v_0],\ \ \Pi(f^-)=\summ_{0\leq n \leq k-1}d_nS^n[\id,v_0]\]
pour des $c_m,d_n\in\bFp$ convenables. En r\'e\'ecrivant $f^+$ sous la forme
\[f^+=c_0[\id,v_0]+\summ_{\mu\in\F_q}\matr{[\mu]}110\summ_{1\leq m\leq k}c_m\Pi S^{m-1}[\id,v_0],\]
on obtient par le calcul  du lemme \ref{lemma-cK-et-D1(pi)} que: pour tout $\lambda\in\F_q$,
\[\matr{[\lambda]}110f\in \summ_{1\leq m\leq k}c_m\Pi S^{m-1}[\id,v_0]+I^+(\sigma),\]
o\`u l'on a utilis\'e le fait que $\Pi S^m[\id,v_0]$ sont fix\'es par $I_1\cap U^+$ pour tout $m\geq 0$.
En particulier,
\[\Phi_{\sigma}\biggl(\matr{[\lambda]}110f\biggr)=\summ_{1\leq m\leq k}c_m\overline{\Pi S^{m-1}[\id,v_0]}\in\pi^{I_1}\]
ce qui montre que $\Phi_{\sigma}(gf)\in\pi^{I_1}$ si $g\in KZ$ en utilisant la d\'ecomposition (\ref{equation-decom-K/I}).
On proc\`ede de m\^eme pour $g\in P^+KZ$, $g\notin KZ$.
\end{proof}

\subsection{Le diagramme canonique associ\'e \`a une repr\'esentation irr\'eductible de $G$}
\label{subsection-definition}


Fixons $\pi$ une repr\'esentation lisse \emph{irr\'eductible} de $G$ (admettant un caract\`ere central). 
On  d\'efinit le \emph{diagramme
canonique}  associ\'e \`a $\pi$ et on en fournit des propri\'et\'es \'el\'ementaires. On le d\'etermine explicitement lorsque $\pi$ est non supersinguli\`ere ou $F=\Q_p$.\vv

\subsubsection{D\'efinition et propri\'et\'es du diagramme canonique}\label{subsubsection-def-diagcan}

En choisissant $\sigma$ une sous-$KZ$-repr\'esentation irr\'eductible de $\pi$, on obtient une surjection $G$-\'equivariante $\cInd_{KZ}^G\sigma\twoheadrightarrow\pi$ ainsi que les sous-espaces $I^+(\sigma,\pi)$ et $I^-(\sigma,\pi)$ de $\pi$ (\S\ref{subsection-preliminaire}).
En g\'en\'eral, le choix de $\sigma$ n'est pas unique. Comme on l'a signal\'e dans la remarque \ref{remark-Ind-nonadm}, $\cInd_{KZ}^G\sigma$ est r\'eductible, donc on peut appliquer \`a $\pi$ les r\'esultats du \S\ref{subsection-preliminaire}, o\`u l'on n'a fait aucune hypoth\`ese sp\'ecifique sur le quotient non trivial de $\cInd_{KZ}^G\sigma$.\vv



\begin{prop}\label{prop-I1-invariant-in-D1}
On a  $\pi^{I_1}\subseteq I^+(\sigma,\pi)\cap
I^-(\sigma,\pi)$.
\end{prop}
\begin{proof}
Par d\'efinition, un vecteur $v\in\pi$ appartient \`a $I^+(\sigma,\pi)\cap I^-(\sigma,\pi)$ si et
seulement si $\ell_{\sigma}(v)=0$. Supposons $v$ non nul fix\'e par $I_1$. Alors $S^nv$ l'est aussi pour tout entier $n\geq 1$ par le lemme \ref{lemma-H-invariant}(iv).

Supposons d'abord $\pi$ non supersinguli\`ere. En particulier, $\pi$ est admissible (\cite{BL2}), donc le corollaire \ref{coro-adm-I1-in-D1} permet de conclure.

Traitons le cas o\`u $\pi$ est supersinguli\`ere. D'apr\`es le lemme \ref{lemma-S^m=0-pas}, on a $S^mv=0$ pour $m\gg 1$ suffisamment grand.
En particulier, $v$ v\'erifie l'hypoth\`ese du lemme \ref{lemma-c_m=0}. De m\^eme, $\Pi(v)\in \pi^{I_1}$ la v\'erifie aussi. Donc l'argument que l'on a employ\'e dans la d\'emonstration du corollaire \ref{coro-adm-I1-in-D1} permet de conclure.
\end{proof}

\begin{cor}\label{corollary-D1-deuxpoids}
Soient $\sigma'$ une autre sous-$KZ$-repr\'esentation irr\'eductible de $\pi$ (peut \^etre isomorphe \`a $\sigma$) et $\cInd_{KZ}^G\sigma'\twoheadrightarrow \pi$
la surjection $G$-\'equivariante induite.
Alors
\[I^+(\sigma,\pi)=I^{+}(\sigma',\pi),\ \ I^-(\sigma,\pi)=I^-(\sigma',\pi).\]
\end{cor}
\begin{proof}
Rappelons que $I^+(\sigma,\pi):=\langle P^+\cdot \sigma\rangle$. Prenons $v_0\in \sigma^{I_1}$ un
vecteur non nul. Par le lemme \ref{lemma-I+(sigma)=P+Pi(v)} ci-apr\`es, on a  $I^+(\sigma,\pi)=\langle P^+\cdot{\mathrm{\Pi}}(v_0)\rangle$ et
$I^-(\sigma,\pi)=\langle{\mathrm{\Pi}} P^+\cdot {\mathrm{\Pi}}(v_0)\rangle$. Autrement dit, $x$ appartient \`a $
I^+(\sigma,\pi)$ (resp. \`a $I^-(\sigma,\pi)$) si et seulement s'il
existe $Q_1\in \bFp[P^+]$ (resp. $Q_2\in \bFp[P^+]$) tel que
$x=Q_1{\mathrm{\Pi}} \cdot v_0$ (resp. $x={\mathrm{\Pi}} Q_2{\mathrm{\Pi}}\cdot v_0$).

Cela dit, le corollaire
d\'ecoule de la proposition \ref{prop-I1-invariant-in-D1}(i).
Plus pr\'ecis\'ement, prenons $w_0\in\sigma'$ un vecteur non nul fix\'e par $I_1$,
alors cette proposition dit que $w_0\in I^+(\sigma,\pi)\cap I^-(\sigma,\pi)$, et donc par ce qui pr\'ec\`ede, il existe
$Q_1,Q_2\in\bFp[P^+]$ tels que
\[w_0=Q_1{\mathrm{\Pi}}\cdot v_0={\mathrm{\Pi}} Q_2{\mathrm{\Pi}}\cdot v_0.\]
En appliquant le lemme \ref{lemma-I+(sigma)=P+Pi(v)} \`a $\sigma'$ et \`a $w_0$, on en d\'eduit que:
\[I^+(\sigma',\pi)=\langle P^+\cdot \Pi(w_0)\rangle =\langle P^+\cdot \Pi^2Q_2\Pi(v_0)\rangle \subseteq \langle P^+\cdot\Pi(v_0)\rangle =I^+(\sigma,\pi),\]
d'o\`u l'inclusion $I^+(\sigma',\pi)\subseteq I^+(\sigma,\pi)$.
Le r\'esultat s'en d\'eduit en \'echangeant $\sigma'$ et $\sigma$ et en remarquant que $I^-(\sigma,\pi)=\Pi\cdot I^+(\sigma,\pi)$.
\end{proof}
\begin{lem}\label{lemma-I+(sigma)=P+Pi(v)}
Soit $v_0\in\sigma^{I_1}$ un vecteur non nul. Dans $\cInd_{KZ}^G\sigma$, on a \[I^+(\sigma)=[P^+\Pi,v_0],\ \ I^-(\sigma)=[\Pi P^+ \Pi,v_0].\]
\end{lem}
\begin{proof}
Il suffit de v\'erifier l'\'enonc\'e concernant $I^+(\sigma)$. Par d\'efinition, $I^+(\sigma)=[P^+,\sigma]=P^+\cdot[\id,\sigma]$. Or, par le corollaire \ref{cor-intersection}, l'espace
$[\id,\sigma]$ est juste l'espace engendr\'e par
\[\biggl\{\Bigl[\matr{[\lambda]}110,v_0\Bigr],\ \ \lambda\in\F_q\biggr\}\]
qui n'est autre que (\`a une constante dans $\bFp^{\times}$ pr\`es)
\[\biggl\{\Bigl[\matr{\varpi}{[\lambda]}01\Pi,v_0\Bigr],\ \ \lambda\in\F_q\biggr\}.\]
Le lemme s'en d\'eduit en notant que $\smatr{\varpi}{[\lambda]}01\in P^+$.
\end{proof}
\vv


Rappelons (\cite{BP,Pa}) qu'un \emph{diagramme} est
par d\'efinition un triplet $(D_0,D_1,r)$ o\`u $D_0$ est une
repr\'esentation lisse de $KZ$, $D_1$ est une repr\'esentation lisse
de $N$ et $r: D_1\ra D_0$ est un morphisme $IZ$-\'equivariant. On
d\'efinit des morphismes entre deux diagrammes de mani\`ere \'evidente et
on note $\DIAG$ la cat\'egorie r\'esult\'ee.  Notons qu'elle est \'equivalente \`a la cat\'egorie des syst\`emes de coefficients \'equivariants sur l'arbre de $G$ (\cite{Pa}). \vspace{1mm}

Donnons quelques exemples de diagrammes.
\begin{exem}\label{exemple-diag}
Soit $\pi'$ une repr\'esentation lisse de $G$ admettant un caract\`ere central. Alors  (o\`u $\mathrm{can}$ d\'esigne l'inclusion naturelle)

(i) $\cK(\pi'):=(\pi'|_{KZ},\pi'|_{N},\mathrm{can})$ est un diagramme;

(ii) $(\pi'^{K_1},\pi'^{I_1},\mathrm{can})$ est un diagramme; plus g\'en\'eralement, $(\pi'^{K_{n}},\pi'^{I_n},\mathrm{can})$ est un diagramme pour tout $n\geq 1$.

(iii)  $(W,W\cap \Pi(W),\mathrm{can})$ est un diagramme pour toute sous-$KZ$-repr\'esentation $W$ de $\pi'$, et $(\langle K\cdot M\rangle, M, \mathrm{can})$ l'est aussi pour toute sous-$N$-repr\'esentation $M$ de $\pi'$.
\end{exem}

D'apr\`es le corollaire \ref{corollary-D1-deuxpoids},
l'espace $I^+(\sigma,\pi)$ (resp.
$I^-(\sigma,\pi)$, $I^{+,n}(\sigma,\pi)$, le niveau
$\ell_{\sigma}(\cdot)$, etc.) ne d\'epend que de $\pi$, on peut donc le noter
$I^{+}(\pi)$ (resp. $I^-(\pi)$, $I^{+,n}(\pi)$,
$\ell(\cdot)$, etc.). Posons:
\begin{equation}\label{equation-diag-can}D_1(\pi):=I^+(\pi)\cap I^-(\pi),\ \ D_0(\pi):=\langle K\cdot D_1(\pi)\rangle\subset \pi.\end{equation}
Alors $D_1(\pi)$ est stable par $N$ et $D_0(\pi)$ par $KZ$.  En fait, $D_1(\pi)$ est le plus grand sous-espace vectoriel de
$D_0(\pi)$ stable par $N$. On pose la d\'efinition suivante:\vspace{1mm}

\begin{defn}\label{definition-diagcano}
Le \emph{diagramme canonique} associ\'e \`a $\pi$ est
le diagramme
\[D(\pi):=(D_0(\pi),D_1(\pi),\mathrm{can})\]
o\`u $\mathrm{can}$ d\'esigne l'inclusion naturelle
$D_1(\pi)\hookrightarrow D_0(\pi)$.
\end{defn}

Le r\'esultat principal de l'article  est le th\'eor\`eme suivant,
qui dit que, en passant de $\pi$ \`a son diagramme canonique, on ne
perd pas d'information.

\begin{thm}\label{theorem-diagcanonique-rep}
Pour toute repr\'esentation lisse irr\'eductible $\pi'$
de $G$, on a $\pi\cong \pi'$ si et seulement si $D(\pi)\cong D(\pi')$ en tant que
diagrammes.
\end{thm}

En fait, on va d\'emontrer un r\'esultat plus g\'en\'eral:
\begin{thm}\label{theorem-diag-morphisme}
Pout toute repr\'esentation lisse $\pi'$ de $G$ (pas forc\'ement irr\'eductible), il existe un isomorphisme naturel d'espaces vectoriels:
\[\Hom_{\DIAG}(D(\pi),\cK(\pi'))\cong \Hom_{G}(\pi,\pi'),\]
o\`u $\cK(\pi')$ est le diagramme d\'efini dans l'exemple \ref{exemple-diag}(i).
\end{thm}

\begin{proof}
On d\'efinit d'abord un morphisme
\begin{equation}\label{equation-def-iota}\iota:
\Hom_{\DIAG}(D(\pi),\cK(\pi'))\ra \Hom_G(\pi,\pi').\end{equation}
Soit $(\varphi_0,\varphi_1):D(\pi)\ra \cK(\pi')$ un morphisme non
nul de diagrammes, c'est-\`a-dire, $\varphi_0:D_0(\pi)\ra
\pi'|_{KZ}$ est un morphisme $KZ$-\'equivariant et
\[\varphi_1=\varphi_0|_{D_1(\pi)}:D_1(\pi)\ra \pi'|_N\] est un morphisme
$N$-\'equivariant. On va d\'efinir un morphisme $G$-\'equivariant
$\varphi:\pi\ra \pi'$ \`a partir de $(\varphi_0,\varphi_1)$
v\'erifiant
$\varphi|_{D_0(\pi)}=\varphi_0$.

D'abord, si un tel morphisme existe, il est n\'ecessairement unique
puisque $\pi$ est engendr\'ee par $D_0(\pi)$ en tant que
$G$-repr\'esentation. Plus pr\'ecis\'ement, si $v\in\pi$, on d\'efinit
$\varphi(v)$ comme suit par r\'ecurrence sur le niveau de $v$:
\begin{itemize}
\item[(a)] Si $\ell(v)=0$, alors $v\in D_1(\pi)$ et on pose
$\varphi(v)=\varphi_1(v)$.

\item[(b)] Si $v^+\in I^+(\pi)$ et $\ell(v^+)=n\geq 1$, soit
$w_{\lambda}\in I^{+,n-1}(\pi)$ (avec $\lambda\in\F_q$) des
\'el\'ements tels que
\[v^+=\summ_{\lambda\in\F_q}\matr{[\lambda]}{1}{1}{0}{\mathrm{\Pi}}(w_{\lambda}),\]
on pose
\[\varphi(v^+):=\summ_{\lambda\in\F_q}\matr{[\lambda]}{1}{1}{0}{\mathrm{\Pi}}\varphi(w_\lambda)\]
o\`u $\varphi(w_{\lambda})$ a \'et\'e d\'efini par r\'ecurrence. C'est bien
d\'efini car: si
\[v^+=\summ_{\lambda\in\F_q}\matr{[\lambda]}{1}{1}{0}{\mathrm{\Pi}}(w_{\lambda}')\]
pour d'autres $w_{\lambda}'\in I^{+,n-1}(\pi)$, alors on a
$w_{\lambda}-w'_{\lambda}\in D_1(\pi)$ pour tout $\lambda\in\F_q$
et donc
\[\begin{array}{rll}&&\summ_{\lambda\in\F_q}\matr{[\lambda]}{1}{1}{0}{\mathrm{\Pi}}\varphi(w_{\lambda})
-\summ_{\lambda\in\F_q}\matr{[\lambda]}{1}{1}{0}{\mathrm{\Pi}}\varphi(w_{\lambda}')\\
&=&\summ_{\lambda\in\F_q}\matr{[\lambda]}{1}{1}{0}{\mathrm{\Pi}}\varphi_1(w_{\lambda}-w_{\lambda}')\\
&=&\varphi_0\Bigl(\summ_{\lambda\in\F_q}\matr{[\lambda]}{1}{1}{0}{\mathrm{\Pi}}(w_{\lambda}-w_{\lambda}')\Bigr)\\
&=&0
\end{array}\]
o\`u la deuxi\`eme \'egalit\'e r\'esulte des faits que $w_{\lambda}-w_{\lambda}'\in D_1(\pi)$ pour tout $\lambda\in\F_q$
et que $(\varphi_0,\varphi_1)$ est un morphisme de
diagrammes.

\item[(c)] Si $v=v^++{\mathrm{\Pi}}(\tilde{v}^+)$ avec $v^+, \tilde{v}^+\in
I^+(\pi)$, alors on pose
\[\varphi(v)=\varphi(v^+)+{\mathrm{\Pi}}\varphi(\tilde{v}^+).\]
On v\'erifie que cette d\'efinition ne d\'epend pas de la d\'ecomposition
en utilisant que
\[I^+(\pi)\cap {\mathrm{\Pi}}(I^+(\pi))=D_1(\pi)\] et que $\varphi_1$ est
$N$-\'equivariant.
\end{itemize}
\

\'Evidemment, l'application $\varphi$ ainsi d\'efinie est lin\'eaire. Il faut v\'erifier qu'elle est
$G$-\'equivariante, i.e. quelle que soit l'\'egalit\'e $v=\summ_{i\in
S}g_iv_i$ (avec $g_i\in G$, $v,v_i\in \pi$, et $S$ un ensemble
fini d'indices) dans $\pi$, on doit avoir dans $\pi'$
\begin{equation}\label{equation-thm-diagmorph}\varphi(v)=\summ_{i\in S}g_i\varphi(v_i).
\end{equation}
Comme ce qu'on veut d\'emontrer est vrai sur $D_0(\pi)$ (voir (b) et
(c) ci-dessus) et comme $\pi$ est engendr\'ee par $D_1(\pi)$ en tant
que $G$-repr\'esentation, on peut supposer que tous les $v_i$ sont
dans $D_1(\pi)$. Puisqu'il existe des $h_j\in G$ et des
$w_j\in D_1(\pi)$ tels que
\[v=\summ_jh_jw_j\ \ \mathrm{et}\ \ \varphi(v)=\summ_jh_j\varphi(w_j),\]
on peut supposer de plus $v=0$.\vspace{1mm}

Rappelons que l'on peut \'ecrire $g=g^{(n)}\cdots g^{(1)}k$ comme dans (\ref{equation-g=g^i-vigneras}) si $g\in P^+KZ$ et $g=\Pi g^{(n)}\cdots g^{(1)}k$ si $g\in \Pi P^+ KZ$.  La longueur de $g$ est d\'efinie par $\ell(g)=n$ si $g\in P^+KZ$ et $\ell(g)=n+1$ si $g\in \Pi P^+KZ$.

On va v\'erifier (\ref{equation-thm-diagmorph}) par r\'ecurrence sur l'entier
$m:=\max_{i\in S}\{\ell(g_i)\}$. Le cas o\`u $m=0$ est \'evident
puisqu'alors $g_i\in KZ$. En g\'en\'eral, on pose
\[S_{\lambda}=\{i\in S|\ g_i\in P^+KZ,\  g_i^{(\ell(g_i))}=g_{\lambda}\},\ \ S_{{\mathrm{\Pi}}}=\{i\in S|\ g_i\in \Pi P^+KZ\}\]
de telle sorte qu'on ait
\[0=\summ_{\lambda\in\F_q}g_{\lambda}\big(\summ_{i\in S_{\lambda}}g_i'v_i\big)+{\mathrm{\Pi}} \big(\summ_{i\in S_{\mathrm{\Pi}}}g_i'v_i\big)\]
o\`u l'on a \'ecrit $g_i=g_{\lambda}g_i'$ si $i\in S_{\lambda}$ et $g_i'=g_i^+$ si $i\in S_{\Pi}$. Comme $g_i'\in
P^+$ et $v_i\in D_1(\pi)\subseteq I^+(\pi)$, le lemme \ref{lemme-deschoices} implique
\[\sum_{i\in
S_{\lambda}}g_i'v_i\in D_1(\pi)\ \ \textrm{et}\ \ \sum_{i\in
S_{\mathrm{\Pi}}}g_i'v_i\in D_1(\pi)\] pour tout $\lambda\in\F_q$. On en
d\'eduit avec $\varphi|_{D_1(\pi)}=\varphi_1$ que:
\[0=\summ_{\lambda\in\F_q}g_{\lambda}\varphi\big(\summ_{i\in S_{\lambda}}g_i'v_i\big)+
{\mathrm{\Pi}}\varphi\big(\summ_{i\in S_{{\mathrm{\Pi}}}}g_i'v_i\big).\] Donc
l'hypoth\`ese de r\'ecurrence permet de conclure que $\varphi$ est
$G$-\'equivariant. Autrement dit,
$\iota:(\varphi_0,\varphi_1)\mapsto \varphi$ donne le morphisme
(\ref{equation-def-iota}) demand\'e
\vv

Ensuite, on d\'efinit un inverse de (\ref{equation-def-iota}). Soit $\varphi:\pi\ra \pi'$ un morphisme
$G$-\'equivariant, on en d\'eduit un morphisme de diagrammes
\[\cK(\varphi):\cK(\pi)\ra\cK(\pi'),\]
et en le composant avec le morphisme naturel
$D(\pi)\hookrightarrow \cK(\pi)$, on obtient un morphisme de diagrammes
$D(\pi)\ra \cK(\pi')$. Ceci d\'efinit un morphisme
\[\kappa: \Hom_G(\pi,\pi')\ra \Hom_{\DIAG}(D(\pi),\cK(\pi')).\]
On v\'erifie que $\iota\circ \kappa(\varphi)=\varphi$ et
$\kappa\circ\iota((\varphi_0,\varphi_1))=(\varphi_0,\varphi_1)$,
ce qui permet de conclure.
\end{proof}
\begin{rem}\label{remark-expliquer-3.2}
On voit que le lemme \ref{lemme-deschoices} joue un r\^ole central pour assurer la $G$-\'equivariance de $\varphi$ dans la d\'emonstration du th\'eor\`eme \ref{theorem-diag-morphisme}. En revanche, l'irr\'eductibilit\'e de $\pi$ n'est utilis\'ee que pour garantir l'ind\'ependance de $\sigma$ de $D(\pi)$ (et de $\ell(\cdot)$, etc.). Donc la construction marche bien pour $\pi$ un quotient non trivial de $\cInd_{KZ}^G\sigma$ (pas forc\'ement irr\'eductible): il suffit de remplacer $D(\pi)$ (resp. $\ell(\cdot)$ etc.) par $D(\sigma,\pi)$ (resp. $\ell_{\sigma}(\cdot)$ etc.). On a choisi de l'\'enoncer ici dans un cadre plus restrictif parce que le cas o\`u $\pi$ est irr\'eductible est plus int\'eressant.
\end{rem}

\begin{rem}\label{remark-irreductible}
Conservons les notations du th\'eor\`eme \ref{theorem-diag-morphisme}.
Par construction, si \[(\varphi_0,\varphi_1):D(\pi)\ra \cK(\pi')\]
est non nul, alors $\varphi=\iota((\varphi_0,\varphi_1)):\pi\ra
\pi'$ l'est aussi. L'irr\'eductibilit\'e de $\pi$ implique que
$\varphi$ est de plus injectif, et donc
$(\varphi_0,\varphi_1)=\kappa(\varphi)$ l'est aussi. Cette
propri\'et\'e peut \^etre vue comme \emph{irr\'eductibilit\'e} de $D(\pi)$. 
\end{rem}
\vv

\begin{proof}[D\'emonstration du th\'eor\`eme
\ref{theorem-diagcanonique-rep}] La condition est
suffisante d'apr\`es le th\'eor\`eme
\ref{theorem-diag-morphisme}.

R\'eciproquement, soit $\varphi:\pi\simto\pi'$ un isomorphisme
$G$-\'equivariant. Alors $\sigma':=\varphi(\sigma)\subset\pi'$ est
une sous-$KZ$-repr\'esentation irr\'eductible de $\pi'$ en rappelant que
$\sigma\subset \pi$ l'est, de telle sorte que $D_1(\pi)=\langle P^+\cdot \sigma\rangle\cap \langle \Pi P^+\cdot\sigma\rangle$ et $D_1(\pi')=\langle P^+\cdot \sigma'\rangle\cap \langle \Pi P^+\cdot\sigma'\rangle$. Cela implique l'inculsion
$\varphi(D_1(\pi))\subseteq D_1(\pi')$  puis une injection de diagrammes
\[\kappa(\varphi):D(\pi)\ra D(\pi'),\]
qui est un isomorphisme
car $\kappa(\varphi^{-1})$ en fournit
un inverse.
\end{proof}\vv

Rappelons que dans \cite[\S 9]{BP} est d\'efini un foncteur $H_0$
de $\DIAG$ dans $\Rep_{G}$ (voir aussi \cite[\S 5.2]{Pa} o\`u l'objet ``syst\`eme de coefficients $G$-\'equivariants'' intervient). Pour $D=(D_0,D_1,r)$ un diagramme, on dispose d'une suite exacte:
\begin{equation}\label{equation-def-H_0}\cInd_{N}^G(D_1\otimes\delta_{-1})\overset{\partial}{\ra}\cInd_{KZ}^GD_0\ra H_0(D)\ra0,\end{equation}
o\`u $\delta_{-1}:N\ra \bFp^{\times}$ d\'esigne le caract\`ere
donn\'e par $g\mapsto (-1)^{\mathrm{val}_F(\det(g))}$, et o\`u $\partial$ est le compos\'e des $G$-morphismes suivants qui sont d\'efinis de mani\`ere \'evidente:
\[\partial: \cInd_{N}^G(D_1\otimes\delta_{-1})\hookrightarrow\cInd_{IZ}^GD_1\overset{r}{\ra} \cInd_{IZ}^GD_0\twoheadrightarrow \cInd_{KZ}^GD_0.\]
De mani\`ere explicite, si $x\in D_1\otimes\delta_{-1}$, on a \begin{equation}\label{equation-def-partial}
\partial([\id,x])=[\id,r(x)]-[\Pi,r(\Pi^{-1}\cdot x)]\in\cInd_{KZ}^GD_0\end{equation}
ce qui d\'etermine compl\`etement $\partial$ par sa $G$-\'equivariance.


En particulier, \`a partir de la repr\'esentation irr\'eductible $\pi$ fix\'ee, on obtient $H_0(D(\pi))$ en appliquant $H_0$ au diagramme canonique $D(\pi)$.

\begin{cor}\label{cor-H0(D)}
Avec les notations pr\'ec\'edentes, on a $H_0(D(\pi))\cong \pi$.
\end{cor}
\begin{proof}
Compte-tenu du th\'eor\`eme \ref{theorem-diag-morphisme}, le corollaire se d\'eduit de la propri\'et\'e que (\cite[proposition  5.4.3]{Pa}):
\[\Hom_{\DIAG}(D(\pi),\cK(\pi'))\cong \Hom_{G}(H_0(D(\pi)),\pi')\]
pour toute repr\'esentation lisse $\pi'$ de $G$. En effet, en prenant $\pi'=\pi$, l'inclusion $D(\pi)\hookrightarrow\cK(\pi)$ induit un $G$-morphisme non trivial $\phi:H_0(D(\pi))\twoheadrightarrow\pi$. D'autre part, en appliquant \cite[lemme 9.9]{BP} \`a $D=D(\pi)$ et \`a $\Omega=\pi$, on voit que le morphisme compos\'e
\begin{equation}\nonumber
D_0(\pi)\ra H_0(D(\pi))\ra \pi\end{equation}
est injectif. En particulier, on obtient une injection de diagrammes $D(\pi)\hookrightarrow \cK\bigl(H_0(D(\pi))\bigr)$ qui induit par le th\'eor\`eme \ref{theorem-diag-morphisme} un $G$-morphisme non trivial $\psi:\pi\ra H_0(D(\pi))$. On v\'erifie que $\phi\circ\psi=\mathrm{id}_{\pi}$, ce qui fait que $H_0(D(\pi))$ s'\'ecrit d'une somme directe de $\pi$ avec une certaine sous-$G$-repr\'esentation $X$. Or, par d\'efinition (\ref{equation-def-H_0}), $H_0(D(\pi))$ est engendr\'ee comme $G$-repr\'esentation par $D_0(\pi)$ qui est en fait contenu dans $\pi$. Cela force que $X=0$ et le corollaire s'en d\'eduit.
\end{proof}
\begin{rem}\label{remark-H0(D)}
Supposons de plus $\pi$ admissible. D'apr\`es \cite[th\'eor\`eme 9.8]{BP}, il existe une injection de diagrammes
\[(\varphi_0,\varphi_1):D(\pi)\hookrightarrow \cK(\Omega)\]
o\`u $\Omega$ est une repr\'esentation lisse de $G$ telle que
$\Omega|_K\cong\rInj_K\rsoc_K(\pi)$ est une enveloppe injective
de $\rsoc_K(\pi)$ dans la cat\'egorie $\Rep_K$. Posons
\[\pi':=\langle G\cdot \varphi_1(D_1(\pi))\rangle\subseteq
\Omega\] la sous-$G$-repr\'esentation de $\Omega$ engendr\'ee par
$\varphi_1(D_1(\pi))$. Le th\'eor\`eme
\ref{theorem-diag-morphisme} entra\^ine alors que $\pi'\cong\pi$. Autrement dit, $\pi'$ ne d\'epend ni du choix de $\Omega$ ni du choix de
$(\varphi_0,\varphi_1)$ (\`a isomorphisme pr\`es).
\end{rem}





\subsubsection{Le cas non supersingulier}

On d\'etermine explicitement $D(\pi)$ lorsque $\pi$ est non supersinguli\`ere.

\begin{thm}\label{theorem-non-super}
Supposons $\pi$ non
supersinguli\`ere. Alors
\[D(\pi)=(\pi^{K_1},\pi^{I_1},\mathrm{can}).\]
En particulier, $D_1(\pi)$ est de dimension $\leq 2$.
\end{thm}
\begin{proof}
Rappelons que $\sigma$ est une sous-$KZ$-repr\'esentation irr\'eductible fix\'ee de $\pi$. D'apr\`es \cite[th\'eor\`emes 33, 34]{BL2}, la surjection $\cInd_{KZ}^G\sigma\twoheadrightarrow\pi$ se factorise par $V(\sigma,\lambda)$ pour $\lambda\in\bFp^{\times}$ convenable (en fait, $\lambda$ est d\'etermin\'e par le choix de $\sigma$), o\`u $V(\sigma,\lambda)$ est la
$G$-repr\'esentation d\'efinie par (\ref{equation-define-V()}).

Si $\pi$ est un caract\`ere, tous les \'enonc\'es sont
\'evidents. Si $\pi$ est une s\'erie principale, i.e. $\pi\cong \Ind_P^G\chi_1\otimes\chi_2$ pour $\chi_1$ et $\chi_2$ deux caract\`eres lisses  convenables de $F^{\times}$ (rappelons que $P\subset G$ d\'esigne le sous-groupe de Borel), la d\'ecomposition d'Iwasawa donne:
\[\pi^{K_1}=(\Ind_P^G\chi_1\otimes\chi_2)^{K_1}\cong\Ind_I^K\chi_1\otimes\chi_2.\]
En cons\'equence, $\pi^{I_1}$ est de dimension 2 et $\pi^{K_1}=\langle K\cdot\pi^{I_1}\rangle$ (voir par exemple \cite[\S2]{BP}). D'autre part, on sait d'apr\`es \cite[th\'eor\`eme 30]{BL2} que $\pi\cong
V(\sigma,\lambda)$, donc $D_1(\pi)$ est aussi de dimension 2 par la proposition \ref{prop-cal-D1}. L'\'enonc\'e s'en d\'eduit  en utilisant la proposition \ref{prop-I1-invariant-in-D1}.

Il reste \`a traiter le cas  o\`u $\pi$ est une s\'erie sp\'eciale. Sans perdre de g\'en\'eralit\'e on suppose que le caract\`ere central de $\pi$ est trivial sur $\varpi\in Z$. D'apr\`es \cite[th\'eor\`eme 30]{BL2}, $\sigma$ est de dimension $q$ et est l'unique sous-$KZ$-repr\'esentation irr\'eductible de $\pi$; de plus, $\lambda=1$ et $\pi$ est le quotient de $V(\sigma,1)$
par un
sous-espace vectoriel de dimension 1. Plus pr\'ecis\'ement, si $v_0\in \sigma$ est un vecteur non nul fix\'e par
$I_1$ et si l'on pose
\[f=[\mathrm{Id},v_0]+[\Pi,v_0]\in\cInd_{KZ}^G\sigma,\]
alors $f\notin (T-1)(\cInd_{KZ}^G\sigma)$  et on a une suite
exacte de $G$-repr\'esentations (voir \cite[\S3.4]{BL1})
\[0\ra R(\sigma,\pi)\ra \cInd_{KZ}^G\sigma\ra \pi\ra0\]
avec $R(\sigma,\pi)=(T-1)(\cInd_{KZ}^G\sigma)+ \bFp f$. Le lemme \ref{lemma-cK=intesection} nous permet d'identifier $D_1(\pi)$ avec l'image de $R(\sigma,\pi)$ dans $\pi$ via le morphisme $\Phi_{\sigma}$ qui est d\'efini par (\ref{equation-define-Phi}). D'une part, on a par d\'efinition $\Phi_{\sigma}(f)=\overline{[\Pi,v_0]}$, d'autre part, la proposition \ref{prop-cal-D1} nous dit que \[\Phi_{\sigma}\bigl((T-1)\cInd_{KZ}^G\sigma\bigr)=\bFp\overline{[\id,v_0]}+\bFp\overline{[\Pi,v_0]}\in\pi.\]
Compte-tenu de l'\'egalit\'e $\overline{[\id,v_0]}=-\overline{[\Pi,v_0]}$ dans $\pi$, on en d\'eduit que $D_1(\pi)=\bFp \overline{[\mathrm{Id},v_0]}$ est de
dimension 1. Cela suffit pour conclure que $D_1(\pi)=\pi^{I_1}$ parce que $\pi^{I_1}\subseteq D_1(\pi)$ (proposition \ref{prop-I1-invariant-in-D1}) et que $\pi^{I_1}$ est
aussi de dimension 1 (\cite[lemme 27]{BL1}). Enfin, l'\'enonc\'e que $\pi^{K_1}=D_0(\pi)=\langle K\cdot\pi^{I_1}\rangle$ s'en d\'eduit aussi:  comme $\sigma$ est   un objet injectif dans  la cat\'egorie $\Rep_{K/K_1}$ (\cite[lemme 3.2(iii)]{BP}), on doit avoir $\pi^{K_1}\cong\sigma\oplus X$ avec $X$ une certaine sous-$K$-repr\'esentation de $\pi^{K_1}$; or,  le fait que $\pi^{I_1}$ est de dimension 1 implique que $\pi^{I_1}=\sigma^{I_1}$, puis $X^{I_1}=0$ et $X=0$, d'o\`u l'\'egalit\'e $\pi^{K_1}=\sigma$ et l'\'enonc\'e s'en d\'eduit.
\end{proof}


\subsubsection{Le cas o\`u $F=\Q_p$}

\begin{thm}\label{theorem-Breuil-Paskunas}
Supposons $F=\Q_p$ et $\pi$ supersinguli\`ere. Alors $D_1(\pi)=\pi^{I_1}$ et $D_1(\pi)$ est de dimension 2.
%
\end{thm}
\begin{proof}
Dans \cite{Pa2}, l'espace $I^+(\pi)$ est not\'e $M$ et, le th\'eor\`eme 6.3 \emph{loc. cit.} fournit une suite exacte de $I$-repr\'esentations
\begin{equation}\nonumber\label{equation-suite-Pas}
0\ra \pi^{I_1}\ra M\oplus {\mathrm{\Pi}}(M)\ra \pi\ra0,\end{equation}
c'est-\`a-dire, $D_1(\pi)=M\cap \Pi(M)=\pi^{I_1}$. La derni\`ere conclusion s'en d\'eduit en utilisant \cite[th\'eor\`eme 3.2.4]{Br1}.

Signalons que la proposition \ref{prop-cal-D1} donne aussi une
preuve du th\'eor\`eme \ref{theorem-Breuil-Paskunas} en utilisant \cite[th\'eor\`emes 1.1 et 3.2.4]{Br1}.
\end{proof}

\begin{rem}
(i) Ce n'est plus vrai que $D_0(\pi)=\pi^{K_1}$ dans le th\'eor\`eme \ref{theorem-Breuil-Paskunas}. En fait, $D_0(\pi)$ est \'egale au $K$-socle de $\pi$ (\cite[th\'eor\`eme 4.8]{Br2}).

(ii) Combin\'e avec le th\'eor\`eme \ref{theorem-diagcanonique-rep}, les th\'eor\`emes \ref{theorem-non-super} et \ref{theorem-Breuil-Paskunas} permettent de retrouver
les entrelacements entre les repr\'esentations lisses irr\'eductibles
de $G$ fournis dans \cite[corollaire 36]{BL2} et
\cite[th\'eor\`eme 1.3]{Br1}. 
\end{rem}\vv

\subsection{Le diagramme canonique associ\'e \`a un quotient non trivial de $\cInd_{KZ}^G\sigma$}\label{subsection-quotient}
Dans ce paragraphe, $\pi$ sera un $G$-quotient \emph{non trivial} de
$\cInd_{KZ}^G\sigma$ avec $\sigma$ un poids fix\'e  Rappelons que $I^+(\sigma):=[P^+,\sigma]$ et $I^-(\sigma)=\Pi\cdot I^+(\sigma)$, avec $I^+(\sigma,\pi)$ et $I^-(\sigma,\pi)$ d\'esignant respectivement leurs images dans $\pi$.\vv

\textbf{Notation}: on pose
$D_1(\sigma,\pi):=I^+(\sigma,\pi)\cap I^-(\sigma,\pi)$ et
\[D(\sigma,\pi):=(\langle KZ\cdot D_1(\sigma,\pi)\rangle, D_1(\sigma,\pi),\mathrm{can})\]
qui est un sous-diagramme de $\cK(\pi)$ qui a \'et\'e d\'efini \`a l'exemple \ref{exemple-diag}(iii). Par abus de notation (en
raison de la d\'efinition \ref{definition-diagcano}), on l'appelle
le \emph{diagramme canonique} associ\'e \`a $\pi$. Au contraire du cas irr\'eductible, $D(\sigma,\pi)$ d\'epend de $\sigma$.

La m\^eme construction que celle du th\'eor\`eme \ref{theorem-diag-morphisme}, comme on l'a d\'ej\`a mentionn\'e dans la remarque \ref{remark-expliquer-3.2}, donne un isomorphisme naturel d'espaces vectoriels
\begin{equation}\label{equation-diag-morp-D(sigma)}
\Hom_{\DIAG}(D(\sigma,\pi),\cK(\pi'))\cong\Hom_G(\pi,\pi')\end{equation}
pour toute repr\'esentation lisse $\pi'$ de $G$.
Autrement dit, on a 
\begin{equation}\label{equation-H0(D(sigma))}
H_0(D(\sigma,\pi))\cong \pi.
\end{equation}

\begin{prop}\label{prop-exact-complexe}
Supposons $\pi$ admissible et soit
$\pi_1$ une sous-$G$-repr\'esentation irr\'eductible de $\pi$.
Notons $\pi_2$ le quotient de $\pi$ par $\pi_1$. Alors la suite
exacte
\[0\ra \pi_1\ra \pi\ra \pi_2\ra0\]
induit (i) une injection $D_1(\pi_1)\hookrightarrow
D_1(\sigma,\pi)$ et

(ii) une surjection $D_1(\sigma,\pi)\twoheadrightarrow
D_1(\sigma,\pi_2)$.
\end{prop}
\begin{proof}

Remarquons tout d'abord que la proposition est triviale si $\pi$ est irr\'eductible, on suppose donc que $\pi$ est r\'eductible. D'autre part, $\pi_1$ \'etant irr\'eductible, $D_1(\pi_1)$ est bien d\'efini par le corollaire \ref{corollary-D1-deuxpoids}.\vv

(i) Il suffit de v\'erifier que $I^+(\pi_1)\subseteq I^+(\sigma,\pi)$, ce qui revient \`a montrer que $\pi_1^{I_1}\subseteq I^+(\sigma,\pi)$, car le lemme \ref{lemma-I+(sigma)=P+Pi(v)} implique qu'il existe $v_1\in\pi_1^{I_1}$ non nul tel que $I^+(\pi_1)=\langle P^+\cdot v_1\rangle$. L'\'enonc\'e d\'ecoule donc du corollaire \ref{coro-adm-I1-in-D1} puisque $\pi$ est admissible.
\vv

(ii) Fixons $v_0\in \sigma\hookrightarrow \pi$ un vecteur non nul fix\'e
par $I_1$ et soit $\overline{v_0}\in\pi_2$ l'image de $v_0$ dans $\pi_2$.
Comme $\pi$ est engendr\'ee par $v_0$ en tant que $G$-repr\'esentation et comme $\pi_1$ est une sous-$G$-repr\'esentation propre de $\pi$, on voit que $\overline{v_0}\neq 0$. D'apr\`es le lemme
\ref{lemma-I+(sigma)=P+Pi(v)}, un vecteur $x\in \pi$ (resp. $\overline{x}\in \pi_2$) appartient \`a $D_1(\sigma,\pi)$ (resp. \`a
$D_1(\sigma,\pi_2)$) si et seulement s'il existe $Q_1,
Q_2\in \bFp[P^+]$ tels que $x=Q_1{\mathrm{\Pi}}\cdot v_0={\mathrm{\Pi}} Q_2{\mathrm{\Pi}}\cdot v_0$
(resp. $\overline{x}=Q_1{\mathrm{\Pi}}\cdot \overline{v_0}={\mathrm{\Pi}} Q_2{\mathrm{\Pi}}\cdot
\overline{v_0}$).

Soit maintenant $\overline{x}\in D_1(\sigma,\pi_2)$ et soient $Q_1,Q_2\in \bFp[P^+]$ tels que
\[\overline{x}=Q_1{\mathrm{\Pi}}\cdot \overline{v_0}={\mathrm{\Pi}} Q_2{\mathrm{\Pi}} \cdot\overline{v_0}.\]
Posons $x_1=Q_1{\mathrm{\Pi}}\cdot v_0\in I^+(\sigma,\pi)$ et $x_2={\mathrm{\Pi}} Q_2 {\mathrm{\Pi}}\cdot
v_0\in I^-(\sigma,\pi)$. Par construction, on a $x_1-x_2\in \pi_1$ et donc il existe $x_1'\in I^+(\pi_1)$ et $x_2'\in I^-(\pi_1)$ tels que \[x_1-x_2=x_1'+x_2'.\]
Par (i), on a $I^+(\pi_1)\subset I^+(\sigma,\pi)$ et $I^-(\pi_1)\subset I^-(\sigma,\pi)$. On en d\'eduit que
$x_1-x_1'$  (resp. $x_2+x_2'$) est un rel\`evement
de $\overline{x}$ dans $I^+(\sigma,\pi)$ (resp.
$I^-(\sigma,\pi)$), d'o\`u le r\'esultat.
\end{proof}
\begin{cor}\label{corollaire-deuxcas=}
Supposons $\pi$ admissible. Si
$D_1(\sigma,\pi)$ est de dimension finie, alors $\pi$ est de
longueur finie.
\end{cor}
\begin{proof}
Comme $\pi$ est admissible, elle admet une sous-repr\'esentation
irr\'eductible $\pi_1$ de $G$. Notons $\pi_2$ le quotient de
$\pi$ par $\pi_1$. Alors la proposition \ref{prop-exact-complexe}
entra\^ine que
\[\dim_{\bFp}D_1(\sigma,\pi)\geq \dim_{\bFp}D_1(\pi_1)+\dim_{\bFp}D_1(\sigma,\pi_2).\]
On a deux cas \`a distinguer selon la caract\'eristique de $F$:

\begin{enumerate}
\item[(i)] Le cas o\`u $F$ est une extension finie de $\Q_p$. Alors d'apr\`es \cite[th\'eor\`eme 2]{Vi}, $\pi_2$ est aussi admissible, donc
une r\'ecurrence imm\'ediate sur la dimension de $D_1(\sigma,\pi)$
permet de conclure.

\item[(ii)] Le cas o\`u $F$ est de caract\'eristique $p$ sera trait\'e au \S\ref{subsubsection-coro-cas(ii)}.
\end{enumerate}
\end{proof}


Comme application, on d\'etermine compl\`etement $D_1(\sigma,\pi)$ dans le cas o\`u $F=\Q_p$:
\begin{cor}\label{coro-Ollivier}
Supposons $F=\Q_p$. Alors $D_1(\sigma,\pi)=\pi^{I_1}$ et $D_1(\sigma,\pi)$ est de dimension finie.
\end{cor}
\begin{proof}
Par la preuve de \cite[lemme 3.2]{Pa3}, la surjection $\cInd_{KZ}^G\sigma\twoheadrightarrow\pi$ se factorise par $\cInd_{KZ}^G\sigma/P(T)$ avec $P(T)\in \bFp[T]$ un polyn\^ome de degr\'e $\geq 1$. Puis, en utilisant la classification de Barthel-Livn\'e \cite{BL2} et Breuil \cite{Br1}, on v\'erifie facilement que $\cInd_{KZ}^G\sigma/P(T)$ est admissible et de longueur finie. En cons\'equence, $\pi$ est aussi admissible par \cite[th\'eor\`eme 2]{Vi}, ce qui fait que la finitude de la dimension de $D_1(\sigma,\pi)$ suivra l'\'egalit\'e $D_1(\sigma,\pi)=\pi^{I_1}$.
Par ailleurs, on d\'eduit du corollaire \ref{coro-adm-I1-in-D1} et de la proposition \ref{prop-I(sigma)/P(T)} que:
\[D_1(\sigma,\cInd_{KZ}^G\sigma/P(T))=\bigl(\cInd_{KZ}^G\sigma/P(T)\bigr)^{I_1}.\]

Soit $\pi'$ le noyau du morphisme surjectif $\cInd_{KZ}^G\sigma/P(T)\twoheadrightarrow\pi$. Si $\pi'=0$, le corollaire est d\'ej\`a montr\'e par ce qui pr\'ec\`ede. Sinon, soient $\pi_1'$ une sous-repr\'esentation irr\'eductible de $\pi'$ et $\pi_1$ le quotient de $\cInd_{KZ}^G\sigma/P(T)$ par $\pi_1'$. Alors, de m\^eme que $\cInd_{KZ}^G\sigma/P(T)$, $\pi_1$ est un quotient admissible de $\cInd_{KZ}^G\sigma$ par \cite[th\'eor\`eme 2]{Vi}. En utilisant le corollaire \ref{coro-adm-I1-in-D1} et la proposition \ref{prop-exact-complexe}(ii), on obtient le diagramme commutatif suivant:
\[\xymatrix{\bigl(\cInd_{KZ}^G\sigma/P(T)\bigr)^{I_1}\ar[r]\ar@{=}[d]&\pi_1^{I_1}\ar@{^{(}->}[d]\\
D_1(\sigma,\cInd_{KZ}^G\sigma/P(T))\ar@{>>}[r]&D_1(\sigma,\pi_1),}\]
d'o\`u l'\'egalit\'e $D_1(\sigma,\pi_1)=\pi_1^{I_1}$. Comme $\cInd_{KZ}^G\sigma/P(T)$ est de longueur finie, en r\'ep\'etant cette proc\'edure, on arrive finalement \`a avoir $D_1(\sigma,\pi)=\pi^{I_1}$.
\end{proof}

\section{La condition de finitude}\label{section-condition-Pe}
Soient $\sigma$ un poids et  $\pi$ un $G$-quotient non trivial de $\cInd_{KZ}^G\sigma$. On a vu au \S\ref{section-diag-canonique}
le r\^ole important jou\'e par $D_1(\sigma,\pi)=I^+(\sigma,\pi)\cap I^-(\sigma,\pi)$.
Sans faire d'hypoth\`ese sur $F$, on donne au \S\ref{subsection-D1(sigma,pi)}  des conditions n\'ecessaires et suffisantes pour que $D_1(\sigma,\pi)$ soit un $\bFp$-espace vectoriel de dimension finie.  Au \S\ref{subsection-U+-invariants}, on consid\`ere l'espace des $I_1\cap U^+$-invariants de $I^+(\sigma,\pi)$ et l'on montre qu'il est contenu dans $D_1(\sigma,\pi)$.

Remarquons qu'au \S\ref{subsection-D(pi)=K(pi)}, on va montrer que l'espace $D_1(\sigma,\pi)$ est de dimension infinie dans le cas particulier o\`u $F$ est de caract\'eristique $p$ et $\pi$ est irr\'eductible supersinguli\`ere.  

\subsection{L'espace $D_1(\sigma,\pi)$}\label{subsection-D1(sigma,pi)}

\subsubsection{Repr\'esentations de pr\'esentation finie}
Conservons les notations ci-dessus.
En g\'en\'eral, il n'est pas clair que $D_1(\sigma,\pi)$ soit de dimension finie
ou non, ce qui conduit \`a introduire la condition suivante:
\[\emph{$D_1(\sigma,\pi)$ est de dimension finie.}\]

On dit que $\pi$ satisfait $(P_e)$ avec $e\geq 0$ un entier (fix\'e), si l'on a $D_1(\sigma,\pi)\subseteq I^+(\sigma,\pi)^{I_{e+1}}$.

\begin{rem}\label{remark-twocond-admissible}
Il est trivial que la condition <<$D_1(\sigma,\pi)$ est de dimension
finie>> est plus forte que la condition $(P_e)$ pour un entier
suffisamment grand, et elles sont \'equivalentes si $\pi$ est
admissible.

R\'eciproquement, dans le cas o\`u $\pi$ est irr\'eductible, si $D_1(\pi)$ est de
dimension finie, alors $\pi$
est admissible. Cet \'enonc\'e se d\'eduit du fait que $D_1(\pi)$
contient tous les $I_1$-invariants de $\pi$
(proposition \ref{prop-I1-invariant-in-D1}).
\end{rem}

Comme cons\'equence directe du th\'eor\`eme \ref{theorem-diag-morphisme}, on a
\begin{cor}
Si $\pi$ et $\pi'$ sont deux quotients non triviaux de
$\cInd_{KZ}^G\sigma$ v\'erifiant $(P_e)$, alors $\pi\cong\pi'$ si et seulement si, en
tant que diagrammes,
\[(\pi^{K_{e+1}},\pi^{I_{e+1}},\mathrm{can})\cong (\pi'^{K_{e+1}},\pi'^{I_{e+1}},\mathrm{can}).\]
\end{cor}
\begin{proof}
En effet, soit \[(\varphi_0,\varphi_1):(\pi^{K_{e+1}},\pi^{I_{e+1}},\mathrm{can})\simto (\pi'^{K_{e+1}},\pi'^{I_{e+1}},\mathrm{can})\]
un isomorphisme de diagrammes. Comme $D(\sigma,\pi)\subseteq (\pi^{K_{e+1}},\pi^{I_{e+1}},\mathrm{can})$ par hypoth\`ese, compte-tenu de l'inclusion naturelle $(\pi'^{K_{e+1}},\pi'^{I_{e+1}},\mathrm{can})\subseteq \cK(\pi')$, on obtient un morphisme $D(\sigma,\pi)\ra \cK(\pi')$ qui induit un $G$-morphisme non trivial $\pi\ra\pi'$ par (\ref{equation-diag-morp-D(sigma)}). C'est un isomorphisme puisque $(\varphi_0,\varphi_1)^{-1}$ en induit un inverse.
\end{proof}

\vv


\begin{thm}\label{theorem-Pe}
Soit $\pi$ un quotient non trivial de $\cInd_{KZ}^G\sigma$ et soit $R(\sigma,\pi)$ le noyau de la projection $\cInd_{KZ}^G\sigma\twoheadrightarrow\pi$. Alors les propri\'et\'es
suivantes sont \'equivalentes:

(i) $D_1(\sigma,\pi)$ est un $\bFp$-espace vectoriel de dimension finie;

(ii) $R(\sigma,\pi)$ est de type fini en tant que $\bFp[G]$-module.
\end{thm}

\begin{proof}
%
Prouvons l'implication (ii)$\Rightarrow$(i). Soit
$\{f_1,\cdots,f_k\}$ un syst\`eme g\'en\'erateur de $R(\sigma,\pi)$ en
tant que $\bFp[G]$-module et soit $m\geq 0$ un entier tel que tout
$f_i$ soit contenu dans $\oplus_{0\leq n\leq m}R_n(\sigma)$. Soit $M$
l'image de $\oplus_{0\leq n\leq m-1}({R_n^+(\sigma)}\oplus R_n^-(\sigma))$ dans $\pi$. On va d\'emontrer que $D_1(\sigma,\pi)\subseteq M$, ce qui suffit pour conclure puisque $M$ est bien de dimension finie. D'apr\`es le lemme \ref{lemma-cK=intesection}, cela revient \`a v\'erifier que $\Phi_{\sigma}(gf_i)\in M$ pour tout $g\in G$ et tout $1\leq i\leq k$. Par d\'efinition, $M$ est stable par $\Pi$, donc on peut supposer que $g\in P^+KZ$ en utilisant la d\'ecomposition (\ref{equation-vigneras}) et le lemme \ref{lemma-Phi(Pi)=-Pi(Phi)}.

\'Ecrivons $g=g^{(n)}\cdots g^{(1)}k$ comme dans (\ref{equation-g=g^i-vigneras}) avec $n=\ell(g)$ sa longueur. Si $g\in KZ$, alors $gf_i\in \oplus_{0\leq n\leq m}R_n(\sigma)$ et l'\'enonc\'e d\'ecoule imm\'ediatement de la d\'efinition de $\Phi_{\sigma}$. Si $\ell(g)\geq 1$, un argument analogue \`a celui de la proposition \ref{prop-cal-D1} utilisant le lemme \ref{lemma-cK-et-D1(pi)} permet aussi de conclure.

\vv 

Passons \`a la d\'emonstration de (i)$\Rightarrow$(ii).  Commen\c{c}ons par remarquer que la surjection
 $\cInd_{KZ}^G\sigma\twoheadrightarrow \pi$ se factorise par
 \[\cInd_{KZ}^G\sigma\twoheadrightarrow \cInd_{KZ}^G\sigma/P(T)\twoheadrightarrow\pi\]
 pour un polyn\^ome $P(T)\in\bFp[T]$ (voir la preuve de \cite[lemme 3.2]{Pa3}). D'apr\`es la proposition \ref{lemma-R_n<R_n+2}(ii), il existe $u\in\N$ tel que
 \begin{equation}\label{equation-f-in>0}\oplus_{0\leq n\leq u}R_n(\sigma)\subseteq \oplus_{1\leq n\leq u}R_n(\sigma,\pi')+P(T)\cInd_{KZ}^G\sigma.\end{equation}
Puisque
$D_1(\sigma,\pi)$ est suppos\'e de dimension finie, il existe un entier
$m\geq \max\{\deg P(T),u\}$ tel que $D_1(\sigma,\pi)$ est contenu dans l'image de
$\oplus_{0\leq n\leq m}R_n(\sigma)$ dans $\pi$. Ici, la condition sur $m$ est pour assurer, d'une part, que $P(T)\sigma\subseteq \oplus_{0\leq n\leq m}R_n(\sigma)$, d'autre part, que (\ref{equation-f-in>0}) reste vrai si l'on remplace $u$ par $m$.
Posons  $\ker=R(\sigma,\pi)\cap (\oplus_{0\leq n\leq m}R_n(\sigma))$.
On a alors $\langle G\cdot \ker\rangle\subseteq R(\sigma,\pi)$, et pour
achever la preuve il suffit de v\'erifier l'inclusion
$R(\sigma,\pi)\subseteq\langle G\cdot \ker\rangle$, ou plut\^ot, l'inclusion
\[R(\sigma,\pi)\cap \bigl(\oplus_{0\leq n\leq k}R_n(\sigma)\bigr)\subseteq \langle G\cdot\ker\rangle\]
pour tout $k\geq 0$. Par le choix de $m$,
 $P(T)(\cInd_{KZ}^G\sigma)\subseteq
 \langle G\cdot \ker\rangle$ puisque $P(T)\sigma\subseteq \ker$  et que $T$ est un endomorphisme $G$-\'equivariant de $\cInd_{KZ}^G\sigma$.

Soit $f\in R(\sigma,\pi)\cap\bigl(\oplus_{0\leq n\leq k}R_{n}(\sigma)\bigr)$. On va d\'emontrer que
$f\in \langle G\cdot \ker\rangle$ par r\'ecurrence sur $k$ comme
suit. D'abord, on peut supposer que $k\geq m+1$ par la d\'efinition de $\ker$,
puis que $f\in \oplus_{1\leq n\leq k} R_{n}(\sigma)$ gr\^ace \`a (\ref{equation-f-in>0}) et au fait que $P(T)\cInd_{KZ}^G\sigma\subseteq \langle G\cdot\ker\rangle$. L'espace
$\oplus_{1\leq n\leq k} R_{n}(\sigma)$ \'etant engendré par $\oplus_{0\leq n\leq k-1} R_{n}^-(\sigma)$ en tant que
$K$-repr\'esentation,
on peut \'ecrire
\[f=\summ_{\lambda\in\F_q}\matr {\varpi}{[\lambda]}01f_{\lambda}+{\mathrm{\Pi}}(f_{{\mathrm{\Pi}}})\]
avec $f_{\lambda},f_{{\mathrm{\Pi}}}\in \oplus_{0\leq n\leq k-1} R_{n}^+(\sigma)$. Comme
$f\in R(\sigma,\pi)$, le lemme \ref{lemme-deschoices} implique que l'image de $f_{\lambda}$ (resp. $f_{{\mathrm{\Pi}}}$) dans $\pi$, not\'ee $\overline{f_{\lambda}}$ (resp. $\overline{f_{{\mathrm{\Pi}}}}$), est
contenue dans $D_1(\sigma,\pi)$. Choisissons un rel\`evement de
$\overline{f_{\lambda}}$ (resp. $\overline{f_{{\mathrm{\Pi}}}}$) dans
$\oplus_{0\leq n\leq m}R_n(\sigma)$, disons $f_{\lambda}'$ (resp. $f'_{{\mathrm{\Pi}}}$),
alors (en notant que $k\geq m+1$)
\[f_{\lambda}-f_{\lambda}',\ f_{{\mathrm{\Pi}}}-f_{{\mathrm{\Pi}}}'\in R(\sigma,\pi)\cap
(\oplus_{0\leq n\leq k-1}R_n(\sigma)),\] et par hypoth\`ese de r\'ecurrence, on voit
que
\[f_{\lambda}-f_{\lambda}',\ f_{{\mathrm{\Pi}}}-f_{{\mathrm{\Pi}}}'\in\langle G\cdot \ker\rangle.\]
Maintenant si l'on pose
\[f'=\summ_{\lambda\in\F_q} \matr {\varpi}{[\lambda]}01 f'_{\lambda}+{\mathrm{\Pi}}(f'_{{\mathrm{\Pi}}}),\]
alors par d\'efinition $f'\in\ker$ et donc $f\in \langle G\cdot
\ker\rangle$, ce qui termine la d\'emonstration.
\end{proof}


\subsubsection{Repr\'esentations de pr\'esentation standard}

On fixe $\pi$ une repr\'esentation lisse de type fini de $G$ (ayant un caract\`ere central). Comme dans \cite{Co}, on note $\mathcal{W}(\pi)$ l'ensemble des sous-espaces de dimension finie de $\pi$, stables par $KZ$, engendrant $\pi$ comme $\bFp[G]$-module. Notons que $\mathcal{W}(\pi)$ n'est pas vide. Si $W\in \mathcal{W}(\pi)$, on dispose par r\'eciprocit\'e de Frobenius d'un morphisme surjectif $G$-\'equivariant $\cInd_{KZ}^GW\twoheadrightarrow\pi$ et on note $R(W,\pi)$ le noyau de ce morphisme. Cela justifie la notation $R(\sigma,\pi)$ lorsque $W=\sigma$ est irr\'eductible. On dit que $\cInd_{KZ}^GW\twoheadrightarrow\pi$ est une \emph{pr\'esentation finie} de $\pi$ si $R(W,\pi)$ est de type fini comme $\bFp[G]$-module. 

\begin{prop}\label{lemma-equi-presfinie-I}
Les deux conditions suivantes sont \'equivalentes:


(i) $\pi$ admet une pr\'esentation finie; 

(ii) pour tout $W\in\mathcal{W}(\pi)$, $\cInd_{KZ}^GW\twoheadrightarrow\pi$ est une pr\'esentation finie.
\end{prop}

Signalons que cette proposition n'est pas \'evidente, parce que $\bFp[G]$ n'est pas noetherien. La d\'emonstration va demander un peu de pr\'eparation: on a besoin de g\'en\'eraliser quelques r\'esultats du \S\ref{subsection-preliminaire}.\vv

Si $W\in\mathcal{W}(\pi)$, on pose pour $n\geq 0$,
\[R_n^+(W):=\Bigl[\matr{\varpi^n}{\cO}01,W\Bigr],\ \ R_n^-(W):=\Pi\cdot R_n^+(W).\]
On pose \'egalement
\[I^+(W):=[P^+,W], \ \ I^-(W):=\Pi\cdot I^+(W)  \]
de telles sortes que $I^+(W)=\oplus_{n\geq0}R_n^+(W)$, $I^-(W)=\oplus_{n\geq 0}R_n^-(W)$ et $\cInd_{KZ}^GW=I^+(W)\oplus I^-(W)$ par la d\'ecomposition (\ref{equation-vigneras}). Remarquons que ces sous-espaces de $\cInd_{KZ}^GW$ sont tous stables par $IZ$.

On note $I^+(W,\pi)$ (resp. $I^-(W,\pi)$, $R_n^+(W,\pi)$, $R_n^-(W,\pi)$) l'image de $I^+(W)$ (resp. $I^-(W)$, $R_n^+(W)$, $R_n^-(W)$) dans $\pi$ et on d\'efinit le morphisme compos\'e (pareil \`a $\Phi_{\sigma}$ du \S\ref{subsection-preliminaire}(\ref{equation-define-Phi})):
\[\Phi_{W}=\Phi_{W,\pi}:\cInd_{KZ}^GW\twoheadrightarrow I^-(W)\twoheadrightarrow I^-(W,\pi)\hookrightarrow \pi. \]

On peut v\'erifier que les lemmes \ref{lemma-decomposition}, \ref{lemma-cK=intesection}, \ref{lemma-Phi(Pi)=-Pi(Phi)} et \ref{lemma-cK-et-D1(pi)}, qui sont \'enonc\'es pour $W=\sigma$ irr\'eductible, restent vrais pour $W\in\mathcal{W}(\pi)$ g\'en\'eral. En fait, le fait que $\sigma$ est irr\'eductible, ainsi que l'op\'erateur $T$ associ\'e \`a $\cInd_{KZ}^G\sigma$, n'intervient jamais dans leurs preuves, mais seulement la stabilit\'e de $\sigma$ par $KZ$ et certaines d\'ecompositions de matrices dans $G$ y interviennent.  En particulier, on a
\begin{equation}\label{equation-I+(W,pi)=phi_W(R)}I^+(W,\pi)\cap I^-(W,\pi)=\Phi_{W}\bigl(R(W,\pi)\bigr).\end{equation}
De plus, on a un analogue du th\'eor\`eme \ref{theorem-Pe}:
\begin{lem}\label{lemma-typefini=finie-W}
Avec les notations pr\'ec\'edentes, $I^+(W,\pi)\cap I^-(W,\pi)$ est de dimension finie sur $\bFp$ si et seulement si $R(W,\pi)$ est de type fini en tant que $\bFp[G]$-module.
\end{lem}
\begin{proof}
On peut argumenter comme dans le th\'eor\`eme \ref{theorem-Pe} pour la direction $\Longleftarrow$, car les lemmes \ref{lemma-cK=intesection}, \ref{lemma-Phi(Pi)=-Pi(Phi)} et \ref{lemma-cK-et-D1(pi)} restent vrais dans ce cas plus g\'en\'eral. Supposons donc que $I^+(W,\pi)\cap I^-(W,\pi)$ est de dimension finie. On va proc\'eder par r\'ecurrence sur le nombre des facteurs de Jordan-H\"{o}lder de $W$. D'abord, le cas o\`u $W$ est irr\'eductible se d\'eduit du th\'eor\`eme \ref{theorem-Pe}. Supposons que $W$ est r\'eductible et soit $\sigma\subset W$ une sous-$KZ$-repr\'esentation irr\'eductible. En notant $\pi_1$ l'image de $\cInd_{KZ}^G\sigma$ dans $\pi$ et $\pi_2$ le quotient de $\pi$ par $\pi_1$, on obtient un diagramme commutatif \`a lignes et \`a colonnes exactes:
\begin{equation}\label{equation-diag-presfini}
\xymatrix{&0\ar[d]&0\ar[d]&0\ar[d]&\\
0\ar[r]&R(\sigma,\pi_1)\ar[r]\ar[d]&R(W,\pi)\ar[r]\ar[d]&R(W/\sigma,\pi_2)\ar[d]\ar[r]&0\\
0\ar[r]&\cInd_{KZ}^G\sigma\ar[r]\ar[d]&\cInd_{KZ}^GW\ar[r]\ar[d]&\cInd_{KZ}^GW/\sigma\ar[r]\ar[d]&0 \\
0\ar[r]&\pi_1\ar[r]\ar[d]&\pi\ar[r]\ar[d]&\pi_2\ar[r]\ar[d]&0\\
&0&0&0&,}\end{equation}
qui, en utilisant (\ref{equation-I+(W,pi)=phi_W(R)}), induit une injection \[D_1(\sigma,\pi_1)=I^+(\sigma,\pi_1)\cap I^-(\sigma,\pi_1)\hookrightarrow I^+(W,\pi)\cap I^-(W,\pi)\] et une surjection \[I^+(W,\pi)\cap I^-(W,\pi)\twoheadrightarrow I^+(W/\sigma,\pi_2)\cap I^-(W/\sigma,\pi_2).\] En particulier, $D_1(\sigma,\pi)$ et $I^+(W/\sigma,\pi_2)\cap I^-(W/\sigma,\pi_2)$ sont de dimension finie. D'apr\`es respectivement le th\'eor\`eme \ref{theorem-Pe} et l'hypoth\`ese de r\'ecurrence, on voit que $R(\sigma,\pi_1)$ et $R(W/\sigma,\pi_2)$ sont de type fini comme $\bFp[G]$-modules et donc $R(W,\pi)$ l'est aussi. Cela termine la d\'emonstration.
\end{proof}

\begin{proof}[D\'emonstration de la proposition \ref{lemma-equi-presfinie-I}]
\'Evidemment, la proposition se r\'eduit \`a montrer l'\'enonc\'e suivant: \emph{si $W_1\subset W_2$ sont deux \'el\'ements de $\mathcal{W}(\pi)$, alors la pr\'esentation $\cInd_{KZ}^GW_1\twoheadrightarrow\pi$ est finie si et seulement si la pr\'esentation $\cInd_{KZ}^GW_2\twoheadrightarrow\pi$ l'est}. De plus, par \cite[proposition 1]{Vi2}, si la pr\'esentation $\cInd_{KZ}^GW_1\twoheadrightarrow\pi$ est finie, alors $\cInd_{KZ}^GW_2\twoheadrightarrow\pi$ l'est aussi.

Supposons que $\cInd_{KZ}^GW_2\twoheadrightarrow\pi$ est une pr\'esentation finie, ce qui implique que $I^+(W_2,\pi)\cap I^+(W_2,\pi)$ est de dimension finie sur $\bFp$ d'apr\`es le lemme \ref{lemma-typefini=finie-W}. Or, l'inclusion naturelle $\cInd_{KZ}^GW_1\hookrightarrow\cInd_{KZ}^GW_2$ induit une injection
\[I^+(W_1,\pi)\cap I^-(W_1,\pi)\hookrightarrow I^+(W_2,\pi)\cap I^-(W_2,\pi).\]
On en d\'eduit que $I^+(W_1,\pi)\cap I^-(W_1,\pi)$ est aussi de dimension finie,  ce qui permet de conclure par le m\^eme lemme.
\end{proof}\vv

D'apr\`es \cite{Co}, si $W\in\mathcal{W}(\pi)$, on dit que $\cInd_{KZ}^GW\twoheadrightarrow\pi$ est \emph{une pr\'esentation standard} si $R(W,\pi)$ est engendr\'e comme $\bFp[G]$-module, par
\[R^{(0)}(W,\pi):=\biggl\{\Bigl[\matr{\varpi}001,x\Bigr]-\Bigl[\id,\matr{\varpi}001x\Bigr],\ \ x\in W\cap \matr{\varpi^{-1}}001 W\biggr\}.\]
Notons que l'on a toujours $\langle G\cdot R^{(0)}(W,\pi)\rangle\subseteq R(W,\pi)$, et que $R^{(0)}(W,\pi)=\{0\}$ si $W\cap\smatr{\varpi^{-1}}001 W=0$.
Comme $\smatr{\varpi^{-1}}001 W=\Pi\smatr0{\varpi^{-1}}{\varpi^{-1}}0 W=\Pi(W)$ et
\[\matr{\varpi^{-1}}001\cdot\Bigl[\id,\matr{\varpi}001 x\Bigr]=\Bigl[\matr01{\varpi}0\matr0{\varpi^{-1}}{\varpi^{-1}}0,\matr{\varpi}001x\Bigr]=[\Pi,\Pi^{-1}x],\]
on d\'eduit que:
\[\matr{\varpi^{-1}}001\cdot R^{(0)}(W,\pi)=\bigl\{[\id,x]-[\Pi,\Pi^{-1}(x)],\ \ x\in W\cap \Pi(W)\bigr\}.\]
D'autre part, en posant $W'=W\cap \Pi(W)$, on obtient un diagramme $(W,W',\mathrm{can})$. Compte-tenu de (\ref{equation-def-H_0}) et de (\ref{equation-def-partial}) du \S\ref{subsubsection-def-diagcan}, on a donc d\'emontr\'e le lemme suivant (je remercie B. Schraen pour m'avoir signal\'e ce fait).
\begin{lem}\label{lemma-benjamin}
Conservons les notations pr\'ec\'edentes.

(i) $\langle G\cdot R^{(0)}(W,\pi)\rangle$ co\"{i}ncide avec l'image de (voir (\ref{equation-def-H_0}))
\[\partial: \cInd_{N}^G(W'\otimes\delta_{-1})\ra \cInd_{KZ}^GW.\]
Par cons\'equent, la pr\'esentation $\cInd_{KZ}^GW\twoheadrightarrow\pi$ se factorise par $H_0\bigl((W,W',\mathrm{can})\bigr)$.

(ii) La pr\'esentation $\cInd_{KZ}^GW\twoheadrightarrow\pi$ est standard si et seulement si le morphisme surjectif $H_0\bigl((W,W',\mathrm{can})\bigr)\twoheadrightarrow\pi$ dans (i) est un isomorphisme.
\end{lem}



\begin{cor}\label{coro-presfinie-equv}
Les deux propri\'et\'es suivantes sont \'equivalentes:


(i) $\pi$ admet une pr\'esentation finie;

(ii) $\pi$ admet une pr\'esentation standard.
\end{cor}

\begin{proof}
L'implication (ii)$\Rightarrow$(i) \'etant triviale, il reste \`a d\'emontrer que (i)$\Rightarrow$(ii). Soit $\cInd_{KZ}^GW\twoheadrightarrow\pi$ une pr\'esentation finie avec $W\in\mathcal{W}(\pi)$. Une r\'ecurrence imm\'ediate sur la longueur de $W$, en utilisant le diagramme (\ref{equation-diag-presfini}), montre que $\pi$ est une extension successive d'une famille finie de $G$-repr\'esentations $\pi_i$ avec chaque $\pi_i$ un quotient de $\cInd_{KZ}^G\sigma_i$ et $\sigma_i$ parmi les facteurs de Jordan-H\"{o}lder de $W$. D'une part, comme $\pi$ est de pr\'esentation finie, le lemme \ref{lemma-typefini=finie-W} implique que tous les $\pi_i$ le sont aussi (en fait, ces deux conditions sont \'equivalentes). D'autre part, d'apr\`es \cite[proposition III.1.16]{Co}, $\pi$ est de pr\'esentation standard si et seulement si $\pi_i$ le sont. Ici, on remarque que, bien que cette proposition est \'enonc\'ee pour les $\pi$ de longueur finie, sa preuve reste valable pour $\pi$ de type fini.  On est donc ramen\'e au cas o\`u $\pi$ est un quotient de $\cInd_{KZ}^G\sigma$ pour $\sigma$ un poids. Si $\pi=\cInd_{KZ}^G\sigma$, (ii) est clair. Si $\pi$ est un quotient non trivial de $\cInd_{KZ}^G\sigma$, le th\'eor\`eme \ref{theorem-Pe} implique que $D_1(\sigma,\pi)$ est de dimension finie, et donc $D_0(\sigma,\pi)$ l'est aussi. En utilisant l'isomorphisme
(\ref{equation-H0(D(sigma))}): $H_0(D(\sigma,\pi))\cong \pi$, le lemme \ref{lemma-benjamin}(ii) montre que $\cInd_{KZ}^GD_0(\sigma,\pi)\twoheadrightarrow\pi$ fournit une pr\'esentation standard de $\pi$. Cela termine la d\'emontration.
\end{proof}

Comme cons\'equence, on obtient le r\'esultat suivant d\^u \`a Colmez \cite{Co}:
\begin{cor}
Toute repr\'esentation irr\'eductible non supersinguli\`ere de $G$ admet une pr\'esentation standard. Toute repr\'esentation de type fini de $\GL_2(\Q_p)$ admet une pr\'esentation standard.
\end{cor}
\begin{proof}
D'apr\`es le corollaire \ref{coro-presfinie-equv}, on est ramen\'e \`a montrer que les deux classes de $G$-repr\'esentations dans l'\'enonc\'e admettent une pr\'esentation finie, ce qui est bien connu pour la premi\`ere (voir \cite{BL2}). Quant \`a la deuxi\`eme classe, comme dans la d\'emonstration du corollaire \ref{coro-presfinie-equv}, il suffit de consid\'erer les quotients non triviaux de $\cInd_{\GL_2({\Z_p})\Z_p^{\times}}^{\GL_2(\Q_p)}\sigma$ avec $\sigma$ irr\'eductible, auquel cas l'\'enonc\'e r\'esulte du corollaire \ref{coro-Ollivier} et du th\'eor\`eme \ref{theorem-Pe}.
\end{proof}

Ce corollaire a aussi \'et\'e d\'emontr\'e par Breuil-Pa\v{s}k\={u}nas \cite{BP}, Vign\'eras \cite{Vi06-2} et Ollivier \cite{Ol}.

\subsection{L'espace des $I_1\cap U^+$-invariants de $I^+(\sigma,\pi)$}\label{subsection-U+-invariants}
Dans ce paragraphe, on fixe $\sigma$ un poids et $\pi$ un quotient non trivial de $\cInd_{KZ}^G\sigma$. Reppelons que $I_1\cap U^+=\smatr1{\cO}01$. Le but de ce paragraphe est de d\'emontrer le r\'esultat suivant:

\begin{prop}\label{prop-I^+-in-D_1-deuxcas}
Supposons $\pi$ irr\'eductible ou admissible. On a
\[I^+(\sigma,\pi)^{I_1\cap U^+}\subseteq D_1(\sigma,\pi).\]
\end{prop}
On va d\'emontrer cette proposition aux \S\ref{subsubsection-irred} et \S\ref{subsubsection-adm} selon deux cas. Donnons une cons\'equence imm\'ediate:

\begin{cor}\label{coro-finitude-->adm}
Supposons $\pi$ irr\'eductible ou admissible. Si $D_1(\sigma,\pi)$ est de dimension finie, alors
$I^+(\sigma,\pi)^{I_1\cap U^+}$ l'est aussi.
\end{cor}
\begin{proof}
C'est une cons\'equence directe de la proposition \ref{prop-I^+-in-D_1-deuxcas}.
\end{proof}

\subsubsection{Le cas supersingulier}\label{subsubsection-irred}
%
On montre la proposition \ref{prop-I^+-in-D_1-deuxcas} pour $\pi$  irr\'eductible supersinguli\`ere. En fait, on va d\'emontrer un peu plus (notons que l'on peut \'ecrire $I^+(\pi)$ et $D_1(\pi)$ au lieu de $I^+(\sigma,\pi)$ et $D_1(\sigma,\pi)$ puisque $\pi$ est irr\'eductible):
\begin{prop}\label{prop-generalised}
Supposons $\pi$ irr\'eductible supersinguli\`ere.
\begin{enumerate}
\item[(i)] Si $v\in I^+(\pi)$ est un vecteur fix\'e par $I_1\cap U^+$, alors $S^mv=0$ pour $m\gg0$.

\item[(ii)] On a l'inclusion
$I^+(\pi)^{I_1\cap U^+}\subseteq D_1(\pi)$.
\end{enumerate}
\end{prop}


Combin\'e avec la proposition \ref{prop-I1-invariant-in-D1}, on obtient de la proposition \ref{prop-generalised} les inclusions suivantes (pour $\pi$ supersinguli\`ere):
\[\pi^{I_1}\subseteq I^+(\pi)^{I_1\cap U^+}\subseteq\{v\in I^+(\pi),\ S^mv=0 \textrm{\ pour\ }m\gg0\}\subseteq D_1(\pi).\]

Commen\c{c}ons la d\'emonstration par un lemme.

\begin{lem}\label{lemma-Alperin}
Soient $H$ un pro-$p$-groupe et $M$ une repr\'esentation lisse de $H$ (sur $\bFp$). Soient $x\in M$ un vecteur non nul et $M_{x}=\langle H\cdot x\rangle$ la sous-repr\'esentation de $M$ engendr\'ee par $x$. 
\begin{enumerate}
\item[(i)] Pour tout $h\in H$, on a $(h-1)x\in \rad_H (M_x)$; en particulier,
\[\dim_{\bFp}M_{(h-1)x} <\dim_{\bFp}M_x.\]

\item[(ii)] Comme $\bFp$-espace vectoriel, $\rad_H (M_x)$ est engendré par les $(h-1)x$ pour $h\in H$.
\end{enumerate}
\end{lem}
\begin{proof}
(i) Comme $M$ est lisse, $x$ est fix\'e par un sous-groupe ouvert $H_1$ de $H$, on peut donc supposer que $H$ est un $p$-groupe fini en rempla\c{c}ant $H$ par $H/H_1$. D'apr\`es \cite[\S1.3, exer. 2]{Al}, $h-1$ appartient \`a $\rad(\bFp[H])$ le radical de $\bFp[H]$, et donc (\cite[\S1.1, prop. 4]{Al})
\[(h-1)x\in \rad(\bFp[H])\cdot M_x=\rad_H(M_{x}).\]
Le dernier \'enonc\'e vient du fait que $M_{(h-1)x}\subset\rad_H(M_x)\subsetneq M_x$.

(ii) Comme dans (i), on peut supposer que $H$ est un $p$-groupe fini, auquel cas l'\'enonc\'e est une cons\'equence de \cite[\S1.3, exer. 2]{Al}.
\end{proof}

\vspace{1mm}

\begin{proof}[D\'emonstration de la proposition \ref{prop-generalised}]
D'apr\`es le lemme \ref{lemma-c_m=0}, (ii) est une cons\'equence de (i) puisque $v\in I^+(\pi)$.
Soit $v\in I^+(\pi)^{I_1\cap U^+}$ un vecteur non nul. 
On va d\'emontrer (i) par r\'ecurrence sur la dimension de $\langle I_1\cdot v\rangle$ (sur $\bFp$) qui sera not\'ee $n(v)$.\vv

(a) \emph{Cas o\`u $n(v)=1$}, i.e.  $v$ est fix\'e par $I_1$. C'est juste le lemme \ref{lemma-S^m=0-pas}.
\vv

(b) \emph{Cas o\`u $n(v)\geq 2$}.
Par hypoth\`ese de r\'ecurrence, l'\'enonc\'e (i) est vrai pour tout $v'\in I^+(\pi)^{I_1\cap U^+}$ tel que $n(v')<n(v)$.
Le lemme \ref{lemma-Alperin}(i) appliqu\'e \`a $H=I_1$ montre que c'est le cas pour tout $(h-1)v$ avec $h\in \smatr{1+\p}00{1+\p}\subset I_1$,  donc il existe un entier $m_h\geq 1$ d\'ependant de $h$ tel que $S^{m_h}((h-1)v)=0$. Comme $\pi$ est lisse, la dimension de $\langle \smatr{1+\p}00{1+\p}\cdot v\rangle$ (sur $\bFp$) est finie, on trouve donc un entier $m\geq 1$ suffisamment grand tel que
\[S^m((h-1)v)=0\]
pour tout $h$ comme ci-dessus. 
En utilisant le lemme \ref{lemma-H-invariant}(ii) plusieurs fois, on obtient pour tout $h$:
\[(h-1)S^mv=S^m((h-1) v)=0,\]
c'est-\`a-dire, $S^mv$ est fix\'e par $\smatr{1+\p}00{1+\p}$.
Ainsi, on peut supposer  que $v$ est fix\'e par $\smatr{1+\p}00{1+\p}$ quitte \`a remplacer $v$ par $S^{m}v$.  

Comme $\pi$ est lisse, il existe un entier $k\geq 1$ tel que $v$ soit fix\'e par $\smatr10{\p^{k+1}}1$. D'apr\`es la formule suivante (o\`u $a\in \cO$):
\begin{equation}\label{equation-Sv-U-}\matr{1}{0}{\varpi^ka}{1}\summ_{\lambda\in\F_q}\matr{\varpi}{[\lambda]}01=\summ_{\lambda\in\F_q}\matr{\varpi}{\frac{[\lambda]}{1+\varpi^ka[\lambda]}}{0}{1}
\matr{\frac{1}{1+\varpi^ka[\lambda]}}{0}{\varpi^{k+1}a}{1+\varpi^ka[\lambda]}\end{equation}
on d\'eduit que $Sv$ est fix\'e par $\smatr10{\p^{k}}1$. De m\^eme, $S^kv$ est fix\'e par $\smatr10{\p}1$. Le lemme \ref{lemma-H-invariant} combin\'e avec la d\'ecomposition d'Iwahori (\ref{equation-Iwahori})
montre alors que $S^kv$ est fix\'e par $I_1$. On est donc ramen\'e  au cas (a) en rempla\c{c}ant $v$ par $S^kv$, et la proposition s'en d\'eduit.
\end{proof}

\subsubsection{Le cas admissible}\label{subsubsection-adm}

Comme toute repr\'esentation irr\'eductible non supersinguli\`ere de $G$ est admissible \cite{BL2}, la preuve suivante  compl\`ete la preuve de la proposition \ref{prop-I^+-in-D_1-deuxcas}.

\begin{proof}[D\'emonstration de \ref{prop-I^+-in-D_1-deuxcas} pour $\pi$ admissible]
(i) Posons $V^+=I^+(\sigma,\pi)^{I_1\cap U^+}$ et pour $n\geq 1$,
\[V^+_n:=\Bigl\{v\in V^+|\ v \textrm{\ est\ fix\'e  par\ } I_n\cap U^-=\matr10{\p^n}1\Bigr\}.\]
Comme $\pi$ admet un caract\`ere central, et comme $\smatr{1+\p^{n+1}}{0}{0}{1}$ (si $p\neq2$, on peut le remplacer
par $\smatr{1+\p^{n}}{0}{0}{1}$) est inclus dans le
groupe engendr\'e par $I_1\cap U^+$, $I_n\cap U^-$ et $I_1\cap Z$ dans $G$, on voit que
tout $x\in V^+_n$ est fix\'e aussi par
$\smatr{1+\p^{n+1}}{0}{0}{1}$. En cons\'equence, $ V^+_n$ est fix\'e par $I_{n+1}$ qui est un sous-groupe ouvert de $G$. On en d\'eduit que chaque $V^+_n$ est de
dimension finie puisque $\pi$ est admissible.


Comme $\pi$ est lisse, il existe un entier $n\geq 1$ tel que $v\in V^+_{n-1}$.
D'apr\`es la proposition \ref{prop-U+-invariant}(i), $Sv\in
V^+$; de plus, pour $\mu\in\F_q$, on a\vv
\begin{itemize}
\item[--] par ce qui pr\'ec\`ede, $v$ est fix\'e par $\smatr{1+\p^{n}}{0}{0}{1}$, et donc $Sv$ l'est aussi: \[\matr{1+\varpi^{n}[\mu]}{0}{0}{1}Sv=\summ_{\lambda\in\F_q}\matr{\varpi}{[\lambda]}01
    \matr{1+\varpi^n[\mu]}{\varpi^{n-1}[\mu\lambda]}01v=Sv;\]
\item[--] $Sv$ est fix\'e par $I_n\cap U^-$ par le calcul (\ref{equation-Sv-U-}).\vspace{1mm}
\end{itemize}
Autrement dit, $Sv$ appartient \`a  $V^+_n$ et est fix\'e par $\smatr{1+\p^n}001$. De m\^eme, $S^{m}v$ l'est aussi pour tout $m\geq 1$. La finitude de la dimension de $V_n^+$ implique qu'il existe une famille
$(c_m)_{m\geq 0}$ d'\'el\'ements de $\bFp$ tels que
\[c_0v+\summ_{m\geq
1}c_mS^{m}v=0,\]
et le lemme \ref{lemma-c_m=0} permet de conclure puisque $v\in I^+(\sigma,\pi)$.
\end{proof}

\section{Le cas o\`u $F$ est de caract\'eristique positive}
Dans ce chapitre, on suppose que $F$ est de caract\'eristique $p$. Cela fait que $F\cong\F_q((\varpi))$ et $[\lambda]=\lambda$ pour $\lambda\in\F_q$. On d\'etermine au \S\ref{subsection-D(pi)=K(pi)} le diagramme canonique associ\'e \`a une repr\'esentation supersinguli\`ere de $G$ et on en donne au \S\ref{subsection-char=p-conse} quelques cons\'equences.

\subsection{$D(\pi)=\cK(\pi)$}\label{subsection-D(pi)=K(pi)}
Le but de ce paragraphe est de d\'emontrer le th\'eor\`eme suivant dont
la d\'emonstration va demander un peu de pr\'eparation.
\begin{thm}\label{theorem-main}
Soit $\pi$ une repr\'esentation supersinguli\`ere de $G$.
 
(i) Pour tout $x\in I^+(\pi)$, on a $S^mx=0$ pour $m\gg0$.

(ii) On a \[D_1(\pi)=D_0(\pi)=I^+(\pi)=\pi.\]
Autrement dit, $D(\pi)=\cK(\pi):=(\pi|_{KZ},\pi|_{N},\mathrm{can})$.
\end{thm}
\vv

Soit $M$ une repr\'esentation lisse de $I$. Posons $W=\Ind_I^K\Pi(M)$ la $K$-repr\'esentation induite. Rappelons (\S\ref{subsection-Ind-I-K}, (\ref{equation-W=M+W+})) que $W^+$ d\'esigne le sous-espace vectoriel de $W$ engendré par les vecteurs
\[\biggl\{\matr{\varpi}{\lambda}01v,\ \ v\in M,\ \lambda\in\F_q\biggr\}.\]
Notons que $W^+$ est stable par $I$, \emph{a fortiori} par $I_1\cap U^+=\smatr1{\cO}01$. Dans le lemme suivant, on note pour plus d'aisance $Sx$ au lieu de $F_{0,x}$ (voir la proposition \ref{prop-U+-invariant}).
\begin{lem}\label{lemma-W+}
Soient $x\in M$ un vecteur et $\{v_i\}_{1\leq i\leq N}$ une $\bFp$-base de $\langle I_1\cap U^+\cdot x\rangle$. Alors $\{Sv_i\}_{1\leq i\leq N}$ forme une $\bFp$-base pour $\langle I_1\cap U^+\cdot Sx\rangle\subset W^+$.
\end{lem}
\begin{proof}
Remarquons que $M$ \'etant lisse, $\langle I_1\cap U^+\cdot x\rangle$ est bien de dimension finie.

Si $n\geq 0$ et si $a=\sum_{i=0}^na_i\varpi^i\in\cO$ avec $a_i\in\F_q\hookrightarrow \cO$, on pose \[\tilde{a}=\summ_{i=1}^na_i\varpi^{i-1}\in\cO\] de telle sorte que (o\`u $\lambda\in\F_q$):
\[\matr{1}a01\matr{\varpi}{\lambda}01=\matr{\varpi}{\lambda+a_0}01\matr1{\tilde{a}}01.\]
On en d\'eduit:
\[\begin{array}{rll}
\matr1{a}01Sx&=&\summ_{\lambda\in\F_q}\matr1{a}01\matr{\varpi}{\lambda}01x\\
&{=}&\summ_{\lambda\in\F_q}\matr{\varpi}{\lambda+a_0}01\matr{1}{\tilde{a}}01x\\
&{=}&\summ_{\lambda\in\F_q}\matr{\varpi}{\lambda}01\matr{1}{\tilde{a}}01x
\end{array}\]
et donc $\langle I_1\cap U^+\cdot Sx\rangle$ est contenu dans l'espace vectoriel engendré par $\{ Sv_i\}_{1\leq i\leq N}$. R\'eciproquement, fixons un indice $i\in\{1,\dots, N\}$ et soit \[Q=\summ_k \alpha_k\matr1{b_k}01\in\bFp[I_1\cap U^+], \ \ \alpha_k\in\bFp\] un \'el\'ement tel que $Q\cdot x=v_i$. Soit $b_k'\in\cO$ un \'el\'ement tel que $b_k=\tilde{b_k'}$ (c'est toujours possible). Le m\^eme calcul que ci-dessus montre que $Q'\cdot Sx=Sv_i$, o\`u $Q'\in\bFp[I_1\cap U^+]$ est d\'efini par
$Q'=\sum_{k}\alpha_k\smatr1{b_k'}01$.
Cela permet de conclure.
\end{proof}
\begin{cor}\label{coro-W+}
Pout tout vecteur $x\in\pi$, on a
\[\dim_{\bFp} \langle I_1\cap U^+\cdot Sx\rangle\leq \dim_{\bFp}\langle I_1\cap U^+\cdot x\rangle,\]
avec in\'egalit\'e stricte s'il existe un vecteur non nul $v\in \langle I_1\cap U^+\cdot x\rangle$ tel que  $Sv=0$.
\end{cor}
\begin{proof}
Cons\'equence triviale du lemme \ref{lemma-W+}.
\end{proof}
\vspace{1mm}

\begin{proof}[D\'emonstration du th\'eor\`eme \ref{theorem-main}]
(i) Soit $x\in I^+(\pi)$ un vecteur non nul.
On va d\'emontrer (i) par r\'ecurrence sur la dimension de $M_x:=\langle I_1\cap U^+\cdot x\rangle$ (sur $\bFp$) qui sera not\'ee $m(x)$. Notons que $m(x)<+\infty$ puisque $\pi$ est lisse.

Comme $I_1\cap U^+$ est un pro-$p$-groupe, l'espace $M_x^{I_1\cap U^+}$ est non nul. Soit $v\in M_x$ un vecteur non nul fix\'e par $I_1\cap U^+$. Alors la proposition \ref{prop-generalised}(i) implique $S^nv=0$ pour $n\gg0$, et donc (par le corollaire \ref{coro-W+})
\[m(S^nx)< m(x).\]
Par hypoth\`ese de r\'ecurrence, on voit que $S^{n+m}x=S^m(S^nx)=0$ pour un entier $m\geq 1$ suffisamment grand, d'o\`u l'\'enonc\'e pour $x$.

(ii) D'apr\`es (i) et le lemme \ref{lemma-c_m=0}, on a $I^+(\pi)\subseteq D_1(\pi)$ puis $\pi=D_1(\pi)$, car $\Pi\cdot D_1(\pi)=D_1(\pi)$ et $\pi=I^+(\pi)+\Pi\cdot I^+(\pi)$. Les autres \'enonc\'es sont imm\'ediats.
\end{proof}
\begin{rem}
Supposons ici que $F$ soit une extension finie de $\Q_p$, avec $e\geq 2$ l'indice de ramification. Soit $\pi$ une repr\'esentation supersinguli\`ere de $G$. Alors pour tout $x\in \pi^{I_{e-1}}$, on a $S^mx=0$ pour $m\gg0$.
 En particulier, $\pi^{I_{e-1}}\subseteq D_1(\pi)$.
\end{rem}
\begin{proof}
D'apr\`es la d\'emonstration du th\'eor\`eme \ref{theorem-main}, on voit qu'il suffit de montrer l'analogue du corollaire \ref{coro-W+} sous l'hypoth\`ese suppl\'ementaire que $x\in \pi$ est fix\'e par $I_{e-1}$. Or, c'est une cons\'equence du fait suivant (par un calcul analogue \`a celui du lemme \ref{lemma-W+}):
\emph{si $a=\sum_{i=0}^n[a_i]\varpi^i\in\cO$ avec $a_i\in\F_q$ et si $\lambda\in\F_q$, on a (voir par exemple \cite{Se3})
\[a+[\lambda]=[a_0+\lambda]+\summ_{i=1}^n[a_i]\varpi^i+X(a_0,\lambda)\varpi^{e}\]
avec $X(a_0,\lambda)\in\cO$ un \'el\'ement d\'ependant de $a_0$ et de $\lambda$.}
\end{proof}
\vv
\begin{cor}\label{coro-non-presfinie}
Soit $\pi$ une repr\'esentation supersinguli\`ere de $G$. Alors $\pi$ n'est pas de pr\'esentation finie.
\end{cor}
\begin{proof}
C'est une cons\'equence directe des th\'eor\`emes \ref{theorem-Pe} et \ref{theorem-main}(ii) et de la proposition \ref{lemma-equi-presfinie-I} en remarquant qu'une repr\'esentation supersinguli\`ere est toujours de dimension infinie sur $\bFp$.
\end{proof}

Par les r\'esultats de Barthel-Livn\'e \cite{BL2} et Breuil \cite{Br1}, les repr\'esentations non supersinguli\`eres de $G$ et les repr\'esentations supersinguli\`eres de $\GL_2(\Q_p)$ sont toutes de pr\'esentation finie. Le corollaire \ref{coro-non-presfinie} donne donc la premi\`ere classe de repr\'esentations lisses irr\'eductibles de $G$ qui ne sont pas de pr\'esentation finie.

\subsection{Cons\'equences}\label{subsection-char=p-conse}

On donne deux applications des r\'esultats au \S\ref{subsection-D(pi)=K(pi)}. Au \S\ref{subsubsection-application-1}, on g\'en\'eralise des r\'esultats de \cite{Pa3} sous l'hypoth\`ese suppl\'ementaire que $F$ est de caract\'eristique $p$. Au  \S\ref{subsubsection-coro-cas(ii)}, on compl\`ete la d\'emonstration du corollaire \ref{corollaire-deuxcas=}.

\subsubsection{Restriction \`a $P^+$}\label{subsubsection-application-1}

 Rappelons que $P$ d\'esigne le sous-groupe de Borel de $G$ et $P^+$ le sous-mono\"{i}de $\smatr{\cO-\{0\}}{\cO}01$ de $G$.  Commen\c{c}ons par rappeler un r\'esultat de \cite{Pa3}:
\begin{lem}\label{lemma-P+-irred}
Soient $\pi$ une repr\'esentation lisse de $G$ (avec caract\`ere central) et $x\in \pi$ un vecteur non nul. Alors il existe un vecteur non nul $v\in\langle P^+\cdot x\rangle\cap \pi^{I_1}$ tel que $\langle K\cdot v\rangle$ soit une $K$-repr\'esentation irr\'eductible.
\end{lem}
\begin{proof}
Faire la m\^eme preuve que celle de  \cite[proposition 4.2]{Pa3}, en remarquant que $I_1\cap P=I_1\cap P^+Z$
et que $\pi$ admet un caract\`ere central.
\end{proof}

\begin{thm}\label{theorem-irr}
Soit $\pi$ une repr\'esentation irr\'eductible supersinguli\`ere de $G$. Alors $\pi|_{P^+}$ est encore irr\'eductible.
\end{thm}
On dit qu'une repr\'esentation $V$ d'un mono\"{i}de $H$ est \emph{irr\'eductible} si $\langle H\cdot x\rangle=V$ pour tout vecteur $x\in V$ non nul.

\begin{proof}
Comparer avec \cite[th\'eor\`eme 4.3]{Pa3}. Soit $x\in \pi$ un vecteur non nul. Il faut d\'emontrer:
  \[\pi=\langle P^+\cdot x\rangle.\]
En utilisant le lemme \ref{lemma-P+-irred}, on trouve un vecteur non nul
$v\in \langle P^+\cdot x\rangle\cap \pi^{I_1}$. De plus, le lemme \ref{lemma-S^m=0-pas} permet de supposer  que $Sv=0$, ce qui fait que $\smatr0110 v\in \langle P^+\cdot v\rangle$  d'apr\`es \cite[lemme 3.4]{Pa3}.  Comme $\smatr{[\lambda]}110=\smatr{1}{[\lambda]}01\smatr0110$,  la d\'ecomposition (\ref{equation-decom-K/I}) implique que $\langle K\cdot v\rangle\subset\langle P^+\cdot v\rangle$.   Prenons $\sigma$ une sous-$K$-repr\'esentation irr\'eductible de $\langle K\cdot v\rangle$. Alors par le th\'eor\`eme \ref{theorem-main}(ii),  on a 
\[\pi=I^+(\pi)=\langle P^+\cdot \sigma\rangle\subseteq\langle P^+\cdot v\rangle\subseteq\langle P^+\cdot x\rangle,\]
d'o\`u le r\'esultat.\end{proof}

\vv


Le reste de ce paragraphe est consacr\'e \`a montrer le th\'eor\`eme suivant (comparer avec \cite[th\'eor\`eme 4.4]{Pa3}).
\begin{thm}\label{theorem-morph}
Soient $\pi$, $\pi'$ deux repr\'esentations lisses de $G$ avec $\pi$ irr\'eductible supersinguli\`ere. Alors
 \[\Hom_{P^+Z}(\pi,\pi')\cong\Hom_G(\pi,\pi').\]
\end{thm}
Rappelons d'abord un r\'esultat de \cite{Pa3}:

\begin{lem}\label{lemma-Pa->pi'}
Conservons les notations du th\'eor\`eme \ref{theorem-morph}. Si $\phi:\pi\ra\pi'$ est un $P^+Z$-morphisme non nul, alors il existe un vecteur non nul $v\in\pi^{I_1}$ tel que $\phi(v)\in\pi'^{I_1}$.
\end{lem}
\begin{proof}
Voir la d\'emonstration du \cite[th\'eor\`eme 4.4]{Pa3}.
\end{proof}

\textbf{Notation}: On pose $R\in \bFp[P^+I]$ l'\'el\'ement d\'efini par:
\begin{equation}
\nonumber
R:=\summ_{\lambda\in\F_q^{\times}}
\matr{\varpi}{\lambda^{-1}}01\matr{-\lambda^{-1}}0{\varpi}{\lambda}.\end{equation}
de telle sorte que l'on ait l'\'egalit\'e suivante dans $\bFp[G]$ (par un calcul direct):
\begin{equation}\label{equation-R}s\cdot S=\Pi+R\end{equation}
en rappelant que $s=\smatr0110$ et $\Pi=\smatr01{\varpi}0$.

\vv

\begin{proof}[D\'emonstration du th\'eor\`eme \ref{theorem-morph}]
Soit $\phi:\pi\ra \pi'$ un morphisme $P^+$-\'equivariant non nul. On va d\'emontrer que $\phi$ est $G$-\'equivariant (ce qui suffit pour conclure). En r\'e\'ecrivant les d\'ecompositions (\ref{equation-vigneras})  et (\ref{equation-decom-K/I}) sous la forme:
\[G=P^+ZK\bigcup sP^+ZK\]
et
\[K=I\coprod\biggl(\coprod_{\lambda\in\F_q}\matr{1}{\lambda}01sI\biggr),\]
on est ramen\'e \`a v\'erifier que:
\begin{itemize}
\item[(a)] $\phi(g\cdot x)=g\cdot \phi(x)$ pour tout $g\in I$ et tout $x\in\pi$;

\item[(b)] $\phi(s\cdot x)=s\cdot\phi(x)$ pour tout $x\in\pi$.
\end{itemize}
\vv

Soit $v\in \pi^{I_1}$ un vecteur non nul tel que $\phi(v)\in\pi'^{I_1}$, dont l'existence est assur\'ee par le lemme \ref{lemma-Pa->pi'}. D'apr\`es le th\'eor\`eme \ref{theorem-irr}, $\pi|_{P^+}$ est irr\'eductible, donc $\phi(v)\neq 0$ et qu'il existe $Q\in\bFp[P^+]$ tel que $x=Q\cdot v$.
Si $g\in I$, on \'ecrit
\begin{equation}\label{equation-gQ}
gQ=\summ_{i}c_iQ_i'g_i'\end{equation}
avce $c_i\in \bFp^{\times}$, $Q_i'\in P^+Z$ et $g_i'\in I_1$. Ceci est toujours possible gr\^ace au lemme \ref{lemma-IP=PI}.
On a alors:
\[\begin{array}{rlll}
\phi(g\cdot x)=\phi(gQ\cdot v)&=&\phi(\summ_ic_i Q_i'g_i'\cdot v)& \mathrm{par\ }(\ref{equation-gQ})\\
&=&\summ_ic_iQ_i'\cdot\phi(v)& \mathrm{car\ }v\in\pi^{I_1} \mathrm{\ et\ }\phi \mathrm{\ est\ } P^+Z\textrm{-\'equivariant}\\
&=&\summ_ic_iQ_i'g_i'\cdot \phi(v)&\mathrm{car\ }\phi(v)\in\pi'^{I_1}\\
&=&gQ\cdot\phi(v)&\mathrm{par\ }(\ref{equation-gQ})\\
&=&g\cdot\phi(x)&\mathrm{par\ }P^+\textrm{-\'equivariance\ de\ } \phi
\end{array}\]
d'o\`u l'\'enonc\'e  (a).

D'apr\`es le th\'eor\`eme \ref{theorem-main}(i), pour tout $x\in \pi$, il existe un entier $m\gg 0$ tel que $S^mx=0$. On d\'efinit $d(x)$ comme le plus petit entier v\'erifiant cette propri\'et\'e. On va montrer (b) par r\'ecurrence sur $d(x)$.

Si $d(x)=1$, alors $Sx=0$. D'une part, on a de l'\'equation (\ref{equation-R}): $\phi(\Pi\cdot x)+\phi(R\cdot x)=0$;
comme $R\in\bFp[P^+I]$, on en d\'eduit de (a) que:
\begin{equation}\nonumber\label{equation-sum=0a}\phi(\Pi\cdot x)+R\cdot\phi(x)=0.\end{equation}
D'autre part, comme $\phi$ est $P^+$-\'equivariant, $S(\phi(x))=\phi(Sx)=0$, donc par (\ref{equation-R})
\begin{equation}\nonumber\label{equation-sum=0b}
\Pi\cdot \phi(x)+R\cdot\phi(x)=0.\end{equation}
On en d\'eduit donc 
l'\'egalit\'e \[\phi(\Pi\cdot x)=\Pi\cdot\phi(x),\] et en lui appliquant la matrice $\smatr{\varpi}001$ et en utilisant la $P^+Z$-\'equivariance de $\phi$, on voit que
\[\chi(\varpi)\phi(s\cdot x)=\chi(\varpi)(s\cdot\phi(x)),\]
où $\chi$ d\'esigne le caract\`ere central commun de $\pi$ et de $\pi'$.
L'\'enonc\'e (b) s'en d\'eduit dans ce cas particulier.

Par hypoth\`ese de r\'ecurrence, on a $\phi(s\cdot y)=s\cdot \phi(y)$ pour tout vecteur $y\in\pi$ v\'erifiant $d(y)\leq d$. Soit maintenant $x\in\pi$ un vecteur v\'erifiant $d(x)=d+1$. Comme $d(Sx)=d$, on a
\[\phi(s\cdot Sx)=s\cdot\phi(Sx)=s\cdot S(\phi(x)),\]
ce qui \'equivaut \`a dire que (par (\ref{equation-R})):
\[\phi(\Pi\cdot x)+\phi(R\cdot x)=\Pi\cdot\phi(x)+R\cdot\phi(x).\]
 Le m\^eme raisonnement que ci-dessus donne $\phi(s\cdot x)=s\cdot\phi(x)$.
Cela montre (b) et ach\`eve la d\'emonstration.
\end{proof}


Comme on l'a fait remarquer dans l'introduction, le th\'eor\`eme \ref{theorem-morph} reste vrai pour tout corps $F$. Voici la d\'emonstration rapide de Pa\v{s}k\={u}nas:
\begin{proof}[D\'emonstration du th\'eor\`eme \ref{theorem-morph} pour tout $F$] Posons \[t=s\Pi=\matr{\varpi}001\in G.\]
Notons comme d'habitude $\bFp[t]$ l'anneau des polyn\^omes en $t$ et $\bFp[t,t^{-1}]$ la localisation de  $\bFp[t]$ de telle sorte qu'on ait l'isomorphisme naturel:
\[\bFp[P^+]\otimes_{\bFp[t]}\bFp[t,t^{-1}]\cong \bFp[P]\]
puis l'isomorphisme
\[I^+(\pi)\otimes_{\bFp[t]}\bFp[t,t^{-1}]\simto \pi\]
o\`u la surjectivit\'e r\'esulte de l'irr\'eductibilit\'e de $\pi$ en tant que $P$-repr\'esentation (\cite[th\'eor\`eme 4.3]{Pa3}).
Ce dernier induit un isomorphisme: 
\begin{equation}\label{equation-Pas}
\Hom_{P}(\pi,\pi')\simto \Hom_{P^+Z}(I^+(\pi),\pi'),\ \ f\mapsto f|_{I^+(\pi)}.\end{equation}

D'autre part, comme $t$ est nilpotent sur la $P^+Z$-repr\'esentation $\pi/I^+(\pi)$ (pour le voir, on utilise le fait que $\pi|_{P}$ est irr\'eductible) et comme $t$ est inversible sur la $P$-repr\'esentation $\pi'$, on a forc\'ement
$\Hom_{P^+Z}(\pi/I^+(\pi),\pi')=0$
et donc on obtient une injection:
\begin{equation}\label{equation-Pas-inj}\Hom_{P^+Z}(\pi,\pi')\hookrightarrow \Hom_{P^+Z}(I^+(\pi),\pi'),\ \ f\mapsto f|_{I^+(\pi)}.\end{equation}
De (\ref{equation-Pas}) et de (\ref{equation-Pas-inj}), on d\'eduit que le morphisme naturel:
\[\Hom_{P}(\pi,\pi')\ra\Hom_{P^+Z}(\pi,\pi')\]
\`a travers lequel (\ref{equation-Pas}) se factorise, est en fait un isomorphisme, et l'isomorphisme cherch\'e s'obtient en le composant avec celui de \cite[th\'eor\`eme 4.4]{Pa3}.
\end{proof}

\subsubsection{D\'emonstration du corollaire \ref{corollaire-deuxcas=}}
\label{subsubsection-coro-cas(ii)}

\begin{proof}[D\'emonstration de \ref{corollaire-deuxcas=} cas (ii)]
Soient $\pi_1$ une sous-repr\'esentation irr\'eductible de $G$ et $\pi_2$ le quotient de $\pi$ par $\pi_1$. Comme  dans la preuve du cas (i), on  a $\dim_{\bFp}(D_1(\sigma,\pi))\geq \dim_{\bFp} D_1(\pi_1)+\dim_{\bFp}D_1(\sigma,\pi_2)$, et pour conclure il suffit de montrer que $\pi_2$ est admissible, ou encore $\pi_2^{I_1}$ est de dimension finie sur $\bFp$.
Comme $D_1(\pi_1)$ est de dimension finie, le th\'eor\`eme \ref{theorem-main}(ii) assure que $\pi_1$ est non supersinguli\`ere.


On suppose par l'absurde que $\pi_2^{I_1}$ est de dimension infinie. 
Comme $D_1(\sigma,\pi_2)$ est de dimension finie,  on peut trouver un vecteur $\overline{x}$ de $\pi_2$ tel que $\overline{x}\in\pi_2^{I_1}$ mais $\overline{x}\notin D_1(\sigma,\pi_2)$. \'Ecrivons  $\overline{x}=\overline{x}^++\overline{x}^-$ avec $\overline{x}^+\in I^+(\sigma,\pi_2)$ et $\overline{x}^-\in I^-(\sigma,\pi_2)$. Quitte \`a remplacer $\overline{x}$ par $\Pi(\overline{x})$, on peut supposer que $\ell_{\sigma}(\overline{x}^+)\geq \ell_{\sigma}(\overline{x}^-)$. Puisque $\overline{x}\notin D_1(\sigma,\pi_2)$, cela fait que $\ell_{\sigma}(\overline{x}^+)\geq 1$, donc le lemme  \ref{lemma-niveau} entra\^ine que $S^n\overline{x}\neq 0$ pour tout $n\geq 0$ et que $S^n\overline{x}\in I^+(\sigma,\pi_2)$ si $n$ est suffisamment grand. Comme $S^n\overline{x}$ est   fix\'e par $I_1$  et n'appartient pas \`a $D_1(\sigma,\pi_2)$ pour tout $n\geq 0$, on peut supposer que
 $\overline{x}\in I^+(\sigma,\pi_2)^{I_1}$ mais $\overline{x}\notin D_1(\sigma,\pi_2)$ en rempla\c{c}ant eventuellement $\overline{x}$ par $S^n\overline{x}$ avec $n\gg0$. 

Soit $x\in \pi$ un rel\`evement de $\overline{x}$ et consid\'erons $\langle I_1\cap U^+\cdot x\rangle$ la sous-repr\'esentation de $\pi$ qu'il engendre.  Puisque $\overline{x}$ est fix\'e par $I_1\cap U^+$, le radical de $\langle I_1\cap U^+\cdot x\rangle$ est contenu dans $\pi_1$ par le lemme \ref{lemma-Alperin}.  On va distinguer deux cas.

Premier cas: $\pi_1$ est un caract\`ere isomorphe \`a $\chi\circ\det$ avec $\chi$ un caract\`ere lisse de $F^{\times}$.  Soit $v_0\in\pi_1$ un vecteur non nul de sorte que $\pi_1=\bFp v_0$.  On a alors dans $\pi_1$:
\[Sv_0=\summ_{\lambda\in\F_q}\matr{\varpi}{[\lambda]}01v_0=q\chi(\varpi)v_0=0, \]
ce qui fait que $Sx$ est fix\'e par $I_1\cap U^+$ d'apr\`es le lemme \ref{lemma-W+}. Ensuite, l'admissibilit\'e de $\pi$ implique que $Sx\in D_1(\sigma,\pi)$ par la proposition \ref{prop-I^+-in-D_1-deuxcas}, puis $x\in D_1(\sigma,\pi)$ par le lemme \ref{lemma-niveau}(ii). Enfin, la  proposition \ref{prop-exact-complexe} montre que $\overline{x}\in D_1(\sigma,\pi_2)$, ce qui donne une contradiction avec le choix de $\overline{x}$.

Deuxi\`eme cas: $\pi_1$ est une s\'erie sp\'eciale ou une s\'erie principale. D'apr\`es \cite[th\'eor\`eme 33]{BL2}, $\pi_1$ admet toujours une sous-$KZ$-repr\'esentation irr\'eductible $\sigma$ de dimension $\geq 2$. Soient $v_0\in\sigma$ un vecteur non nul fix\'e par $I_1$ et $\langle P^+\cdot v_0\rangle$ le sous-espace vectoriel de $\pi_1$ engendr\'e par $v_0$.  Par les lemmes \ref{lemma-W+} et \ref{lemma-Sm-in-P+v0} (ci-apr\`es) et quitte \`a remplacer $ \overline{x}$  par $S^m\overline{x}$  avec $m\gg 0$, on peut supposer que le radical de $\langle I_1\cap U^+\cdot x\rangle$ est contenu dans $\langle P^+\cdot v_0\rangle$.  Ainsi, on obtient une suite exacte de $I_1\cap U^+$-repr\'esentations:
\[0\ra \langle P^+\cdot v_0\rangle \ra \langle P^+\cdot v_0\rangle\oplus\bFp x\ra \bFp\overline{x}\ra0.\]
Or, d'apr\`es le lemme \ref{lemma-P+v=inj}, $\langle P^+\cdot v_0\rangle$ est un objet injectif de la cat\'egorie $\Rep_{I_1\cap U^+}$, la suite s'y scinde et il existe donc un vecteur $v'\in I^+(\pi_1)$ tel que $v-v'$ soit fix\'e par $I_1\cap U^+$. On conclut alors comme dans le premier cas par l'admissibilit\'e de $\pi$.
\end{proof}


\begin{lem}\label{lemma-Sm-in-P+v0}
Conservons les notations de la d\'emonstration pr\'ec\'edente du deuxi\`eme cas. 
 Pour tout $x\in\pi_1$, il existe un entier $m\gg0$ tel que $S^mx\in\langle P^+\cdot v_0\rangle$.
\end{lem}
\begin{proof}
 Tout d'abord, l'espace $I^+(\pi_1)$ \'etant stable par $S$, le lemme \ref{lemma-niveau}(iii) implique que $S^mx\in I^+(\pi_1)$ pour $m\gg0$ et on peut donc supposer  $x\in I^+(\pi_1)$. Or, on a $I^+(\pi_1)=\langle P^+\cdot\Pi(v_0)\rangle$  d'apr\`es le lemme \ref{lemma-I+(sigma)=P+Pi(v)}, ce qui permet de conclure si $\pi$ est une s\'erie sp\'eciale parce que $\pi_1^{I_1}$ est de dimension 1 dans ce cas et donc  $\Pi(v_0)\in\bFp v_0$ (voir la preuve du th\'eor\`eme \ref{theorem-non-super}).

Supposons  que $\pi_1$ est une s\'erie principale. Par ce qui pr\'ec\`ede, on peut supposer que $x=g\Pi(v_0)$ avec $g\in P^+$. Supposons d'abord $g=1$ et montrons que $m=1$ convient. En effet, on a (o\`u $\chi_{\pi_1}$ d\'esigne le caract\`ere central de $\pi$) : 
\[S\Pi(v_0)=\chi_{\pi_1}(\varpi)\sum_{\lambda\in\F_q}\matr{[\lambda]}110v_0\in\sigma,\]
et que $S\Pi(v_0)$ est fix\'e par $I_1$ puisque $\Pi(v_0)$ l'est. Or,  $\sigma^{I_1}$ est de dimension 1 sur $\bFp$ dont $v_0$ est un vecteur de base, 
d'o\`u  $S\Pi(v_0)\in\bFp v_0$.  

Supposons $g\neq 1$ et  \'ecrivons $g=\smatr{\varpi^na}{b}01$ avec $a\in\cO^{\times}$, $b\in\cO$ et $n\geq 0$. On raisonne par r\'ecurrence sur $n$.
Si $n=0$, alors  $g\in I$ puis $g\Pi(v_0)\in\bFp\Pi(v_0)$, et le r\'esultat se d\'eduit du cas o\`u $g=1$.
Si $n\geq 1$,  on a
\[g\Pi(v_0)=\matr{\varpi^na}b01\matr01{\varpi}0v_0=\matr{\varpi}00{\varpi}\matr{\varpi^{n-1}a}{b}01\matr0110v_0.\]
Par ailleurs, comme la $K$-repr\'esentation $\pi_1^{K_1}$  est isomorphe \`a une s\'erie principale de $K/K_1\cong\GL_2(\F_q)$ contenant $\sigma$ dans son socle,  elle est engendr\'ee par $\Pi(v_0)$ comme $K$-repr\'esentation (voir \cite[\S2]{BP}). Par cons\'equent, on obtient de la d\'ecomposition (\ref{equation-decom-K/I})
\[\matr0110v_0\in\sigma\subseteq\langle K\cdot \Pi(v_0)\rangle\subseteq \bFp\Pi(v_0)+ \langle P^+\cdot v_0\rangle.\]
Le r\'esultat s'en d\'eduit facilement par hypoth\`ese de r\'ecurrence.
\end{proof}

\hspace{5cm} \hrulefill\hspace{5.5cm}\vv

D\'epartement de Math\'ematiques, B\^atiment 425,  Universit\'e de
Paris-Sud, 91405 Orsay Cedex, France

E-mail: yongquan.hu@math.u-psud.fr

\end{document}

%% file: def.tex
\newcommand{\p}{\mathfrak{p}}
\newcommand{\cO}{\mathcal{O}}
\newcommand{\cK}{\mathcal{K}}
\newcommand{\cH}{\mathcal{H}}

\newcommand{\F}{\mathbb F}
\newcommand{\N}{\mathbb N}
\newcommand{\Q}{\mathbb Q}

\newcommand{\Z}{\mathbb Z}

\newcommand{\rInj}{\mathrm{Inj}}
\newcommand{\rsoc}{\mathrm{soc}}

\newcommand{\ra}{\rightarrow}

\newcommand{\im}{\mathrm{Im}}
\newcommand{\End}{\mathrm{End}}
\newcommand{\Hom}{\mathrm{Hom}}

\newcommand{\GL}{\mathrm{GL}}

\newcommand{\Gal}{\mathrm{Gal}}
\newcommand{\id}{\mathrm{Id}}

\newcommand{\bFp}{\overline{\F}_p}
\newcommand{\bQp}{\overline{\Q}_p}
\newcommand{\rad}{\mathrm{rad}}

\newcommand{\Rep}{\underline{\mathrm{Rep}}}
\newcommand{\Ind}{\mathrm {Ind}}

\providecommand{\cInd}{\mathrm{c}\textrm{-}\mathrm{Ind}}

\newcommand{\xto}[1][]{\xrightarrow{#1}}
\newcommand{\simto}{
\xto[\sim]} 
\newcommand{\summ}{\displaystyle\sum\limits}

\newcommand{\matr}[4]{\begin{pmatrix}{#1}&{#2}\\ {#3}&{#4}\end{pmatrix}}
\newcommand{\smatr}[4]{\bigl(\begin{smallmatrix} {#1}& {#2}\\ {#3}&{#4}\end{smallmatrix}\bigl)}
\newcommand{\DIAG}{\mathcal{DIAG}}

\newcommand{\vv}{\vspace{2mm}}

